\documentclass{article}

\usepackage[
left=1.25in, right=1.25in]{geometry}
\usepackage[all]{xy}
\usepackage{graphicx}
\usepackage{indentfirst}
\usepackage{bm}
\usepackage{amsthm}
\usepackage{mathrsfs}
\usepackage{latexsym}
\usepackage{amsmath}
\usepackage{amssymb}

\newcommand{\gra}[1]{\raisebox{-.4cm}{\includegraphics[height=1cm]{BF#1.eps}}}
\newcommand{\graa}[1]{\raisebox{-.6cm}{\includegraphics[height=1.5cm]{BF#1.eps}}}
\newcommand{\grb}[1]{\raisebox{-.8cm}{\includegraphics[height=2cm]{BF#1.eps}}}
\newcommand{\grc}[1]{\raisebox{-1.3cm}{\includegraphics[height=3cm]{BF#1.eps}}}

\begin{document}

\theoremstyle{plain}
\newtheorem{theorem}{Theorem~}[section]
\newtheorem*{main}{Main Theorem~}
\newtheorem{lemma}[theorem]{Lemma~}
\newtheorem{proposition}[theorem]{Proposition~}
\newtheorem{corollary}[theorem]{Corollary~}
\newtheorem{definition}{Definition~}[section]
\newtheorem{notation}{Notation~}[section]
\newtheorem{example}{Example~}[section]
\newtheorem*{remark}{Remark~}
\newtheorem*{question}{Question}
\newtheorem*{claim}{Claim}
\newtheorem*{ac}{Acknowledgement}
\newtheorem*{conjecture}{Conjecture~}
\renewcommand{\proofname}{\bf Proof}

\title{Composed inclusions of $A_3$ and $A_4$ subfactors}
\author{Zhengwei Liu}
\date{\today}
\maketitle

\begin{abstract}
In this article, we classify all standard invariants that can arise from a composed inclusion of an $A_3$ with an $A_4$ subfactor.
More precisely, if $\mathcal{N}\subset \mathcal{P}$ is the $A_3$ subfactor and $\mathcal{P}\subset\mathcal{M}$ is the $A_4$ subfactor, then only four standard invariants can arise from the composed inclusion $\mathcal{N}\subset\mathcal{M}$.
This answers a question posed by Bisch and Haagerup in 1994.
The techniques of this paper also show that there are exactly four standard invariants for the composed inclusion of two $A_4$ subfactors.
\end{abstract}

\section{Introduction}
Jones classified the indices of subfactors of type II$_1$ in \cite{Jon83}. It is given by
$$\{4\cos^2(\frac{\pi}{n}), n=3,4,\cdots \}\cup [4,\infty].$$

For a subfactor $\mathcal{N}\subset\mathcal{M}$ of type II$_1$ with finite index,
the Jones tower is a sequence of factors
obtained by repeating $the$ $basic$ $construction$. The system of higher relative commutants
is called the standard invariant of the subfactor \cite{GHJ,Pop90}.
A subfactor is said to be finite depth, if its $principal~graph$ is finite. The standard invariant is a complete invariant of a finite depth subfactor \cite{Pop90}. So we hope to classify the standard invariants of subfactors.



Subfactor planar algebras were introduced by Jones as a diagrammatic axiomatization of the standard invariant \cite{JonPA}.
Other axiomatizations are known as Ocneanu's paragroups \cite{Ocn88} and Popa's $\lambda$-lattices \cite{Pop95}.
Each subfactor planar algebra contains a Temperley-Lieb planar subalgebra which is generated by the sequence of Jones projections.
When the index of the Temperley-Lieb subfactor planar algebra is $4\cos^2(\frac{\pi}{n+1})$, its principal graph is the Coxeter-Dynkin diagram $A_n$.

Given two subfactors $\mathcal{N}\subset\mathcal{P}$ and $\mathcal{P}\subset\mathcal{M}$, the composed inclusion $\mathcal{N}\subset\mathcal{P}\subset\mathcal{M}$ tells the $relative~position$ of these factors. The $group~type~inclusion$ $\mathcal{R}^H\subset\mathcal{R}\subset\mathcal{R}\rtimes K$ for outer actions of finite groups $H$ and $K$ on the hyperfinite factor $\mathcal{R}$ of type II$_1$ was discussed by Bisch and Haagerup \cite{BisHaa96}.

We are interested in studying the composed inclusion of two subfactors of type A, i.e., a subfactor $\mathcal{N}\subset\mathcal{M}$ with an intermediate subfactor $\mathcal{P}$, such that the principal graphs of $\mathcal{N}\subset\mathcal{P}$ and $\mathcal{P}\subset\mathcal{M}$ are type $A$ Coxeter-Dynkin diagrams.
From the planar algebra point of view, the planar algebra of $\mathcal{N}\subset\mathcal{M}$ is a composition of two Temperley-Lieb subfactor planar algebras.
Their tensor product is well known \cite{JonPA}\cite{Liuex}. Their $free~product$ as a minimal composition is discovered by Bisch and Jones $\cite{BisJonFC}$, called the Fuss-Catalan subfactor planar algebra.
In general, the composition of two Temperley-Lieb subfactor planar algebras is still not understood.

The easiest case is the composed inclusion of two $A_3$ subfactors. In this case, the index is 4, and such subfactors are extended type $D$ \cite{GHJ}\cite{Pop94}. They also arise as a group type inclusion $\mathcal{R}^H\subset\mathcal{R}\subset\mathcal{R}\rtimes K$, where $H\cong\mathbb{Z}_2$ and $K\cong\mathbb{Z}_2$.

The first non-group-like case is the composed inclusion of an $A_3$ with an $A_4$ subfactor. Its principal graph is computed by Bisch and Haagerup in their unpublished manuscript in 1994.
Either it is a free composed inclusion, then its planar algebra is Fuss-Catalan; or its principal graph is a Bisch-Haagerup fish graph as
$$\gra{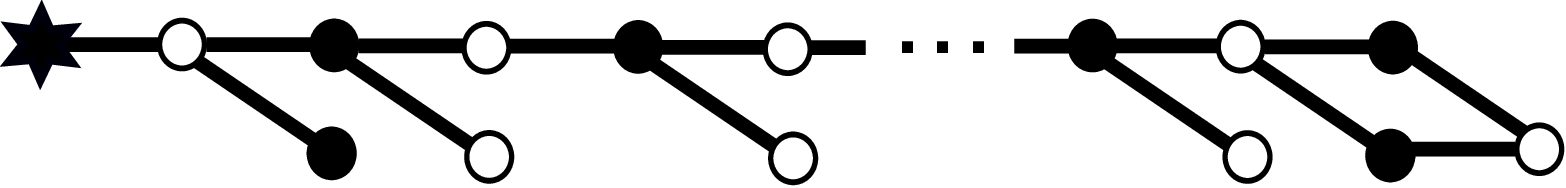}.$$
Then they asked whether this sequence of graphs are the principal graphs of subfactors.
The first Bisch-Haagerup fish graph is the principal graph of the tensor product of an $A_3$ and an $A_4$ subfactor. By considering the flip on $\mathcal{R}\otimes \mathcal{R}$, Bisch and Haagerup constructed a subfactor whose principal graph is the second Bisch-Haagerup fish graph. Later Izumi generalised the Haagerup factor \cite{AsaHaa} while considering endomorphisms of Cuntz algebras \cite{Izu01}, and he constructed an Izumi-Haagerup subfactor for the group $\mathbb{Z}_4$ in his unpublished notes, also called the $3^{\mathbb{Z}_4}$ subfactor \cite{PetPen}. The third Bisch-Haagerup fish graph is the principal graph of an intermediate subfactor of a reduced subfactor of the dual of $3^{\mathbb{Z}_4}$. It turns out the even half is Morita equivalent to the even half of $3^{\mathbb{Z}_4}$.

In this paper, we will prove the following results.
\begin{theorem}\label{main1}
There are exactly four subfactor planar algebras as a composition of an $A_3$ with an $A_4$ planar algebra.
\end{theorem}

This answers the question posed by Bisch and Haagerup.
When $n\geq 4$, the $n_{th}$ Bisch-Haagerup fish graph is not the principal graph of a subfactor.

\begin{theorem}\label{main2}
There are exactly four subfactors planar algebras as a composition of two $A_4$ planar algebras.
\end{theorem}

Now we sketch the ideas of the proof. Following the spirit of \cite{Pet10} \cite{BMPS}, if the principal graph of a subfactor planar algebra is the $n_{th}$ Bisch-Haagerup fish graph, then
by the embedding theorem \cite{JonPen}, the planar algebra is embedded in the $graph~planar~algebra$ \cite{Jon00}. 
By the existence of a ``normalizer" in the Bisch-Haagerup fish graph, there will be a $biprojection$ in the subfactor planar algebra, and the planar subalgebra generated by the biprojection is Fuss-Catalan. The image of the biprojecion is determined by the unique possible $refined~principal~graph$, see Definition \ref{refined principal graph of PA} and Theorem \ref{embedding P}. Furthermore the planar algebra is decomposed as an $annular$ $Fuss-Catalan$ module, similar to the Temperely-Lieb case, \cite{Jonann,JonRez}. Comparing the principal graph of this Fuss-Catalan subfactor planar algebra and the Bisch-Haagerup fish graph, there is a lowest weight vector in the orthogonal complement of Fuss-Catalan. It will satisfy some specific relations,
and there is a ``unique" potential solution of these relations in the graph planar algebra.

The similarity of all the Bisch-Haagerup fish graphs admits us to compute the coefficients of loops of the potential solutions simultaneously.  The coefficients of two sequences of loops has periodicity 5 and 20 with respect to $n$. Comparing with the coefficients of the other two sequences of loops, we will rule out the all the Bisch-Haagerup fish graphs, except the first three.

The existence of the first three follows from the construction mentioned above. The uniqueness follows from the ``uniqueness" of the potential solution.

Furthermore we consider the composition of two $A_4$ planar algebras in the same process. In this list, there are exactly four subfactor planar algebras.
They all arise from reduced subfactors of the four compositions of $A_3$ with $A_4$.

The skein theoretic construction of these subfactor planar algebras could be realized by the $Fuss-Catalan~Jellyfish~relations$ of a generating vector space.

In the meanwhile, Izumi, Morrison and Penneys have ruled out the $4_{th}-10_{th}$ Bisch-Haagerup fish graphs using a different method, see \cite{IMD}.



\begin{ac}
I would like to thank my advisor Vaughan Jones and Dietmar Bisch for a fruitful discussion about this problem and to thank Corey Jones and Jiayi Jiang for computations.
\end{ac}
\section{Background}
We refer the reader to \cite{Jon12} for the definition of planar algebras.
\begin{notation}
In a planar tangle, we use a thick string with a number $k$ to indicate $k$ parallel strings.
\end{notation}
A subfactor planar algebra $\mathscr{S}=\{\mathscr{S}_{n,\pm}\}_{n\in \mathbb{N}_0}$ will be a spherical planar *-algebra over $\mathbb{C}$, such that $\dim(\mathscr{S}_{n,\pm})<\infty$, for all $n$, $\dim(\mathscr{S}_{0,\pm})=1$, and the Markov trace induces a positive definite inner product of $\mathscr{S}_{n,\pm}$ \cite{Jon12}\cite{JonPA}.
Note that $\dim(\mathscr{S}_{0,\pm})=1$, then $\mathscr{S}_{0,\pm}$ is isomorphic to $\mathbb{C}$ as a field. It is spherical means
$$\gra{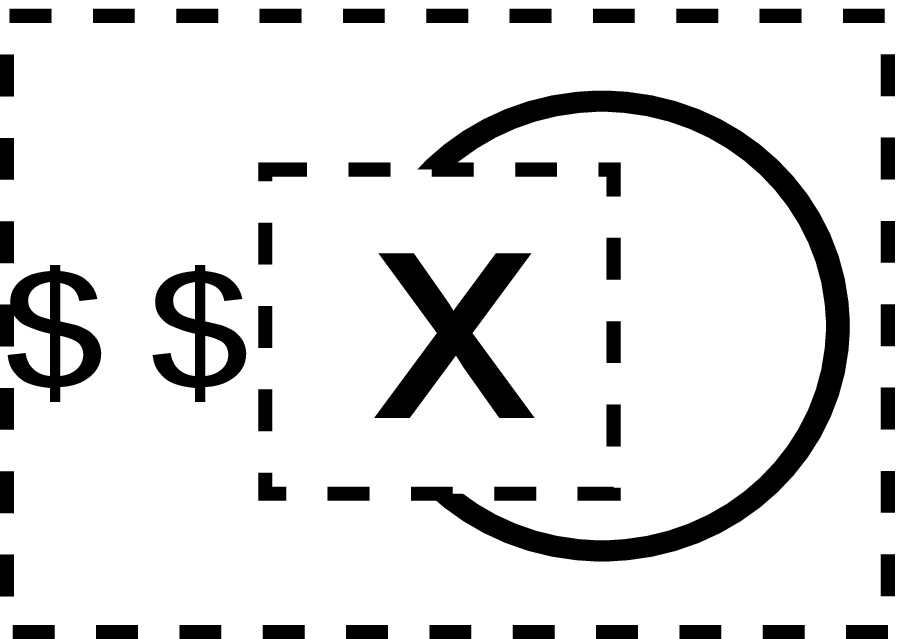}=\gra{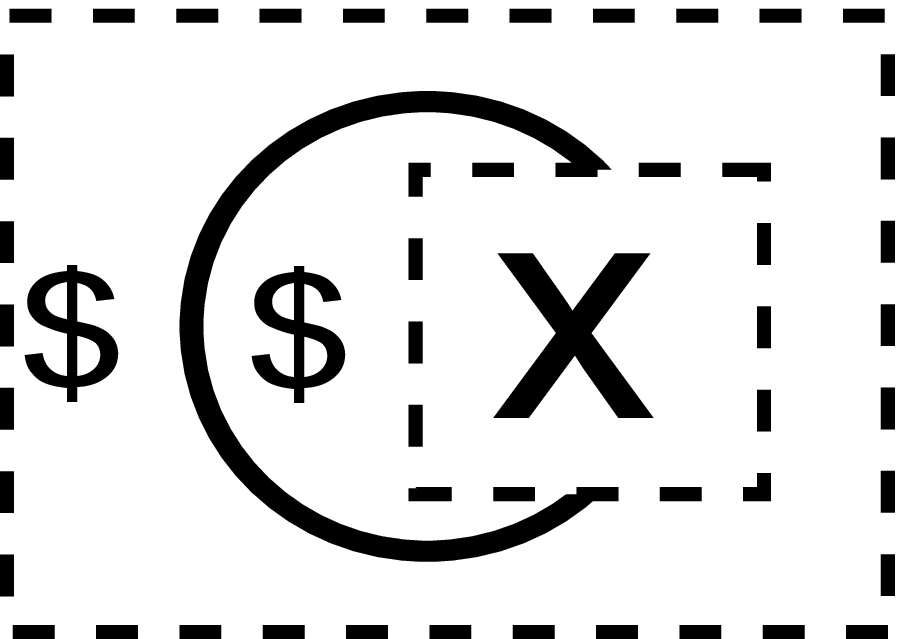}$$
as a number in $\mathbb{C}$, for any $x\in\mathscr{S}_{1,\pm}$.
The inner product of $\mathscr{S}_{n,\pm}$ defined as
$$<y,z>=tr(z^*y)=\grb{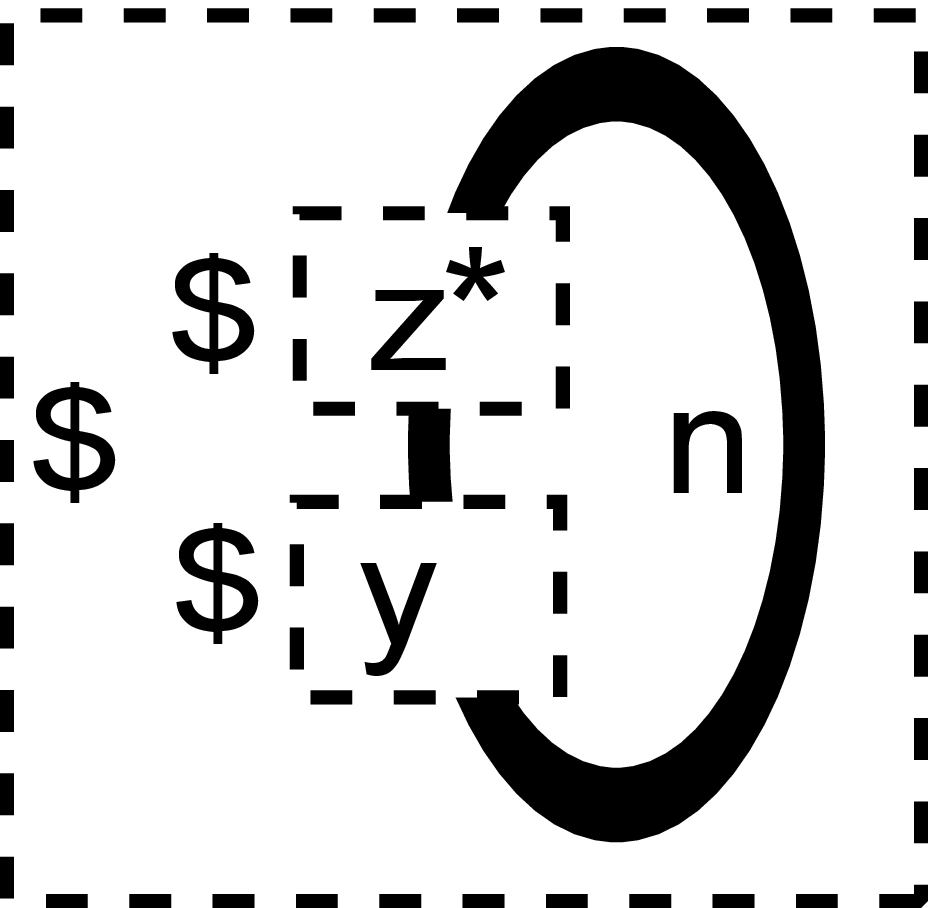},$$
the Markov trace of $z^*y$, for any $y,z\in \mathscr{S}_{n,\pm}$, is positive definite.

It is called a subfactor planar algebra, because it is the same as the standard invariant of a finite index extremal subfactor $\mathcal{N}$ of a factor $\mathcal{M}$ of type II$_1$ \cite{JonPA}.

A subfactor planar algebra is always unital, where unital means any tangle without inner discs can be identified as a vector of $\mathscr{S}$. Note that $\mathscr{S}_{0,\pm}$ is isomorphic to $\mathbb{C}$, the (shaded or unshaded) empty diagram can be identified as the number $1$ in $\mathbb{C}$. The value of a (shaded or unshaded) closed string is $\delta$.
And $\delta^{-1}\gra{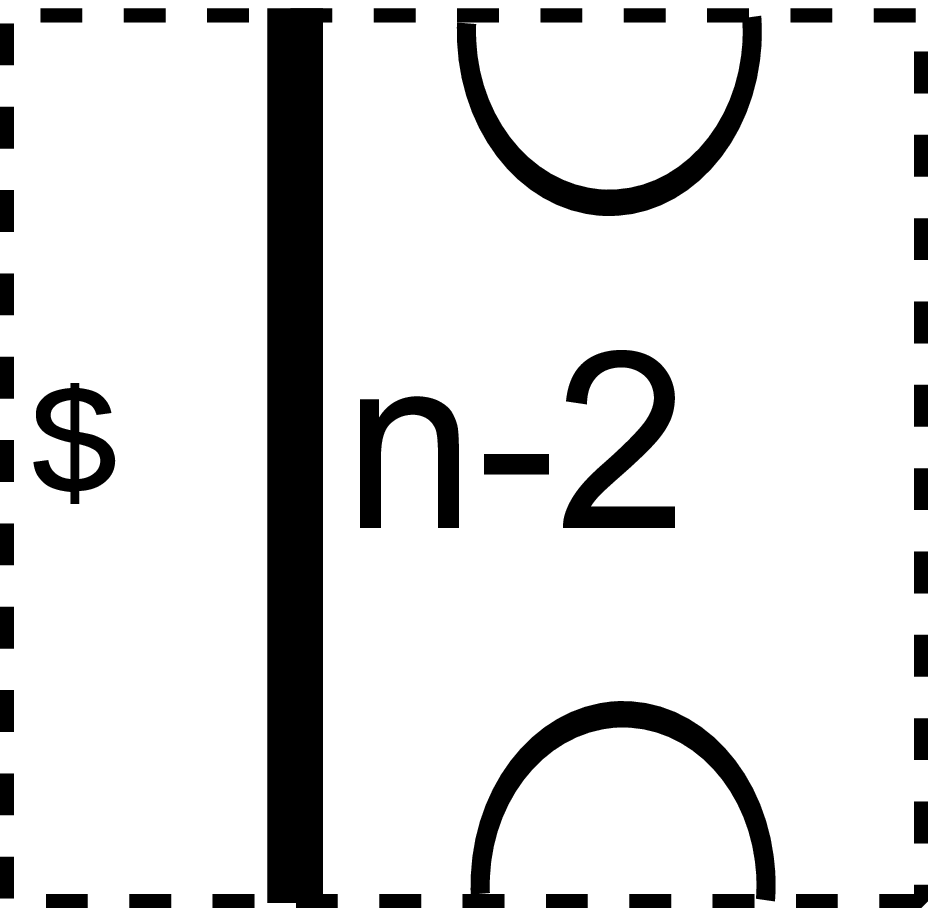}$ in $\mathscr{S}_{n,+}$, denoted by $e_{n-1}$, is the Jones projection $e_{\mathcal{M}_{n-3}}$, for $n\geq2$.
The graded algebra generated by Jones projections is the smallest subfactor planar algebra, well known as the Temperley-Lieb algebra, denoted by $TL(\delta)$. Its vector can be written as a linear sum of tangles without inner discs.

\begin{notation}
We may identify $\mathscr{S}_{-,m}$ as a subspace of $\mathscr{S}_{+,m+1}$ by adding one string to the left.
\end{notation}

\begin{definition}
Let us define the ($1$-string) coproduct of $x\in \mathscr{S}_{i,\pm}$ and $y\in\mathscr{S}_{j,\pm}$, for $i,j\geq 1$, to be
$$x*y=\grb{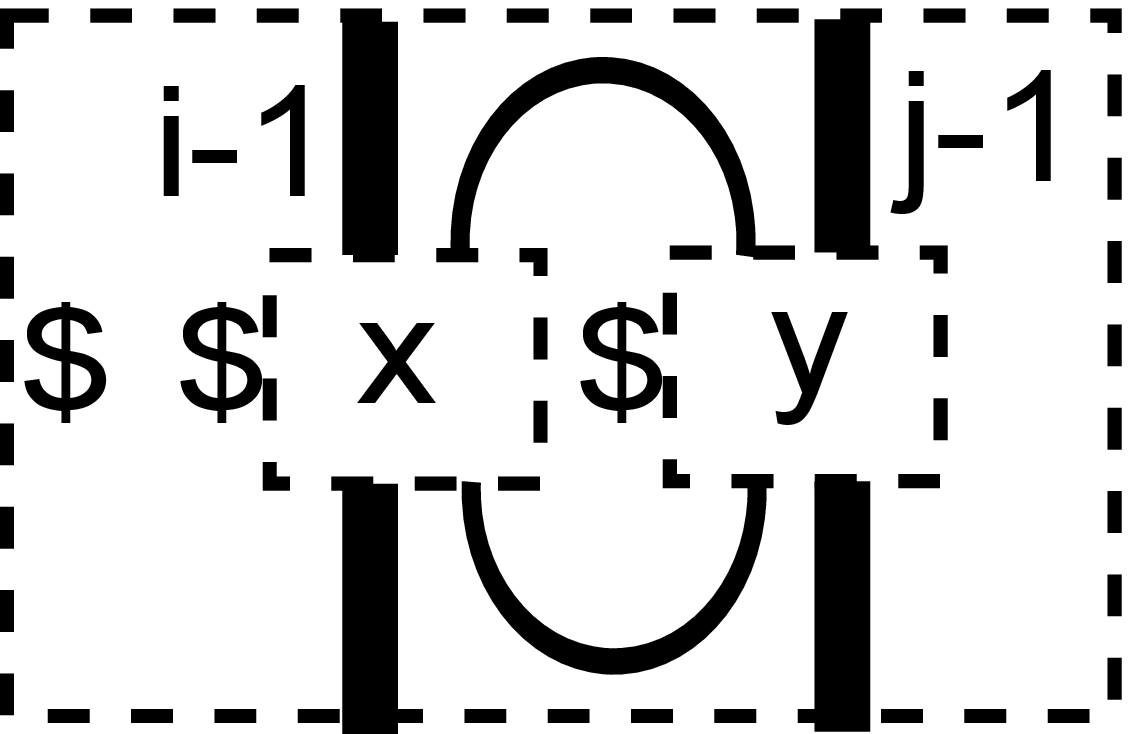},$$ whenever the shading matched.
\end{definition}

Let us recall some facts about the embedding theorem. Then we generalize these results to prove the embedding theorem for an intermediate subfactor in the next section.

\subsection{Principal graphs}
Suppose $\mathcal{N}\subset \mathcal{M}$ is an irreducible subfactor of type II$_1$ with finite index.
Then $L^2(\mathcal{M})$ forms an irreducible $(\mathcal{N},\mathcal{M})$ bimodule, denoted by $X$. Its conjugate $\overline{X}$ is an $(\mathcal{M},\mathcal{N})$ bimodule.
The tensor products $X\otimes\overline{X}\otimes\cdots\otimes \overline{X}$, $X\otimes\overline{X}\otimes\cdots\otimes X$, $\overline{X}\otimes X\otimes\cdots\otimes X$ and $\overline{X}\otimes X\otimes\cdots\otimes\overline{X}$ are decomposed into irreducible bimodules over $(\mathcal{N},\mathcal{N})$, $(\mathcal{N},\mathcal{M})$, $(\mathcal{M},\mathcal{N})$ and $(\mathcal{M},\mathcal{M})$ respectively, where $\otimes$ is Connes fusion of bimodules.
\begin{definition}\label{def principal graph}
The principal graph of the subfactor $\mathcal{N}\subset \mathcal{M}$ is a bipartite graph. Its vertices are equivalent classes of irreducible bimodules over $(\mathcal{N},\mathcal{N})$ and $(\mathcal{N},\mathcal{M})$ in the above decomposed inclusion.
The number of edges connecting two vertices, a $(\mathcal{N},\mathcal{N})$ bimodule $Y$ and a $(\mathcal{N},\mathcal{M})$ bimodule $Z$, is the multiplicity of the equivalent class of $Z$ as a sub bimodule of $Y\otimes \overline{X}$.
The vertex corresponds to the $(\mathcal{N},\mathcal{N})$ bimodule $L^2(\mathcal{N})$ is marked by a star sign.
The dimension vector of the bipartite graph is a function $\lambda$ from the vertices of the graph to $\mathbb{R}^+$. Its value at a vertex is defined to be the dimension of the corresponding bimodule.

The dual principal graph is defined in a similar way.
\end{definition}

\begin{remark}
By Frobenius reciprocity theorem, the multiplicity of $Z$ in $Y\otimes \overline{X}$ equals to the multiplicity of $Y$ in $Z\otimes X$.
\end{remark}


\subsection{The standard invariant}
For an irreducible subfactor $\mathcal{N}\subset\mathcal{M}$ of type II$_1$ with finite index,
the Jones tower is a sequence of factors
$\mathcal{N}\subset \mathcal{M}\subset \mathcal{M}_1 \subset \mathcal{M}_2\subset \cdot\cdot\cdot$
obtained by repeating the~basic~construction. The system of higher relative commutants
$$\begin{array}{ccccccccc}
\mathbb{C}=\mathcal{N}'\cap\mathcal{N}&\subset&\mathcal{N}'\cap\mathcal{M}&\subset&\mathcal{N}'\cap\mathcal{M}_1&\subset&\mathcal{N}'\cap\mathcal{M}_2&\subset&\cdots\\
&&\cup&&\cup&&\cup&&\\
&&\mathrm{C}=\mathcal{M}'\cap\mathcal{M}&\subset&\mathcal{M}'\cap\mathcal{M}_1&\subset&\mathcal{M}'\cap\mathcal{M}_2&\subset&\cdots\\
\end{array}$$
is called the standard invariant of the subfactor \cite{GHJ}\cite{Pop90}.

There is a natural isomorphism between homomorphisms of bimodules $X\otimes\overline{X}\otimes\cdots\otimes \overline{X}$, $X\otimes\overline{X}\otimes\cdots\otimes X$, $\overline{X}\otimes X\otimes\cdots\otimes X$ and $\overline{X}\otimes X\otimes\cdots\otimes\overline{X}$ and the standard invariant of the subfactor \cite{Bis97}. Then the equivalent class of a minimal projection corresponds to an irreducible bimodule. So the principal graph tells how minimal projections are decomposed after the inclusion. Then we may define the principal graph for a subfactor planar algebra without the presumed subfactor.

\begin{proposition}\label{frob e 1}
Suppose $\mathscr{S}$ is a subfactor planar algebra. If $P_1,~P_2$ are minimal projections of $\mathscr{S}_{m,+}$. Then $P_1e_{m+1},~P_2e_{m+1}$ are minimal projections of $\mathscr{S}_{m+2,+}$.
Moreover $P_1$ and $P_2$ are equivalent in $\mathscr{S}_{m,+}$ if and only if $P_1e_{m+1}$ and $P_2e_{m+1}$ are equivalent in $\mathscr{S}_{m+2,+}$.
\end{proposition}

\begin{proposition}[Frobenius Reciprocity]\label{frob e 2}
Suppose $\mathscr{S}$ is a subfactor planar algebra. If $P$ is a minimal projection of $\mathscr{S}_m$ and $Q$ is a minimal projection of $\mathscr{S}_{m+1}$, then $\dim(P\mathscr{S}_{m+1}Q)=\dim(Pe_{m+1}\mathscr{S}_{m+2}Q)$.
\end{proposition}

By the above two propositions, the Bratteli diagram of $\mathscr{S}_m\subset \mathscr{S}_{m+1}$ is identified as a subgraph of the Bratteli diagram of $\mathscr{S}_{m+1}\subset \mathscr{S}_{m+2}$.
So it makes sense to take the limit of the Bratteli diagram of $\mathscr{S}_m\subset \mathscr{S}_{m+1}$, when $m$ approaches infinity.
\begin{definition}
The principal graph of a subfactor planar algebra $\mathscr{S}$ is the limit of the Bratteli diagram of $\mathscr{S}_{m,+}\subset \mathscr{S}_{m+1,+}$. The vertex corresponds to the identity in $\mathscr{S}_{0,+}$ is marked by a star sign. The dimension vector $\lambda$ at a vertex is defined to be the Markov trace of the minimal projection corresponding to that vertex.

Similarly the dual principal graph of a subfactor planar algebra $\mathscr{S}$ is the limit of the Bratteli diagram of $\mathscr{S}_{m,-}\subset \mathscr{S}_{m+1,-}$. The vertex corresponds to the identity in $\mathscr{S}_{0,-}$ is marked by a star sign. The dimension vector $\lambda'$ at a vertex is defined to be the Markov trace of the minimal projection corresponding to that vertex.
\end{definition}

The Bratteli diagram of $\mathscr{S}_m\subset \mathscr{S}_{m+1}$, as a subgraph of the Bratteli diagram of $\mathscr{S}_{m+1}\subset \mathscr{S}_{m+2}$, corresponds to the two-sided ideal $\mathscr{I}_{m+1}$ of $\mathscr{S}_{m+1}$ generated by the Jones projection $e_m$. So the two graphs coincide if and only if $\mathscr{S}_{m+1}=\mathscr{I}_{m+1}$.

\begin{definition}
For a subfactor planar algebra $\mathscr{S}$,
if its principal graph is finite, then the subfactor planar algebra is said to be finite depth. Furthermore it is of depth $m$, if $m$ is the smallest number such that $\mathscr{S}_{m+1}=\mathscr{S}_{m+1}e_m\mathscr{S}_{m+1}$.
\end{definition}

\subsection{Finite-dimensional inclusions}
We refer the reader to Chapter 3 of \cite{JS} for the inclusions of finite dimensional von Neumann algebras.

\begin{definition}
Suppose $\mathcal{A}$ is a finite-dimensional von Neumann algebra and $\tau$ is a trace on it. The dimension vector $\lambda_\mathcal{A}^\tau$ is a function from the set of minimal central projections (or equivalent classes of minimal projections or irreducible representations up to unitary equivalence) of $\mathcal{A}$ to $\mathbb{C}$ with following property, for any minimal central projection $z$, $\lambda_A^\tau(z)=\tau(x)$, where $x\in A$ is a minimal projection with central support $z$.
\end{definition}

The trace of a minimal projection only depends on its equivalent class, so the dimension vector is well defined.
On the other hand, given a function from the set of minimal central projections of $\mathcal{A}$ to $\mathbb{C}$, we may construct a trace of $\mathcal{A}$, such that the corresponding dimension vector is the given function. So it is a one-to-one map.

Let us recall some facts about the inclusion of finite dimensional von Neumann algebras $\mathcal{B}_0\subset \mathcal{B}_1$.

The Bratteli diagram $Br$ for the inclusion $\mathcal{B}_0\subset \mathcal{B}_1$ is a bipartite graph. Its even or odd vertices are indexed by the equivalence classes of irreducible representations of $\mathcal{B}_0$ or $\mathcal{B}_1$ respectively.  The number of edges connects a vertex corresponding to an irreducible representation $U$ of $\mathcal{B}_0$ to a vertex corresponding to an irreducible representation $V$ of $\mathcal{B}_1$ is given by the multiplicity of $U$ in the restriction of $V$ on $\mathcal{B}_0$.

Let $Br_\pm$ be the even/odd vertices of $Br$. The Bratteli diagram can be interpreted as the adjacent matrix $\Lambda=\Lambda_{\mathcal{B}_0}^{\mathcal{B}_1}: L^2(Br_-)\rightarrow L^2(Br_+)$, where $\Lambda_{u,v}$ is defined as the number of edges connects $u$ to $v$ for any $u\in Br_+,~ v\in Br_-.$

\begin{proposition}
For the inclusion $\mathcal{B}_0\subset \mathcal{B}_1$ and a trace $\tau$ on it, we have $\lambda_{\mathcal{B}_0}^{\tau}=\Lambda\lambda_{\mathcal{B}_1}^{\tau}$.
\end{proposition}

If the trace $\tau$ is a faithful state, then by GNS construction we will obtain a right $\mathcal{B}_1$ module $L^2(\mathcal{B}_1)$. And $L^2(\mathcal{B}_0)$ is identified as a subspace of $L^2(\mathcal{B}_1)$. Let $e$ be the Jones projection on to the subspace $L^2(\mathcal{B}_0)$. Let $\mathcal{B}_2$ be the von Neumann algebra $(\mathcal{B}_1\cup\{e\})''$. Then we obtain a tower $\mathcal{B}_0\subset\mathcal{B}_1\subset\mathcal{B}_2$ which is called the basic construction. Furthermore if the tracial state $\tau$ satisfies the condition $\Lambda^*\Lambda\lambda_{\mathcal{B}_1}^\tau=\mu\lambda_{\mathcal{B}_1}^\tau$ for some scalar $\mu$, then it is said to be a Markov trace. In this case the scalar $\mu$ is $||\Lambda||^2$. Then $\lambda^{\tau}=
\begin{bmatrix}
\lambda_{\mathcal{B}_0}^{\tau}\\
\delta\lambda_{\mathcal{B}_1}^{\tau}\\
\end{bmatrix}$
is a Perron-Frobenius eigenvector for
$\begin{bmatrix}
0&\Lambda\\
\Lambda^*&0\\
\end{bmatrix}$.

\begin{definition}
We call $\lambda^{\tau}$ the Perron-Frobenius eigenvector with respect to the Markov trace $\tau$.
\end{definition}

The existence of a Markov trace for the inclusion $\mathcal{B}_0\subset\mathcal{B}_1$ follows from the Perron-Frobenius theorem. The Markov trace is unique if and only if the Bratteli diagram for the inclusion $\mathcal{B}_0\subset\mathcal{B}_1$ is connected.

We will see the importance of the Markov trace from the following proposition.

\begin{proposition}
If $\tau$ is a Markov trace for the inclusion $\mathcal{B}_0\subset\mathcal{B}_1$, then $\tau$ extends uniquely to a trace on $\mathcal{B}_2$, still denoted by $\tau$. Moreover $\tau$ is a Markov trace for the inclusion $\mathcal{B}_1\subset\mathcal{B}_2$.
\end{proposition}

In this case, we may repeat the basic construction to obtain a sequence of finite dimensional von Neumann algebras $\mathcal{B}_0\subset\mathcal{B}_1\subset\mathcal{B}_2\subset\mathcal{B}_3\subset\cdots$ and a sequence of Jones projections $e_1,e_2,e_3\cdots$.

\subsection{Graph Planar Algebras}\label{graphpa}
Given a finite connected bipartite graph $\Gamma$, it can be realised as the Bratteli diagram of the inclusion of finite dimensional von Neumann algebras $\mathcal{B}_0\subset \mathcal{B}_1$ with a (unique) Markov trace. Applying the basic construction, we will obtain the
sequence of finite dimensional von Neumann algebras $\mathcal{B}_0\subset\mathcal{B}_1\subset\mathcal{B}_2\subset\mathcal{B}_3\subset\cdots$. Take $\mathscr{S}_{m,+}$ to be $\mathcal{B_0}'\cap \mathcal{B}_m$ and $\mathscr{S}_{m,-}$ to be $\mathcal{B_1}'\cap \mathcal{B}_{m+1}$. Then $\{\mathscr{S}_{m,\pm}\}$ forms a planar algebra, called the $graph~planar~algebra$ of the bipartite graph $\Gamma$. Moreover $\mathscr{S}_{m,\pm}$ has a natural basis given by length $2m$ loops of $\Gamma$. We refer the reader to \cite{Jon00,JonPen} for more details.
We cite the conventions used in section 3.4 of \cite{JonPen}.

\begin{definition}
Let us define $\mathscr{G}=\{\mathscr{G}_{m,\pm}\}$ to be the graph planar algebra of a finite connected bipartite graph $\Gamma$.
Let $\lambda$ be the Perron-Frobenius eigenvector with respect to the Markov trace.
\end{definition}

A vertex of the $\Gamma$ corresponds to an equivalent class of minimal projections, so $\lambda$ is also defined as a function from $\mathcal{V}_\pm$ to $\mathbb{R}^+$.
If $\Gamma$ is the principal graph of a subfactor, then its dimension vector is a multiple of the Perron-Frobenius eigenvector.
In this paper, we only need the proportion of values of $\lambda$ at vertices. We do not have to distinguish these two vectors.

Let $\mathcal{V}_\pm$ be the sets of black/white vertices of $\Gamma$, and let $\mathcal{E}$ be the sets of all edges of $\Gamma$ directed from black to white vertices. Then we have the source and target functions $s:\mathcal{E}\rightarrow \mathcal{V}_+$ and $t:\mathcal{E}\rightarrow \mathcal{V}_-$. For a directed edge $\varepsilon\in \mathcal{E}$, we define $\varepsilon^*$ to be the same edge with an opposite direction. The source function $s:\mathcal{E}^*=\{ \varepsilon^* | \varepsilon\in\mathcal{E}\}\rightarrow \mathcal{V}_-$ and the target function $t:\mathcal{E}^*\rightarrow \mathcal{V}_+$ are defined as $s(\varepsilon^*)=t(\varepsilon)$ and $t(\varepsilon^*)=s(\varepsilon)$.

A length 2m loop in $\mathscr{G}_{m,+}$ is denoted by $[\varepsilon_1\varepsilon_2^*\cdots\varepsilon_{2m-1}\varepsilon_{2m}^*]$ satisfying

(i)$t(\varepsilon_k)=s(\varepsilon_{k+1}^*)=t(\varepsilon_{k+1})$, for all odd $k<2m$;

(ii)$t(\varepsilon_k^*)=s(\varepsilon_{k})=t(\varepsilon_{k+1})$, for all even $k<2m$;

(iii)$t(\varepsilon_{2m}^*)=s(\varepsilon_{2m})=t(\varepsilon_{1})$.

The graph planar algebra is always unital. The unshaded empty diagram is given by $\sum_{v\in\mathcal{V}_+} v$; And the shaded empty diagram is given by $\sum_{v\in\mathcal{V}_-} v$.
It is mentioning that the Jones projection is given by
$$e_1=\delta^{-1}\gra{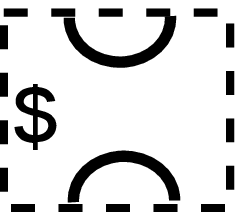}=\delta^{-1}\sum_{s(\varepsilon_1)=s(\varepsilon_3)} \sqrt{\frac{\lambda(t(\varepsilon_1))\lambda(t(\varepsilon_3))}{\lambda(s(\varepsilon_1))\lambda(s(\varepsilon_3))}}[\varepsilon_1\varepsilon_1^*\varepsilon_3\varepsilon_3^*].$$

Now let us describe the actions on $\mathscr{G}$. The adjoint operation is defined as the anti-linear extension of $$[\varepsilon_1\varepsilon_2^*\cdots\varepsilon_{2m-1}\varepsilon_{2m}^*]^*=[\varepsilon_{2m}\varepsilon_{2m-1}^*\cdots\varepsilon_{2}\varepsilon_{1}^*].$$
For $\mathscr{G}_{m,-}$, we have similar conventions.

\begin{definition}
The Fourier transform $\mathcal{F}: \mathscr{G}_{m,+}\rightarrow \mathscr{G}_{m,-}, m>0$ is defined as the linear extension of
$$\mathcal{F}([\varepsilon_1\varepsilon_2^*\cdots\varepsilon_{2m-1}\varepsilon_{2m}^*])=
\left\{
\begin{array}{ll}
\sqrt{\frac{\lambda(s(\varepsilon_{2m}))}{\lambda(t(\varepsilon_{2m}))}} \sqrt{\frac{\lambda(s(\varepsilon_m))}{\lambda(t(\varepsilon_m))}}[\varepsilon_{2m}^*\varepsilon_1\varepsilon_2^*\cdots\varepsilon_{2m-1}]& for~m~even.\\
\sqrt{\frac{\lambda(s(\varepsilon_{2m}))}{\lambda(t(\varepsilon_{2m}))}} \sqrt{\frac{\lambda(t(\varepsilon_m))}{\lambda(s(\varepsilon_m))}}[\varepsilon_{2m}^*\varepsilon_1\varepsilon_2^*\cdots\varepsilon_{2m-1}]& for~m~odd
\end{array}
\right.$$

Similarly it is also defined from $\mathscr{G}_{m,-}$ to $\mathscr{G}_{m,+}$.
\end{definition}

The Fourier transform has a diagrammatic interpretation as a one-click rotation
$$\grb{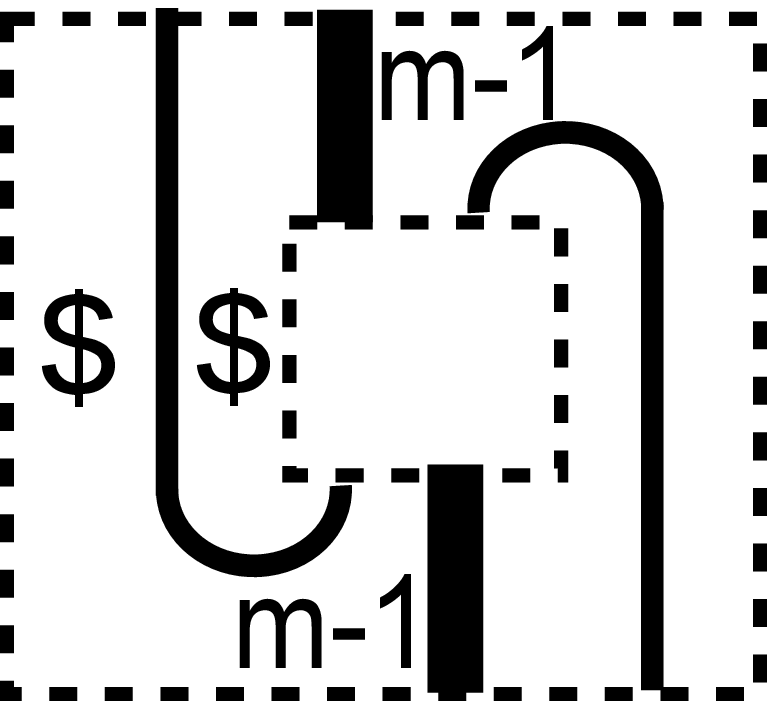}.$$

\begin{definition}
Let us define $\rho$ to be $\mathcal{F}^2$. Then $\rho$ is defined from $\mathscr{G}_{m,+}$ to $\mathscr{G}_{m,+}$ as a two-click rotation for $m>0$,
$$\rho([\varepsilon_1\varepsilon_2^*\cdots\varepsilon_{2m-1}\varepsilon_{2m}^*])=
\sqrt{\frac{\lambda(s(\varepsilon_{2m}))}{\lambda(s(\varepsilon_{2m-1}))}} \sqrt{\frac{\lambda(s(\varepsilon_m))}{\lambda(s(\varepsilon_{m-1}))}}
[\varepsilon_{2m-1}\varepsilon_{2m}^*\varepsilon_1\varepsilon_2^*\cdots\varepsilon_{2m-3}\varepsilon_{2m-2}^*].$$
It is similar for $\mathscr{G}_{m,-}$.
\end{definition}

For $l_1,l_2\in\mathscr{G}_{m,+},~l_1=[\varepsilon_1\varepsilon_2^*\cdots\varepsilon_{2m-1}\varepsilon_{2m}^*],~l_2=[\xi_1\xi_2^*\cdots\xi_{2m-1}\xi_{2m}]$, we have

$$\grb{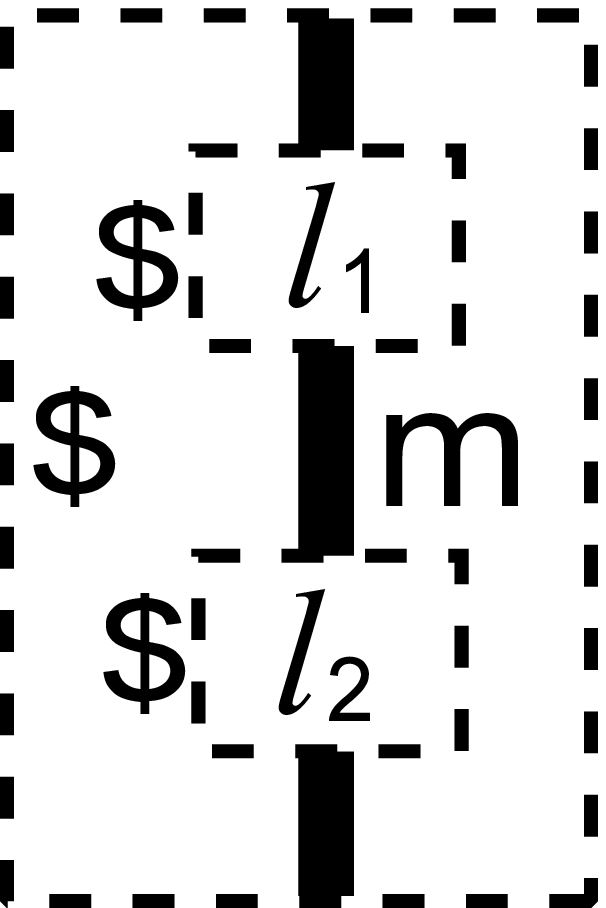}=
\left\{
\begin{array}{ll}
\prod_{1\leq k\leq m} \delta_{\varepsilon_{m+k}, \xi_{m+1-k}} [\varepsilon_1 \varepsilon_2^*\cdots\varepsilon_m^* \xi_{m+1} \cdots\xi_{2m-1}\xi_{2m}^*]     & when~m~is~even;\\
\prod_{1\leq k\leq m} \delta_{\varepsilon_{m+k}, \xi_{m+1-k}} [\varepsilon_1 \varepsilon_2^*\cdots\varepsilon_m \xi_{m+1}^* \cdots\xi_{2m-1}\xi_{2m}^*]     & when~m~is~odd.
\end{array}
\right.$$

$$\graa{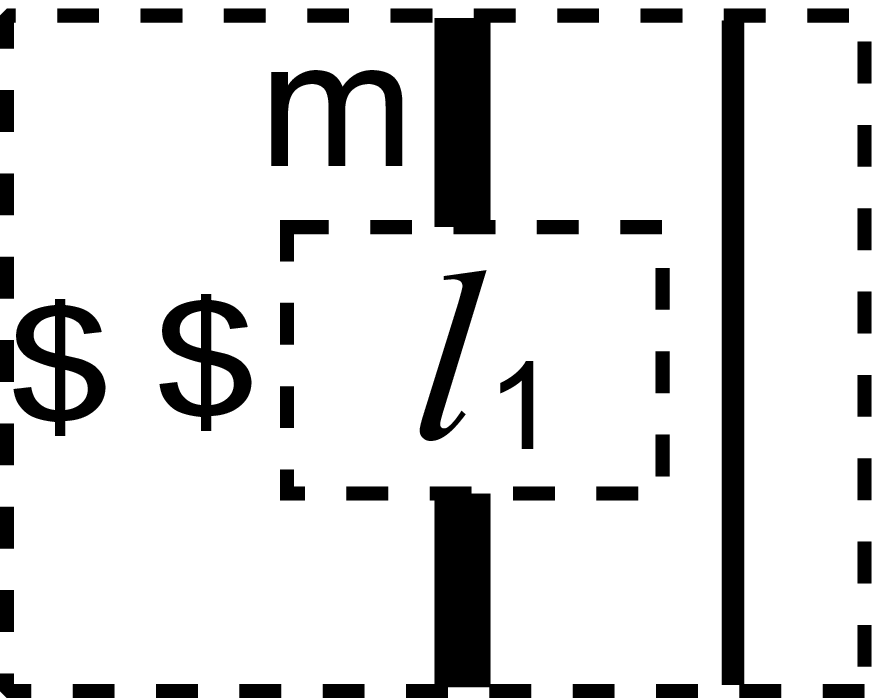}=
\left\{
\begin{array}{ll}
\sum_{s(\varepsilon)=s(\varepsilon_m)} [\varepsilon_1 \varepsilon_2^*\cdots\varepsilon_m^*\varepsilon \varepsilon^* \varepsilon_{m+1} \cdots\varepsilon_{2m-1}\varepsilon_{2m}^*]     & when~m~is~even;\\
\sum_{t(\varepsilon)=t(\varepsilon_m)} [\varepsilon_1 \varepsilon_2^*\cdots\varepsilon_m \varepsilon^* \varepsilon \varepsilon_{m+1}^* \cdots\varepsilon_{2m-1}\varepsilon_{2m}^*]     & when~m~is~odd.
\end{array}
\right.$$

$$\graa{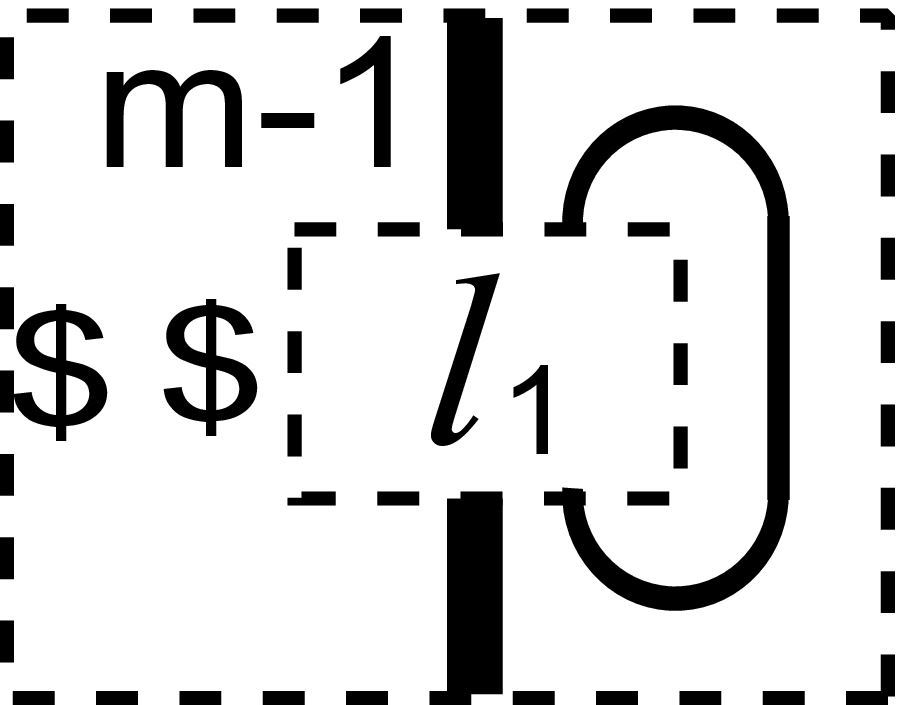}=
\left\{
\begin{array}{ll}
\delta_{\varepsilon_{m}, \varepsilon_{m+1}} \frac{\lambda(s(\varepsilon_m))}{\lambda(t(\varepsilon_m))}[\varepsilon_1 \varepsilon_2^*\cdots\varepsilon_m^*\varepsilon \varepsilon^* \varepsilon_{m+1} \cdots\varepsilon_{2m-1}\varepsilon_{2m}^*]     & when~m~is~even;\\
\delta_{\varepsilon_{m}, \varepsilon_{m+1}} \frac{\lambda(t(\varepsilon_m))}{\lambda(s(\varepsilon_m))} [\varepsilon_1 \varepsilon_2^*\cdots\varepsilon_m \varepsilon^* \varepsilon \varepsilon_{m+1}^* \cdots\varepsilon_{2m-1}\varepsilon_{2m}^*]     & when~m~is~odd.
\end{array}
\right.$$

In general, the action of a planar tangle could be realised as a composed inclusion of actions mentioned above. It has a nice formula, see page 11 in \cite{Jon00}.

\subsection{The embedding theorem}
For a depth $2r$ (or $2r+1$) subfactor planar algebra $\mathscr{S}$, we have $$\mathscr{S}_{m+1}=\mathscr{S}_{m+1}e_{m}\mathscr{S}_{m+1}=\mathscr{S}_{m}e_{m+1}\mathscr{S}_{m}, ~\text{whenever}~ m\geq 2r+1.$$ So $\mathscr{S}_{m-1}\subset \mathscr{S}_{m} \subset \mathscr{S}_{m+1}$ forms a basic construction. Note that the Bratteli diagram of $\mathscr{S}_{2r}\subset\mathscr{S}_{2r+1}$ is the principal graph. So the graph planar algebra $\mathscr{G}$ of the principal graph is given by $$\mathscr{G}_{k,+}=\mathscr{S}_{2r}'\cap\mathscr{S}_{2r+k}; \quad \mathscr{G}_{k,-}=\mathscr{S}_{2r+1}'\cap\mathscr{S}_{2r+k+1}.$$ Moreover the map $\Phi: \mathscr{S}\rightarrow\mathscr{G}$ by adding $2r$ strings to the left preserves the planar algebra structure. It is not obvious that the left conditional expectation is preserved. We have the following embedding theorem, see Theorem 4.1 in \cite{JonPen}.
\begin{theorem}
A finite depth subfactor planar algebra is naturally embedded into the graph planar algebra of its principal graph.
\end{theorem}

\begin{remark}
A general embedding theorem is proved in \cite{MorWal}.
\end{remark}
\subsection{Fuss-Catalan}
The Fuss-Catalan subfactor planar algebras are discovered by Bisch and Jones as $free~products$ of Temperley-Lieb subfactor planar algebras while studying the intermediate subfactors of a subfactor \cite{BisJonFC}.
We refer the reader to \cite{BisJonfree} \cite{Lan02} for the definition of the free product of subfactor planar algebras.
It has a nice diagrammatic interpretation. For two Temperley-Lieb subfactor planar algebras $TL(\delta_a)$ and $TL(\delta_b)$, their free product $FC(\delta_a,\delta_b)$ is a subfactor planar algebra. A vector in $FC(\delta_a,\delta_b)_{m,+}$ can be expressed as a linear sum of Fuss-Catalan diagrams, a diagram consisting of disjoint $a,b$-colour strings whose boundary points are ordered as $\underbrace{abba~abba\cdots abba}_m$, $m$ copies of $abba$, after the dollar sign. It is similar for a vector in $FC(\delta_a,\delta_b)_{m,-}$, but the boundary points are ordered as $\underbrace{baab~baab\cdots baab}_m$. For the action of a planar tangle on a simple tensor of Fuss-Catalan diagrams, first we replace each string of the planar tangle by a pair of parallel a-colour and b-colour strings which matches the a,b-colour boundary points, then the out put is $gluing$ the new tangle with the input diagrams. If there is an $a$ or $b$-colour closed circle, then it contributes to a scalar $\delta_a$ or $\delta_b$ respectively.

The Fuss-Catalan subfactor planar algebra $FC(\delta_a,\delta_b)$ is naturally derived from an intermediate subfactor of a subfactor.
Suppose $\mathcal{N}\subset \mathcal{M}$ is an irreducible subfactor with finite index, and $\mathcal{P}$ is an intermediate subfactor. Then there are two Jones projections $e_{\mathcal{N}}$ and $e_{\mathcal{P}}$ acting on $L^2(\mathcal{M})$,
and we have the basic construction $\mathcal{N}\subset\mathcal{P}\subset\mathcal{M}\subset \mathcal{P}_1\subset \mathcal{M}_1$.
Repeating this process, we will obtain a sequence of factors $\mathcal{N}\subset\mathcal{P}\subset\mathcal{M}\subset \mathcal{P}_1\subset \mathcal{M}_1\subset\mathcal{P}_2\subset\mathcal{M}_2\cdots$ and a sequence of Jones projections $e_{\mathcal{N}}$, $e_{\mathcal{P}}$, $e_{\mathcal{M}}$, $e_{\mathcal{P}_1}\cdots$.
The algebra generated by these Jones projections forms a planar algebra. That is $FC(\delta_a,\delta_b)$, where $\delta_a=\sqrt{[\mathcal{P}:\mathcal{N}]}$ and $\delta_b=\sqrt{[\mathcal{M}:\mathcal{P}]}$.
Moreover $e_{\mathcal{P}}\in FC(\delta_a,\delta_b)_{2,+}$ and $e_{\mathcal{P}_1}\in FC(\delta_a,\delta_b)_{2,-}$ could be expressed as $\delta_b^{-1}\graa{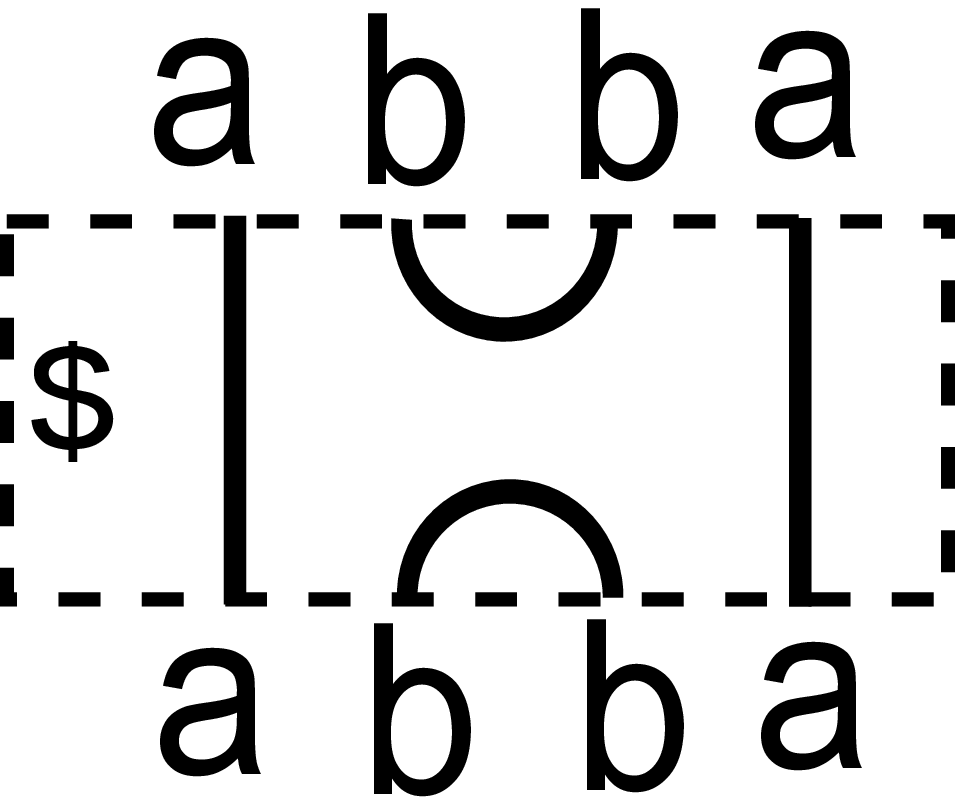}$ and $\delta_a^{-1}\graa{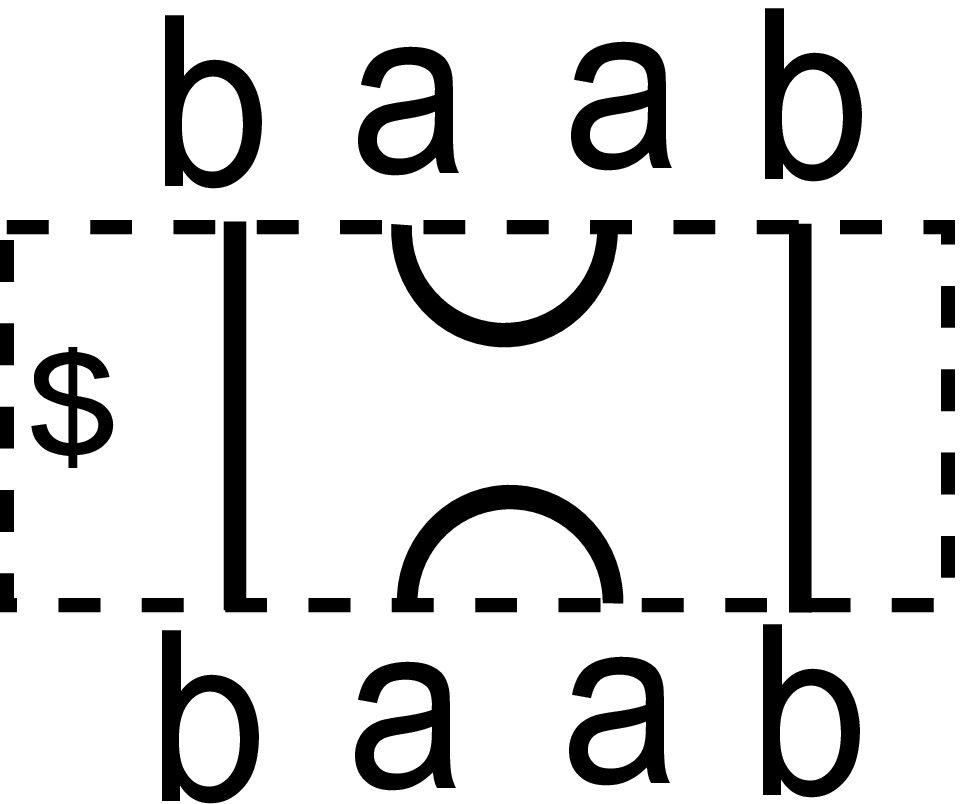}$ respectively.
Specifically $\mathcal{F}(e_{\mathcal{P}})$ is a multiple of $e_{\mathcal{P}_1}$.

\begin{definition}
For a subfactor planar algebra $\mathscr{S}$,
a projection $Q\in \mathscr{S}_{2,+}$ is called a biprojection, if $\mathcal{F}(Q)$ is a multiple of a projection.
\end{definition}

Suppose $\mathscr{S}$ is the planar algebra for $\mathcal{N}\subset \mathcal{M}$, then $e_{\mathcal{P}}\in\mathscr{S}_{2,+}$ is a biprojection. Conversely all the biprojections in $\mathscr{S}_{2,+}$ are realised in this way. That means there is a one-to-one correspondence between intermediate subfactors and biprojections.

\begin{proposition}\label{ex of bipro}
If we identify $\mathscr{S}_{2,-}$ as a subspace of $\mathscr{S}_{3,+}$ by adding a string to the left, then a biprojection $Q\in\mathscr{S}_{2,+}$ will satisfy
$Q\mathcal{F}(Q)=\mathcal{F}(Q)Q$, i.e.
$$\graa{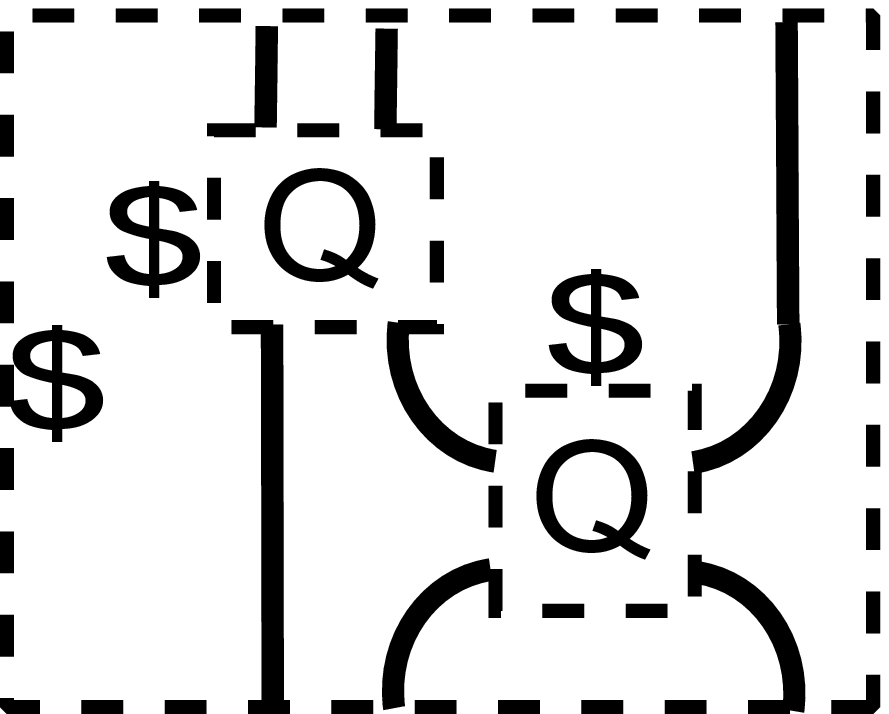}=\graa{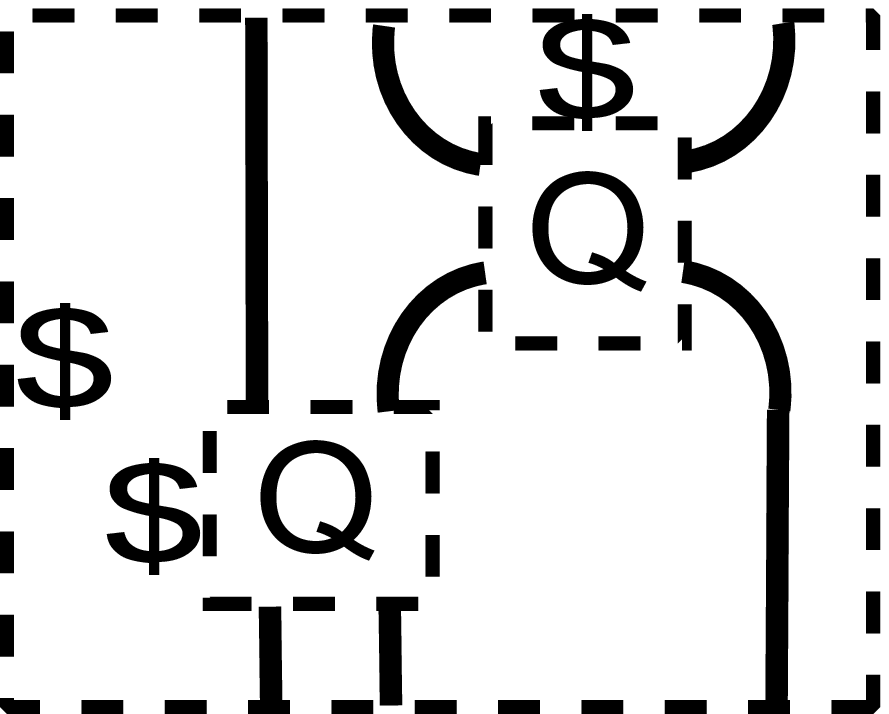},$$
called the exchange relation of a biprojection.
\end{proposition}

Conversely if a self-adjoint operator in $\mathscr{S}_{2,+}$ satisfies the exchange relation, then it is a biprojection.
We refer the reader to \cite{Liuex} for some other approaches to the biprojection.
The Fuss-Catalan subfactor planar algebra could also be viewed as a planar algebra generated by a biprojection with its exchange relation.

If there is a subfactor planar algebra whose principal graph is a Bisch-Haagerup fish graph, then it has a trace-2 biprojection, due to the existence of a ``normalizer".
So it contains $FC(\delta_a,\delta_b)$, where $\delta_a=\sqrt{2},\delta_b=\frac{\sqrt{5}+1}{2}$, as a planar subalgebra.
The principal graph and dual principal graph of $FC(\delta_a,\delta_b)$ are given as

$$\gra{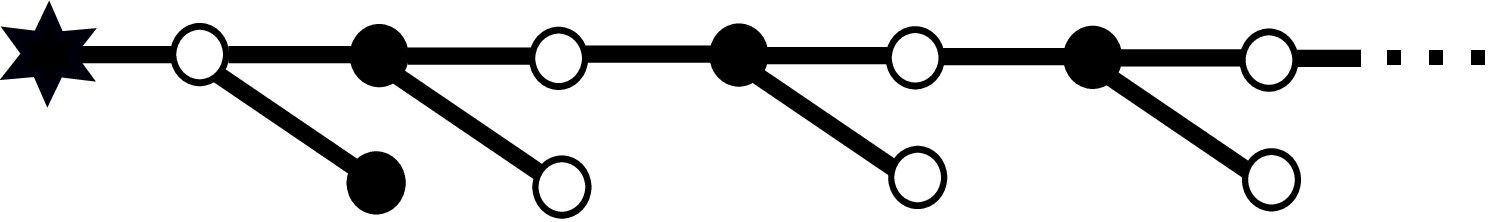}$$
$$\grb{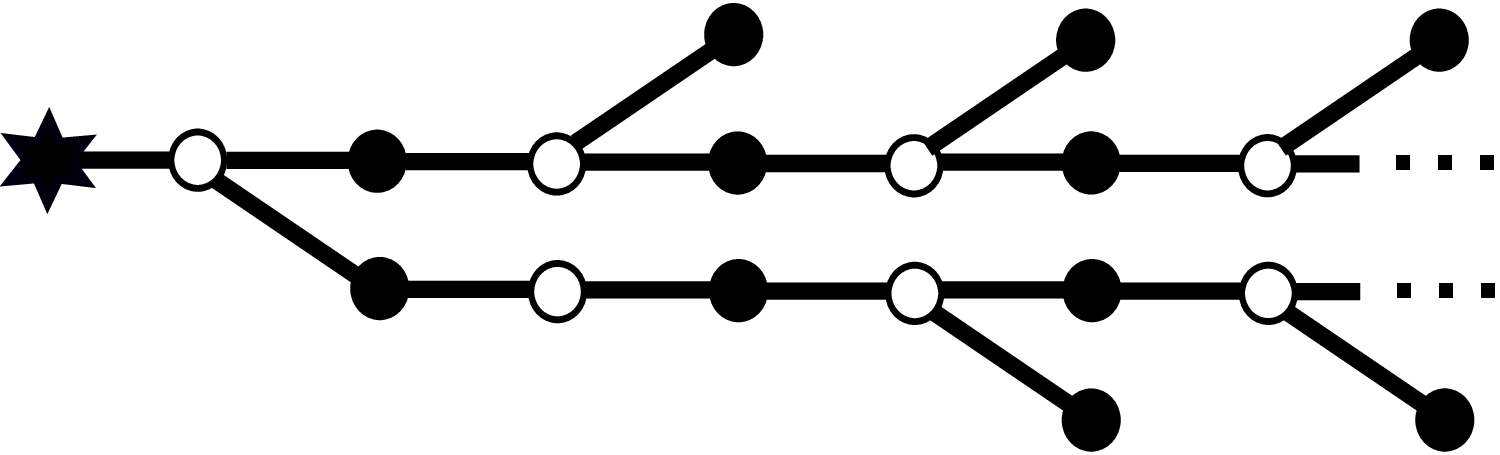}.$$

\section{The embedding theorem for an intermediate subfactor}
If there is a subfactor planar algebra $\mathscr{S}$ whose principal graph is a Bisch-Haagerup fish graph $\Gamma$, then it is embedded in the graph planar algebra $\mathscr{G}$ of $\Gamma$, by the embedding theorem.
While $\mathscr{S}_{2,+}$ contains a trace-2 biprojection. We hope to know the image of the biprojection in $\mathscr{G}$. Recall that the image of the Jones projection $e_1$ is determined by the principal graph,
$$\delta e_1=\sum_{s(\varepsilon_1)=s(\varepsilon_3)} \sqrt{\frac{\lambda(t(\varepsilon_1))}{\lambda(s(\varepsilon_1))}\frac{\lambda(t(\varepsilon_3))}{\lambda(s(\varepsilon_3))}}[\varepsilon_1\varepsilon_1^*\varepsilon_3\varepsilon_3^*].$$
The image of the biprojection has a similar formula. It is determined by the $refined~principal~graph$. The refined principal graph is already considered by Bisch and Haagerup for bimodules, by Bisch and Jones for planar algebras. For the embedding theorem, we will use the one for planar algebras.

The lopsided version of embedding theorem for an intermediate subfactor is involved in a general embedding theorem proved by Morrison in \cite{MorWal}.
To consider some algebraic structures, we need the spherical version of the embedding theorem. Their relations are described in\cite{MorPet12}.
For convenience, we prove the spherical version of embedding theorem, similar to the one proved by Jones and Penneys in \cite{JonPen}.

In this section, we always assume $\mathcal{N}\subset\mathcal{M}$ is an irreducible subfactor of type II$_1$ with finite index, and $\mathcal{P}$ is an intermediate subfactor.
If the subfactor has an intermediate subfactor, then its planar algebra becomes an $\mathcal{N}-\mathcal{P}-\mathcal{M}$ planar algebras. For $\mathcal{N}-\mathcal{P}-\mathcal{M}$ planar algebras, we refer the reader to Chapter 4 in \cite{Harth}.
In this case, the subfactor planar algebra contains a biprojection $P$, and a planar tangle labeled by $P$ can be replaced by a $Fuss-Catalan~planar~tangle$. In this paper, we will use planar tangles labeled by $P$, instead of Fuss-Catalan planar tangles.

\subsection{Principal graphs}
For the embedding theorem, we will consider the principal graph of $\mathcal{N}\subset\mathcal{P}\subset\mathcal{M}$.
It refines the principal graph of $\mathcal{N}\subset\mathcal{M}$. Instead of a bipartite graph, it will be an ($\mathcal{N},\mathcal{P},\mathcal{M}$) coloured graph.

\begin{definition}\label{def three}
An ($\mathcal{N},\mathcal{P},\mathcal{M}$) coloured graph $\Gamma$ is a locally finite graph, such that the set $\mathcal{V}$ of its vertices is divided into three disjoint subsets $\mathcal{V}_{\mathcal{N}}$, $\mathcal{V}_{\mathcal{P}}$ and $\mathcal{V}_{\mathcal{M}}$, and the set $\mathcal{E}$ of its edges is divided into two disjoint subsets $\mathcal{E}_{+}$, $\mathcal{E}_{-}$. Moreover every edge in $\mathcal{E}_{+}$ connects a vertex in $\mathcal{V}_{\mathcal{N}}$ to one in $\mathcal{V}_{\mathcal{P}}$ and
every edge in $\mathcal{E}_{-}$ connects a vertex in $\mathcal{V}_{\mathcal{P}}$ to one in $\mathcal{V}_{\mathcal{M}}$.
Then we define the source function $s: \mathcal{E}\rightarrow \mathcal{V}_{\mathcal{N}} \cup \mathcal{V}_{\mathcal{M}}$ and the target function $t: \mathcal{E}\rightarrow \mathcal{V}_{\mathcal{P}}$ in the obvious way.
The operation $*$ reverses the direction of an edge.
\end{definition}


\begin{definition}
From an ($\mathcal{N},\mathcal{P},\mathcal{M}$) coloured graph $\Gamma$, we will obtain a ($\mathcal{N},\mathcal{M}$) coloured bipartite graph $\Gamma'$ as follows,
the $\mathcal{N}/\mathcal{M}$ coloured vertices of $\Gamma'$ are identical to the $\mathcal{N}/\mathcal{M}$ coloured vertices of $\Gamma$;
for two vertices $v_n$ in $\mathcal{V}_{\mathcal{N}}$ and $v_m\in\mathcal{V}_{\mathcal{M}}$,
the number of edges between $v_n$ and $v_m$ in $\Gamma$ is given by the number of length two pathes from $v_n$ to $v_m$ in $\Gamma'$.
The graph $\Gamma'$ is said to be the bipartite graph induced from the graph $\Gamma$.
The graph $\Gamma$ is said to be a refinement of the graph $\Gamma'$.
\end{definition}

For a factor $\mathcal{M}$ of type II$_1$, if $\mathcal{N}\subset\mathcal{P}\subset\mathcal{M}$ is a sequence of irreducible subfactors with finite index,
then $L^2(\mathcal{P})$ forms an irreducible $(\mathcal{N},\mathcal{P})$ bimodule, denoted by $X$,
and $L^2(\mathcal{M})$ forms an irreducible $(\mathcal{P},\mathcal{M})$ bimodule, denoted by $Y$.
Their conjugates $\overline{X}$, $\overline{Y}$ are $(\mathcal{P},\mathcal{N})$, $(\mathcal{P},\mathcal{M})$ bimodules respectively.
The tensor products
$X\otimes Y\otimes\overline{Y}\otimes\overline{X}\otimes\cdots\otimes \overline{X}$,
$X\otimes Y\otimes\overline{Y}\otimes\overline{X}\otimes\cdots\otimes X$,
$X\otimes Y\otimes\overline{Y}\otimes\overline{X}\otimes\cdots\otimes Y$,
$X\otimes Y\otimes\overline{Y}\otimes\overline{X}\otimes\cdots\otimes \overline{Y}$,
are decomposed into irreducible bimodules over $(\mathcal{N},\mathcal{N})$, $(\mathcal{N},\mathcal{P})$, $(\mathcal{N},\mathcal{M})$ and $(\mathcal{N},\mathcal{P})$ respectively.

\begin{definition}
The principal graph for the inclusion of factors $\mathcal{N}\subset\mathcal{P}\subset\mathcal{M}$ is an ($\mathcal{N},\mathcal{P},\mathcal{M}$) coloured graph. Its vertices are equivalent classes of irreducible bimodules over $(\mathcal{N},\mathcal{N})$, $(\mathcal{N},\mathcal{P})$ and $(\mathcal{N},\mathcal{M})$ in the above decomposed inclusion.
The number of edges connecting two vertices, a $(\mathcal{N},\mathcal{N})$ (or $(\mathcal{N},\mathcal{M})$) bimodule $U$ (or $V$) and a $(\mathcal{N},\mathcal{P})$ bimodule $W$, is the multiplicity of the equivalent class of $U$ (or $V$) as a sub bimodule of $W\otimes \overline{X}$ (or $W\otimes Y$).
The vertex corresponding to the irreducible $(\mathcal{N},\mathcal{N})$ bimodule $L^2(\mathcal{N})$ is marked by a star sign $*$.
The dimension vector of the pincipal graph is a function $\lambda$ from the vertices of the graph to $\mathbb{R}^+$. Its value at a point is defined to be the dimension of the corresponding bimodule.

Similarly the dual principal graph for the inclusion of factors is defined by considering the decomposed inclusion of $(\mathcal{M},\mathcal{M})$, $(\mathcal{M},\mathcal{P})$, $(\mathcal{M},\mathcal{N})$ bimodules.
\end{definition}



There is another principal graph given by decomposed inclusions of $(\mathcal{P},\mathcal{N})$, $(\mathcal{P},\mathcal{P})$ and $(\mathcal{P},\mathcal{M})$ bimodules, but we do not need it in this paper.

\begin{proposition}
The (dual) principal graph for the inclusion of factors $\mathcal{N}\subset\mathcal{P}\subset\mathcal{M}$ is a refinement of the (dual) principal graph of the subfactor $\mathcal{N}\subset\mathcal{M}$¡£
\end{proposition}

\begin{proof}
If follows from the definition and the fact that $X\otimes Y$ is the $(\mathcal{N},\mathcal{M})$ bimodule $L^2(\mathcal{M})$.
\end{proof}

Let $\delta_a$ be $\sqrt{[\mathcal{P}:\mathcal{N}]}$, the dimension of $X$,
and $\delta_b$ be $\sqrt{[\mathcal{M}:\mathcal{P}]}$, the dimension of $Y$.
Then by Frobenius reciprocity theorem, we have the following proposition.

\begin{proposition}\label{dab}
For the principal graph of factors $\mathcal{N}\subset\mathcal{P}\subset\mathcal{M}$ and the dimension vector $\lambda$, we have

$$\delta_a\lambda(u)=\sum_{\varepsilon\in \mathcal{E}_+,s(\varepsilon)=u} \lambda(t(\varepsilon)),~\forall u\in\mathcal{V}_\mathcal{N}; \quad \delta_b\lambda(w)=\sum_{\varepsilon\in \mathcal{E}_-,s(\varepsilon)=w} \lambda(t(\varepsilon)),~\forall w\in\mathcal{V}_\mathcal{M};$$

$$\delta_a\lambda(v)=\sum_{\varepsilon\in \mathcal{E}_{+}, t(\varepsilon)=v} \lambda(s(\varepsilon)),\quad
\delta_b\lambda(v)=\sum_{\varepsilon\in \mathcal{E}_{-}} \lambda(s(\varepsilon)),~\forall v\in\mathcal{V}_\mathcal{P};$$

\end{proposition}

\begin{definition}
For an $(\mathcal{N},\mathcal{P},\mathcal{M})$ coloured graph $\Gamma$, if there exits a function $\lambda: \mathcal{V}\rightarrow \mathbb{R}^+$ with the proposition mentioned above, then we call it a graph with parameter $(\delta_a,\delta_b)$.
\end{definition}

\begin{proposition}
The principal graph of factors $\mathcal{N}\subset\mathcal{P}\subset\mathcal{M}$ is a graph with parameter $(\sqrt{[\mathcal{P:N}]},\sqrt{[\mathcal{M:P}]})$.
Consequently if $\mathcal{N}\subset\mathcal{M}$ is finite depth, then the principal graph of $\mathcal{N}\subset\mathcal{P}\subset\mathcal{M}$ is finite.
\end{proposition}

\begin{proof}
The first statement follows from the definition.
Note that the dimension of a bimodule is at least 1. By this restriction, $\mathcal{N}\subset\mathcal{M}$ is finite depth implies the principal graph of $\mathcal{N}\subset\mathcal{P}\subset\mathcal{M}$ is finite.
\end{proof}


\subsection{The standard invariant}
We will define the refined (dual) principal graph for a subfactor planar algebra with a biprojection. This definition coincides with the definition given by bimodules, but we do not need this fact in this paper.
Given $\mathcal{N}\subset\mathcal{P}\subset\mathcal{M}$, there are two Jones projections $e_{\mathcal{N}}$ and $e_{\mathcal{P}}$ acting on $L^2(\mathcal{M})$.
Then we have the basic construction $\mathcal{N}\subset\mathcal{P}\subset\mathcal{M}\subset \mathcal{P}_1\subset \mathcal{M}_1$.
Repeating this process, we will obtain a sequence of factors $\mathcal{N}\subset\mathcal{P}\subset\mathcal{M}\subset \mathcal{P}_1\subset \mathcal{M}_1\subset\mathcal{P}_2\subset\mathcal{M}_2\cdots$ and a sequence of Jones projections $e_{\mathcal{N}}$, $e_{\mathcal{P}}$, $e_{\mathcal{M}}$, $e_{\mathcal{P}_1}\cdots$.
Then the standard invariant is refined as

$\begin{array}{ccccccccccc}
\mathbb{C}=\mathcal{N}'\cap\mathcal{N}&\subset&\mathcal{N}'\cap\mathcal{P}&\subset&\mathcal{N}'\cap\mathcal{M}&\subset&\mathcal{N}'\cap\mathcal{P}_1&\subset& \mathcal{N}'\cap\mathcal{M}_1&\subset&\cdots\\
&&\cup&&\cup&&\cup&&\\
&&\mathbb{C}=\mathcal{P}'\cap\mathcal{P}&\subset&\mathcal{P}'\cap\mathcal{M}&\subset&\mathcal{P}'\cap\mathcal{P}_1&\subset& \mathcal{P}'\cap\mathcal{M}_1&\subset&\cdots\\
&&&&\cup&&\cup&&\cup&&\\
&&&&\mathrm{C}=\mathcal{M}'\cap\mathcal{M}&\subset&\mathcal{M}'\cap\mathcal{P}_1&\subset&\mathcal{M}'\cap\mathcal{M}_1&\subset&\cdots\\
\end{array}$

For Fuss-Catalan, the corresponding Bratteli diagram is describe by the $middle patterns$, see page 114-115 in \cite{BisJonFC}.

We hope to define the refined principal graph as the limit of the Bratteli diagram $Br_{k}$ of $\mathcal{N}'\cap\mathcal{M}_{k-2} \subset \mathcal{N}'\cap\mathcal{P}_{k-1}\subset \mathcal{N}'\cap\mathcal{M}_{k-1}$.
To show the limit is well defined, we need to prove that $Br_{k}$ is identified as a subgraph of $Br_{k+1}$.
To define it for a subfactor planar algebra with a biprojection without the presumed factors, we need to do some translations motivated by the fact
$$\mathcal{N}'\cap\mathcal{P}_k=\mathcal{N}'\cap(\mathcal{M}_k \cap \{e_{\mathcal{P}_k}\}')=(\mathcal{N}'\cap\mathcal{M}_k) \cap \{e_{\mathcal{P}_k}\}'.$$

\begin{definition}
Let $\mathscr{S}=\mathscr{S}_{m,\pm} $ be a subfactor planar algebra. And $e_1,e_2,\cdots$ be the sequence of Jones projections.

Suppose $p_1$ is a biprojection in $\mathscr{S}_{2,+}$.
Then we will obtain another sequence of Jones projections $p_1,p_2,p_3,\cdots$, corresponding to the intermediate subfactors,
precisely $p_2$ in $\mathscr{S}_{2,-}\subset \mathscr{S}_{3,+}$ is a multiple of $\mathcal{F}(p_1)$, and $p_k$ is obtained by adding two strings on the left side of $p_{k-2}$.

For $m\geq 1$,
let us define  $\mathscr{S}'_{m,+}$ to be  $\mathscr{S}_{m,+} \cap \{p_{m}\}'$ and
$\mathscr{S}'_{m,-}$ to be  $\mathscr{S}_{m,-} \cap \{p_{m+1}\}'$.
\end{definition}

\begin{proposition}\label{equ def}
For $X\in\mathscr{S}_{m,+}$, $m\geq1$, we have
$$Xp_{m}=p_{m}X \iff \mathcal{F}(X)=\mathcal{F}(X)p_{m}.$$
\end{proposition}

That means $\mathscr{S}'_{m,+}$ is the invariant subspace of $\mathscr{S}_{m,+}$ under the ``right action" of the biprojection. Diagrammatically its consists of vectors with one $a/b$-colour through string on the rightmost.

\begin{proof}
If $p_mX=Xp_m$, then take the action given by the planar tangle $\graa{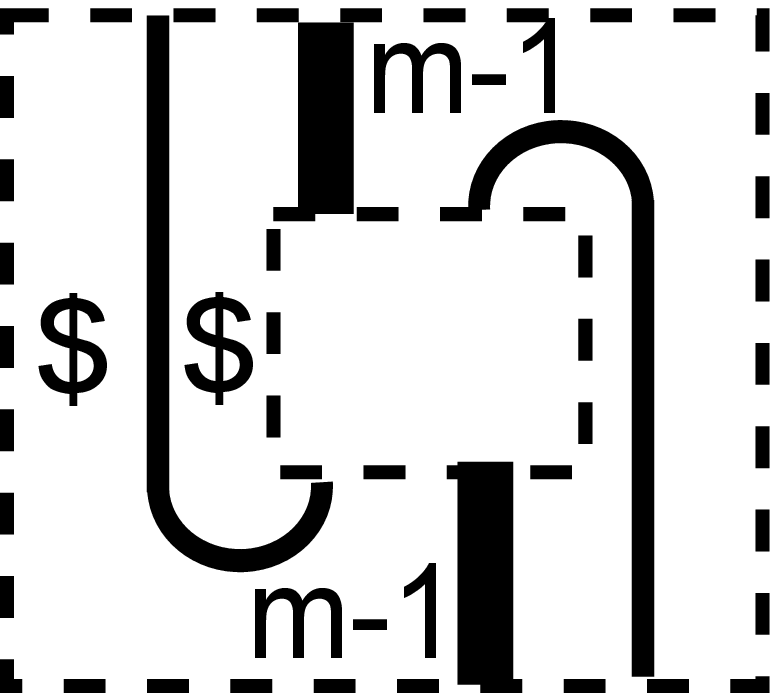}$, we have $\mathcal{F}(X)=\mathcal{F}(X)p_{m}$.

For $m$ odd,
if $\mathcal{F}(X)=\mathcal{F}(X)p_{m}$, then $X=X*\mathcal{F}(p_1)$, i.e.
$$X=\graa{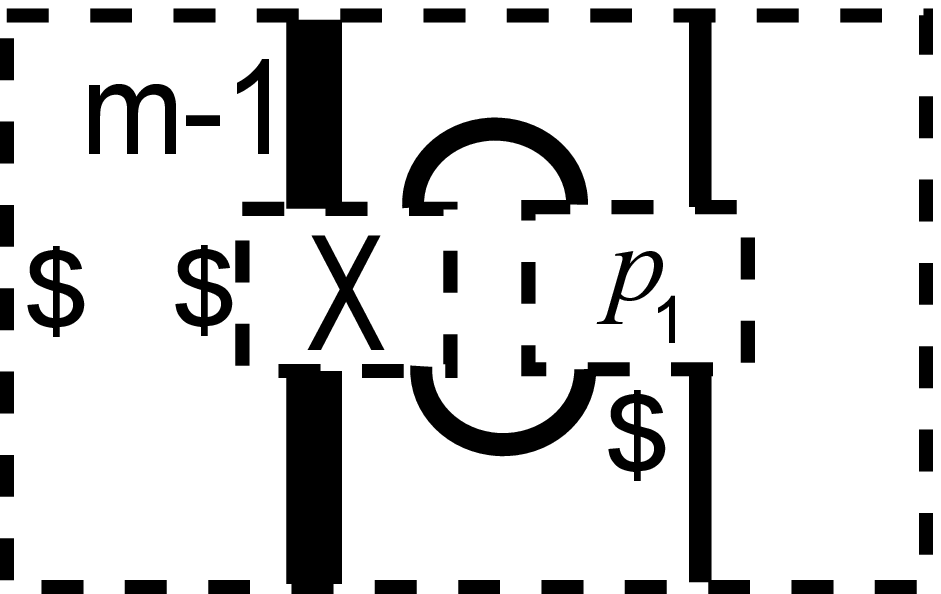}.$$
By the exchange relation of the biprojection, we have
$$\grb{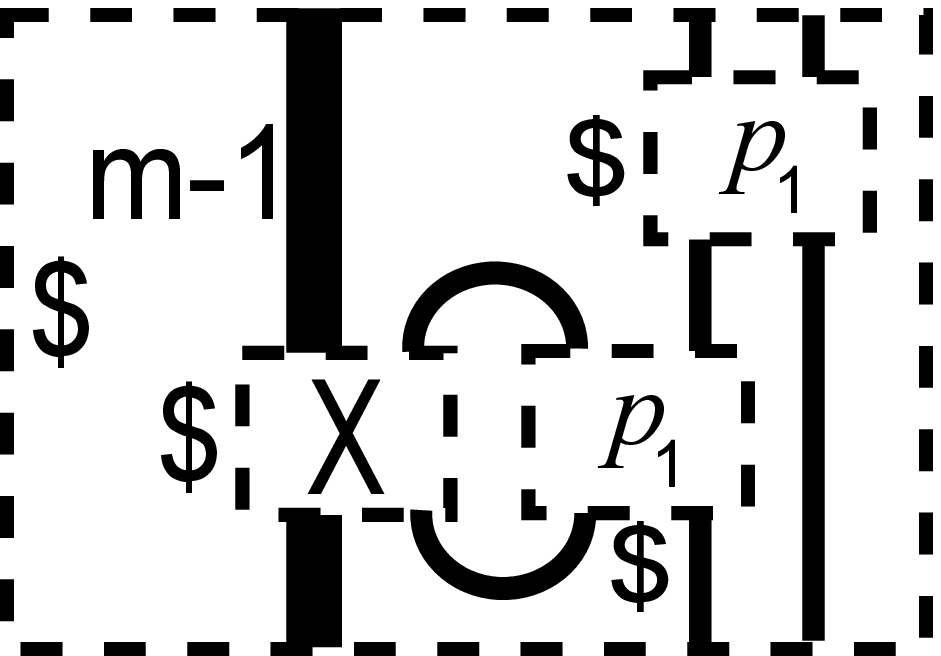}=\grb{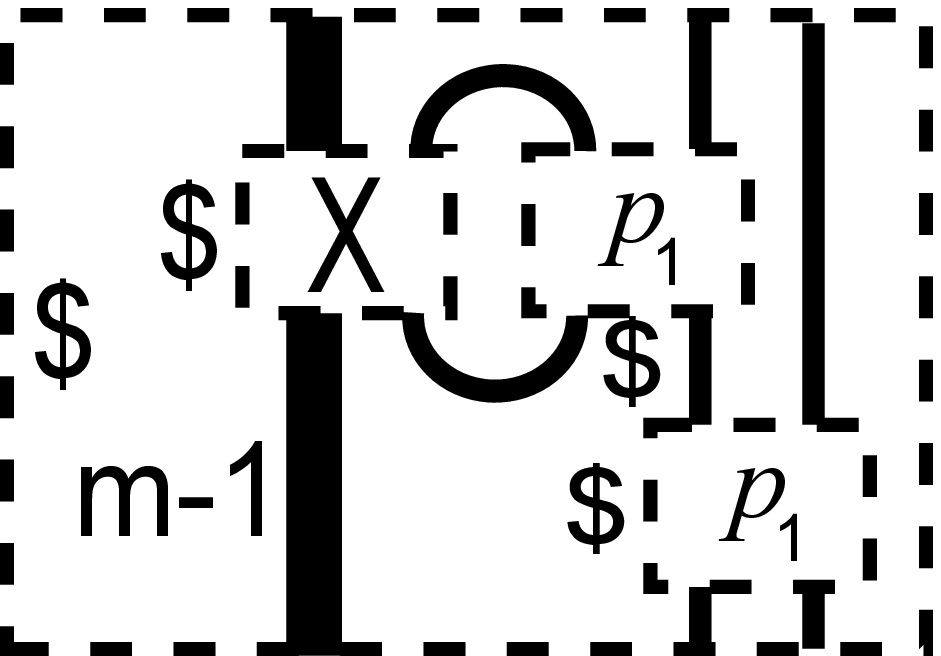}.$$
So $p_mX=Xp_m$.

For $m$ even, the proof is similar.
\end{proof}

Note that $\mathscr{S}_{m-1,+}$ is in the commutant of ${p_{m}}'$.
So we have the inclusion of finite dimensional von Neumann algebras
$$\mathscr{S}_{0,+}\subset\mathscr{S}_{1,+}'\subset\mathscr{S}_{1,+}\subset\mathscr{S}_{2,+}'\subset\mathscr{S}_{2,+}\subset\cdots.$$
Then we obtain the Bratteli diagram $Br_m$ for the inclusion $\mathscr{S}_{m-1,+}\subset\mathscr{S}_{m,+}'\subset\mathscr{S}_{m,+}$.
To take the limit of $Br_m$, we need to prove that $Br_m$ is identified as a subgraph of $Br_{m+1}$.


\begin{proposition}\label{equ pro}
If $P_1,~P_2$ are minimal projections of $\mathscr{S}_{m,+}'$. Then $P_1p_m,~P_2p_m$ are minimal projections of $\mathscr{S}_{m+1,+}'$.
Moreover $P_1$ and $P_2$ are equivalent in $\mathscr{S}_{m,+}'$ if and only if  $P_1p_m$ and $P_2p_m$ are equivalent in $\mathscr{S}_{m+1,+}'$.
\end{proposition}

This proposition is the same as Proposition \ref{frob e 1}.


\begin{proposition}[Frobenius Reciprocity]\label{Frob}
\mbox{}

(1) For a minimal projection $P\in\mathscr{S}_{m-1,+}$ and a minimal projection $Q\in\mathscr{S}_{m,+}'$, we have
$Qp_m$ is a minimal projection of $\mathscr{S}_{m+1,+}'$, $Pe_m$ is a minimal projection of $\mathscr{S}_{m+1,+}'$, and
$$\dim(P(\mathscr{S}_{m,+}')Q)=\dim(Pe_m(\mathscr{S}_{m+1,+})Qp_m).$$

(2) For a minimal projection $P'\in\mathscr{S}_{m,+}'$ and a minimal projection $Q'\in\mathscr{S}_{m,+}$, we have
$P'p_m$ is a minimal projection of $\mathscr{S}_{m+1,+}'$, and
$$\dim(P'(\mathscr{S}_{m,+})Q')=\dim(P'p_m(\mathscr{S}_{m+1,+}')Q').$$
\end{proposition}

\begin{proof}
(1) Consider the maps $$\phi_1=\graa{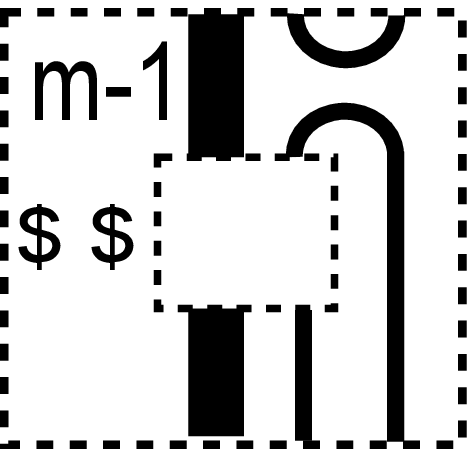}: \mathscr{S}_{m,+}\rightarrow \mathscr{S}_{m+1,+}, \quad \phi_2=\graa{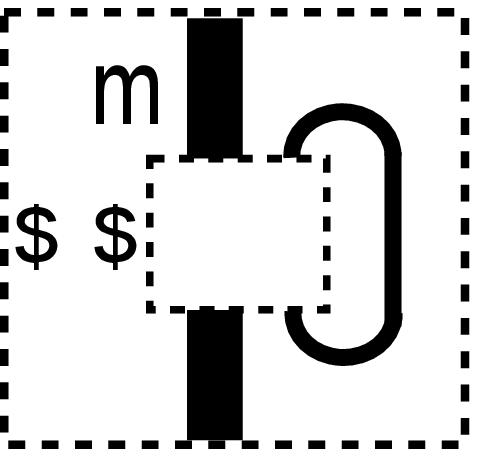} :\mathscr{S}_{m+1,+}\rightarrow \mathscr{S}_{m,+}.$$
For $m$ odd,
if $X\in P(\mathscr{S}_{m,+}')Q$, then by Proposition \ref{equ def}, we have $X=P(X'*\mathcal{F}(p_1))Q$ for some $X'\in \mathscr{S}_{m,+}$.
So $\phi_1(X)\in Pe_m(\mathscr{S}_{m+1,+})Qp_m$.
On the other hand, if $Y\in Pe_m(\mathscr{S}_{m+1,+})Qp_m$, then $\phi_2(Y)\in P(\mathscr{S}_{m,+}')Q$. While $\phi_1\circ\phi_2$ is the identity map on $Pe_m(\mathscr{S}_{m+1,+})Qp_m$ and $\phi_2\circ\phi_1$ is the identity map on $P(\mathscr{S}_{m,+}')Q$. So $\dim(P'(\mathscr{S}_{m,+})Q')=\dim(P'p_m(\mathscr{S}_{m+1,+}')Q')$.

For $m$ even, the proof is similar.

(2) This is the same as Proposition \ref{frob e 2}.
\end{proof}

By Proposition(\ref{frob e 1})(\ref{Frob}), the Bratteli diagram $Br_m$ is identified as a subgraph of $Br_{m+1}$.

\begin{definition}\label{refined principal graph of PA}
Let us define the refined principal graph of $\mathscr{S}$ with respect to the biprojection $p_1$ to be the limit of the Bratteli diagram of
$\mathscr{S}_{m,+}\subset\mathscr{S}_{m+1,+}'\subset\mathscr{S}_{m+1,+}$. The vertex corresponds to the identity in $\mathscr{S}_{0,+}$ is marked by a star sign.

Similarly let us define the refined dual principal graph of $\mathscr{S}$ with respect to the biprojection $p_1$ to be the limit of the Bratteli diagram of
$\mathscr{S}_{m,-}\subset\mathscr{S}_{m+1,-}'\subset\mathscr{S}_{m+1,-}$. The vertex corresponds to the identity in $\mathscr{S}_{0,-}$ is marked by a star sign.
\end{definition}

The refined principal graph is an $(\mathcal{N},\mathcal{P},\mathcal{M})$ coloured graph. The $\mathcal{N},\mathcal{P},\mathcal{M}$ coloured vertices are given by equivalence classes of minimal projections of
$\mathscr{S}_{2m,-},\mathscr{S}_{2m+1,-}',\mathscr{S}_{2m+1,-}$ respectively, for $m$ approaching infinity.
Similarly the refined dual principal graph is an $(\mathcal{M},\mathcal{P},\mathcal{N})$ coloured graph.

\begin{definition}
The dimension vector $\lambda$ of the principal graph is defined as follows,
for an $\mathcal{N}$ or $\mathcal{M}$ coloured vertex, its value is the Markov trace of the minimal projection corresponding to that vertex;
for a $\mathcal{P}$ coloured vertex $v$, suppose $Q\in \mathscr{S}_{m,+}'$ is a minimal projection corresponding to $v$.
Then $\lambda(v)=\delta_a^{-1}tr(Q)$, when $m$ is even, where $\delta_a=\sqrt{tr(p_1)}$;   $\lambda(v)=\delta_b^{-1}tr(Q)$, when $m$ is odd, where $\delta_b=\delta\delta_a^{-1}$.
\end{definition}

\begin{remark}
An element in $\mathscr{S}_{m,+}'$ has an $a/b$-colour through string on the rightmost. When we compute the dimension vector for a minimal projection in $\mathscr{S}_{m,+}'$, that string should be omitted. So there is a factor $\delta_a^{-1}$ or $\delta_b^{-1}$.
\end{remark}

Note that the dimension vector satisfies Proposition \ref{dab}. So the refined principal graph is a graph with parameter $(\delta_a,\delta_b)$.
If the Bratteli diagram of $\mathscr{S}_{m,+}\subset\mathscr{S}_{m+1,+}$ is the same as that of $\mathscr{S}_{m+1,+}\subset\mathscr{S}_{m+2,+}$, i.e. $\mathscr{S}$ has finite depth,
then $Br_{m+1}=Br_{m+2}$ by the restriction of the dimension vector.
Specifically the Bratteli diagram of $\mathscr{S}_{m+1,+}'\subset\mathscr{S}_{m+1,+}$ is the same as that of $\mathscr{S}_{m+1,+}\subset\mathscr{S}_{m+2,+}'$.
So $\mathscr{S}_{m+1}'\subset\mathscr{S}_{m+1,+}\subset\mathscr{S}_{m+2,+}'$ forms a basic construction, and $p_{m+1}$ is the Jones projection.
Then the Jones projection can be expressed as a linear sum of loops. We will see the formula later.


The subfactor planar algebra $FC(\sqrt{2},\frac{1+\sqrt{5}}{2})$ contains a trace-2 biprojection.
Considering the middle pattern of its minimal projections, we have its refined principal graph as
$$\gra{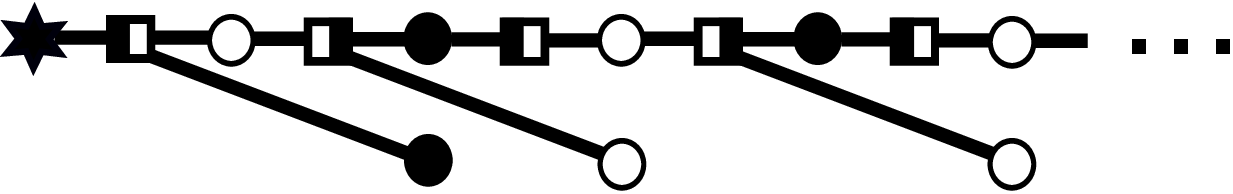};$$
and its refined dual principal graph as
$$\grb{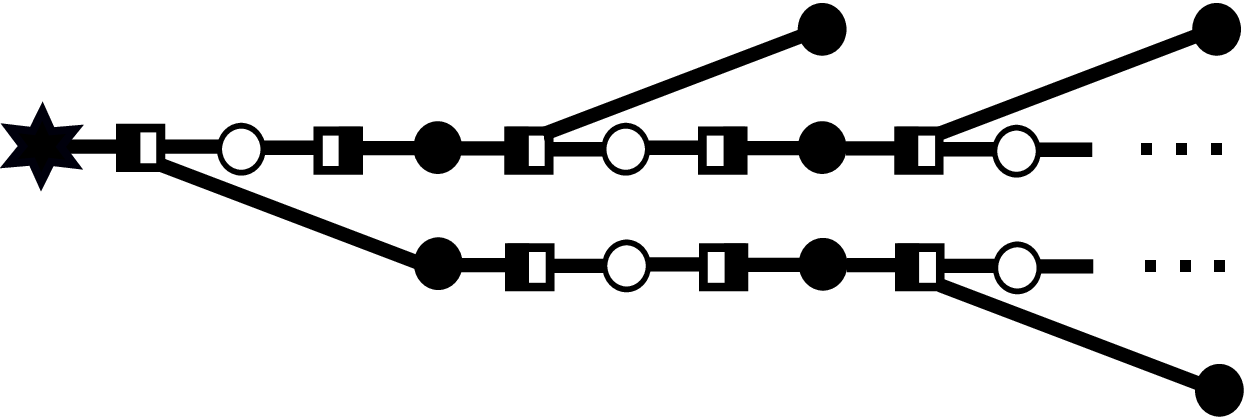},$$
where the black, mixed, white points are $\mathcal{N},\mathcal{P}, \mathcal{M}$ coloured vertices.

\subsection{Finite-dimensional inclusions}

Now given an inclusion of finite dimensional von Neumann algebras $\mathcal{B}_0\subset \mathcal{B}_1\subset \mathcal{B}_2$, similarly we may consider its Bratteli diagram, adjacent matrixes, Markov trace, and the basic construction.

\begin{definition}
The Bratteli diagram $Br$ for the inclusion $\mathcal{B}_0\subset \mathcal{B}_1\subset \mathcal{B}_2$ is a $(\mathcal{B}_0, \mathcal{B}_1, \mathcal{B}_2)$ coloured graph. Its $\mathcal{B}_i$ coloured vertices are indexed by the minimal central projections
(or equivalently the irreducible representations) of $\mathcal{B}_i$, for $i=0,1,2$. The subgraph of $Br$ consisting of $\mathcal{B}_0$, $\mathcal{B}_1$ coloured vertices and the edges connecting them is the same as the Bratteli diagram for the inclusion $\mathcal{B}_0\subset \mathcal{B}_1$.  The subgraph of $Br$ consisting of $\mathcal{B}_1$, $\mathcal{B}_2$ coloured vertices and the edges connecting them is the same as the Bratteli diagram for the inclusion $\mathcal{B}_1\subset \mathcal{B}_2$.
\end{definition}

Let $\Lambda$, $\Lambda_1$ and $\Lambda_2$ be the adjacent matrixes of $\mathcal{B}_0\subset \mathcal{B}_2$, $\mathcal{B}_0\subset \mathcal{B}_1$ and $\mathcal{B}_1\subset \mathcal{B}_2$ respectively. Then $\Lambda=\Lambda_1\Lambda_2$. Take a faithful tracial state $\tau$ on $\mathcal{B}_2$. Let $L^2(\mathcal{B}_2)$ be the Hilbert space given by the GNS construction with respect to $\tau$. Then $L^2(\mathcal{B}_0)$ and $L^2(\mathcal{B}_1)$ are naturally identified as subspaces of $L^2(\mathcal{B}_2)$. Let $e_1$, $p_1$ be the Jones projections onto the subspaces $L^2(\mathcal{B}_0)$, $L^2(\mathcal{B}_1)$ respectively. Then $\mathcal{B}_3=(\mathcal{B}_2 \cup {p_1})'', ~\mathcal{B}_4=(\mathcal{B}_2 \cup {e_1})''$ are obtained by the basic construction. So $Z(\mathcal{B}_0)=Z(\mathcal{B}_4),~ Z(\mathcal{B}_1)=Z(\mathcal{B}_3)$. And the adjacent matrixes of $\mathcal{B}_2\subset \mathcal{B}_3$, $\mathcal{B}_2\subset \mathcal{B}_4$ are $\Lambda_2^T$, $\Lambda^T$.

\begin{proposition}
The adjacent matrix of $\mathcal{B}_3\subset \mathcal{B}_4$ is $\Lambda_1^T$.
\end{proposition}

\begin{proof}
We assume that the adjacent matrix of $\mathcal{B}_3\subset \mathcal{B}_4$ is $\tilde{\Lambda}$.
Let $J$ denote the modular conjugation operator on $L^2(\mathcal{B}_0)$. Then
$z\rightarrow Jz*J$ is a *-isomorphism of $Z(\mathcal{B}_0)$ onto $Z(\mathcal{B}_4)$, of $Z(\mathcal{B}_1)$ onto $Z(\mathcal{B}_3)$.
Take a minimal central projection $x$ of $\mathcal{B}_0$ and a minimal central projection $y$ of $\mathcal{B}_1$, we have $\tilde{x}=JxJ$ is a minimal central projection of $\mathcal{B}_4$, and $\tilde{y}=JyJ$ is a minimal central projection of $\mathcal{B}_3$.
The definition of the adjacent matrix implies that
$$\Lambda_{y,x}=[\dim(xy\mathcal{B}_0'xy\cap xy\mathcal{B}_1xy)]^{\frac{1}{2}};$$
$$\tilde{\Lambda}_{\tilde{x},\tilde{y}}=[\dim(\tilde{x}\tilde{y}\mathcal{B}_3'\tilde{x}\tilde{y}\cap \tilde{x}\tilde{y}\mathcal{B}_4\tilde{x}\tilde{y})]^{\frac{1}{2}}.$$
Note that
$$\tilde{x}\tilde{y}\mathcal{B}_3'\tilde{x}\tilde{y}\cap \tilde{x}\tilde{y}\mathcal{B}_4\tilde{x}\tilde{y}=JxyJ\mathcal{B}_3'JxyJ\cap JxyJ\mathcal{B}_4JxyJ
=J(xy\mathcal{B}_0'xy\cap xy\mathcal{B}_1xy)J.$$
So $\tilde{\Lambda}_{\tilde{x},\tilde{y}}=\Lambda_{y,x}=\Lambda^T_{x,y}$.
\end{proof}

\begin{definition}
We say $\tau$ is a Markov trace for the inclusion $\mathcal{B}_0\subset\mathcal{B}_1\subset\mathcal{B}_2$, if $\tau$ is a Markov trace for the inclusions $\mathcal{B}_0\subset\mathcal{B}_1$ and $\mathcal{B}_1\subset\mathcal{B}_2$.
\end{definition}

\begin{proposition}
If $\tau$ is a Markov trace for the inclusion $\mathcal{B}_0\subset\mathcal{B}_1\subset\mathcal{B}_2$, then $\tau$ is a Markov trace for the inclusion $\mathcal{B}_0\subset\mathcal{B}_2$.
Moreover $\tau$ extends uniquely to a Markov trace for the inclusion $\mathcal{B}_2\subset\mathcal{B}_3\subset\mathcal{B}_4$.
\end{proposition}

\begin{proof}
Let $\lambda_i=\lambda_{\mathcal{B}_i}^\tau$ be the dimension vectors for $i=0,1,2$.
If $\tau$ is a Markov trace for the inclusion $\mathcal{B}_0\subset\mathcal{B}_1\subset\mathcal{B}_2$, then
by the definition $\tau$ is a Markov trace for the inclusions $\mathcal{B}_0\subset\mathcal{B}_1$ and $\mathcal{B}_1\subset\mathcal{B}_2$.
So $\Lambda_2\lambda_2=\lambda_1$; $\Lambda_1\lambda_1=\lambda_0$; $\Lambda_1^T\lambda_0=||\Lambda_1||^2\lambda_1$; and $\Lambda_2^T\lambda_1=||\Lambda_2||^2\lambda_2$.
Then
$\Lambda^T\Lambda\lambda_2
=\Lambda_2^T\Lambda_1^T\Lambda_1\Lambda_2\lambda_2
=||\Lambda_1||^2||\Lambda_2||^2\lambda^2$.
So $\tau$ is a Markov trace for the inclusion $\mathcal{B}_0\subset\mathcal{B}_2$ and $||\Lambda||=||\Lambda_1||\dot||\lambda_2||$. Then $\tau$ extends uniquely to a Markov trace for the inclusion $\mathcal{B}_2\subset\mathcal{B}_4$.
Let $\lambda_i=\lambda_{\mathcal{B}_i}^\tau$ be the dimension vectors for $i=3,4$.
We have $\lambda_4=||\Lambda||^{-2}\lambda_0$ by the uniqueness of the extension of $\tau$. And $\lambda_3=\Lambda_1^T\lambda_4=||\Lambda||^{-2}\Lambda_1^T\lambda_0=||\Lambda_2||^{-2}\lambda_1$. Then by a direct computation $\Lambda_1\Lambda_1^T\lambda_4=||\Lambda_1||^{2}\lambda_4$ and $\Lambda_2\Lambda_2^T\lambda_3=||\Lambda_2||^{2}\lambda_3$. That means $\tau$ extends to a Markov trace for the inclusion $\mathcal{B}_2\subset\mathcal{B}_3\subset\mathcal{B}_4$.

On the other hand, if $\tau$ extends to a Markov trace for the inclusion $\mathcal{B}_2\subset\mathcal{B}_3\subset\mathcal{B}_4$, then it also extends to a Markov trace for the inclusion $\mathcal{B}_0\subset\mathcal{B}_2$. That implies the uniqueness of such an extension.
\end{proof}

\begin{definition}
Given the Bratteli diagram $Br$ for the inclusion $\mathcal{B}_0\subset \mathcal{B}_1\subset \mathcal{B}_2$,
let us define the dimension vector with respect to the Markov trace $\tau$ to be $\lambda^\tau$, a function from the vertices of the Bratteli diagram the into $\mathbb{R}^+$, as follows
for a $\mathcal{B}_0$ coloured vertex, its value is the trace of the minimal projection corresponding to that vertex;
for a $\mathcal{B}_1$ coloured vertex, its value is $||\Lambda_1||$ times the trace of the minimal projection corresponding to that vertex;
for a $\mathcal{B}_1$ coloured vertex, its value is $||\Lambda||$ times the trace of the minimal projection corresponding to that vertex.
\end{definition}

\begin{proposition}\label{unique markov trace}
The inclusion $\mathcal{B}_0\subset\mathcal{B}_1\subset\mathcal{B}_2$ admits a Markov trace if and only if the Bratteli diagram for the inclusion is a graph with parameter $(\delta_a,\delta_b)$. In this case $\delta_a=||\Lambda_1||$ and $\delta_b=||\Lambda_2||$. Under this condition, the Markov trace is unique if and only if the Bratteli diagram is connected.
\end{proposition}

\begin{proof}
The first statement follows from the definitions.

In this case, $\delta_a=||\Lambda_1||$ and $\delta_b=||\Lambda_2||$ follows from the fact that the eigenvalue of $\Lambda_i^T\Lambda_i$ with a positive eigenvector has to be $||\Lambda_i||^2$.

Suppose the inclusion $\mathcal{B}_0\subset\mathcal{B}_1\subset\mathcal{B}_2$ admits a Markov trace.
If the bratteli diagram $Br$ is not connected, then we may adjust the proportion to obtain different Markov traces.
If the bratteli diagram $Br$ for the inclusion $\mathcal{B}_0\subset\mathcal{B}_1\subset\mathcal{B}_2$ is connected, we want to show that the bratteli diagram $Br'$ for the inclusion $\mathcal{B}_0\subset\mathcal{B}_2$ is connected. Actually if two $\mathcal{B}_0$ (or $\mathcal{B}_2$) coloured vertices are adjacent to the same $\mathcal{B}_1$ coloured vertex in $Br$. then they are adjacent to the same $\mathcal{B}_2$ (or $\mathcal{B}_0$) coloured vertex in $Br'$, because any $\mathcal{B}_1$ coloured point is adjacent to a $\mathcal{B}_2$ (or $\mathcal{B}_0$) coloured vertex in $Br$. While the bratteli diagram $Br'$ is connected implies the uniquness of the Markov trace for the inclusion $\mathcal{B}_0\subset\mathcal{B}_2$. Then the dimension vectors $\lambda_0$ and $\lambda_2$ are unique. So $\lambda_1$ is also unique. That means the Markov trace for the inclusion $\mathcal{B}_0\subset\mathcal{B}_1\subset\mathcal{B}_2$ is unique.

\end{proof}

\begin{corollary}\label{unique dimension vector}
Given the principal graph for the inclusion $\mathcal{N}\subset\mathcal{P}\subset\mathcal{M}$, its dimension vector is uniquely determined by the graph.
\end{corollary}

\begin{proof}
The dimension vector is a multiple of the dimension vector $\lambda^\tau$ with respect to the unique Markov trace $\tau$. While the value of the marked point is 1, so the dimension vector is unique.
\end{proof}

Now we may repeat the basic construction to obtain the Jones tower $\mathcal{B}_0\subset\mathcal{B}_1\subset\mathcal{B}_2\subset\mathcal{B}_3\subset\mathcal{B}_4\subset\cdots$ and a sequence of Jones projections $e_1,p_1,e_2,p_2\cdots$.

\begin{proposition}\label{realtion of e p}
The algebra generated by the sequences of projections $\{e_i\}$ and $\{p_j\}$ forms a Fuss-Catalan subfactor planar algebra.
\end{proposition}

This proposition is essentially the same as Proposition 5.1 in \cite{BisJonFC}.
In that case the Jones projections are derived from the inclusion of factors.
The proof is similar. We only need a fact that the trace preserving conditional expectation induced by a Markov trace maps the Jones projections to a multiple of the identity.

\subsection{Graph planar algebras and the embedding theorem}
Given a connected three $(\mathcal{N},\mathcal{P},\mathcal{M})$ coloured graph $\Gamma$ with parameter $(\delta_a,\delta_b)$, we have $\mathcal{V}_{N}$, $\mathcal{V}_{P}$, $\mathcal{V}_{M}$, $\mathcal{E}_{\pm}$, $s$, $t$, $*$ as in Definition \ref{def three}. Let $\lambda$ be the (unique) dimension vector. Let $\Gamma'$ be the bipartite graph induced from $\Gamma$.
Suppose the Bratteli diagram for the inclusion of finite dimensional von Neumann algebras $\mathcal{B}_0\subset\mathcal{B}_1\subset\mathcal{B}_2$ is $\Gamma$. Then the Bratteli diagram for the inclusion of $\mathcal{B}_0\subset\mathcal{B}_2$ is $\Gamma'$. Let $\Lambda_2$ be the adjacent matrix for $\mathcal{B}_1\subset\mathcal{B}_2$.
Applying the basic construction, we will obtain the tower $\mathcal{B}_0\subset\mathcal{B}_1\subset\mathcal{B}_2\subset\mathcal{B}_3\subset\mathcal{B}_4\subset\cdots$.
Let $\{e_i\},\{p_i\}$ be the sequences of Jones projections arising from the basic construction.
Note that the relative commutant of $\mathcal{B}_0$ in the tower can be expressed as linear sums of loops of $\Gamma$. While the even parts of the relative commutant is exactly the graph planar algebra $\mathscr{G}$ of $\Gamma'$. So an element in $\mathscr{G}$ could be expressed as a linear sums of loops of $\Gamma$, instead of loops of $\Gamma'$. Actually an edge of $\Gamma'$ is replaced by a length 2 path $\varepsilon_1\varepsilon_2^*$. It is convenient to express $p_1$ by loops of $\Gamma$.

\begin{proposition}\label{ploop}
Note that $p_1\in \mathcal{B}_1'\cap \mathcal{B}_3$, we have
$$ p_1=\delta_b^{-1}\sum_{\varepsilon_3,\varepsilon_7\in\mathcal{E}_{-},t(\varepsilon_3)=t(\varepsilon_7)} \sqrt{\frac{\lambda(s(\varepsilon_3))\lambda(s(\varepsilon_7))}
{\lambda(t(\varepsilon_3))\lambda(t(\varepsilon_7))}} [\varepsilon_3^*\varepsilon_3\varepsilon_7^*\varepsilon_7].$$

To express $p_1$ as an element in $\mathscr{G}_{2,+}=\mathcal{B}_0'\cap \mathcal{B}_4$, we have
$$ p_1=\delta_b^{-1}\sum_{\substack{\varepsilon_3,\varepsilon_7\in\mathcal{E}_{-}\\
\varepsilon_1,\varepsilon_5\in\mathcal{E}_{+}\\
t(\varepsilon_1)=t(\varepsilon_3)=t(\varepsilon_5)=t(\varepsilon_7)}}
\sqrt{\frac{\lambda(s(\varepsilon_3))\lambda(s(\varepsilon_7))}
{\lambda(t(\varepsilon_3))\lambda(t(\varepsilon_7))}}
[\varepsilon_1\varepsilon_3^*\varepsilon_3\varepsilon_5^*\varepsilon_5\varepsilon_7^*\varepsilon_7\varepsilon_1^*].$$
\end{proposition}

\begin{proof}
Note that $p_1$ is the Jones projection for the basic construction $\mathcal{B}_1\subset\mathcal{B}_2\subset\mathcal{B}_3$. So we have the first formula. Take the inclusion from $\mathcal{B}_1' \cap \mathcal{B}_3$ to $\mathcal{B}_0' \cap \mathcal{B}_4$ for $p_1$, we obtained the second formula.
\end{proof}

\begin{theorem}\label{embedding P}
Suppose $\mathscr{S}$ is a finite depth subfactor planar algebra, $p$ is a biprojection in $\mathscr{S}_{2,+}$, $\Gamma'$ is the principal graph of $\mathscr{S}$, and $\Gamma$ is the refined principal graph with respect to the biprojection $p$.
Let $\phi$ the embedding map from $\mathscr{S}$ to the graph planar algebra $\mathscr{G}$.
Then $\phi(p)=p_1$ is a linear some of loops as in Proposition \ref{ploop}.
\end{theorem}

\begin{proof}
Note that $p_m$ is the Jones projection for the basic construction $\mathscr{S}_m'\subset\mathscr{S}_m\subset\mathscr{S}_{m+1}'$, when $m$ is odd and greater than the depth of $\mathscr{S}$.
So $\phi(p)$ is the Jones projection for the basic construction $\mathcal{B}_1\subset\mathcal{B}_2 \subset\mathcal{B}_3$, which implies $\phi(p)=p_1$.
\end{proof}

\section{Bisch-Haagerup fish graphs}
The following result is proved by Bisch and Haagerup.
\begin{theorem}
Suppose $\mathcal{N}\subset\mathcal{P}\subset\mathcal{M}$ is an inclusion of factors of type II$_1$, such that $[\mathcal{M}:\mathcal{P}]=\frac{3+\sqrt{5}}{2}$ and $[\mathcal{P}:\mathcal{N}]=2$.
Then either it is a free composed inclusion, or
the principal graph of the subfactor $\mathcal{N}\subset\mathcal{M}$ is
$$\gra{principalgraph8},$$
called the $n_{th}$ Bisch-Haagerup fish graph, when it is of depth $2n+1$.
\end{theorem}

It follows from computing the relation of $(\mathcal{P},\mathcal{P})$ bimodules arisen from the two subfactors $\mathcal{N}\subset\mathcal{P}$ and $\mathcal{P}\subset\mathcal{M}$.

\begin{remark}
It is a free composed inclusion means there is no extra relation between $(\mathcal{P},\mathcal{P})$ bimodules.
In this case, the planar algebra of $\mathcal{N}\subset\mathcal{M}$ is Fuss-Catalan.
\end{remark}

By the embedding theorem, if the principal graph of a subfactor planar algebra is the $n_{th}$ Bisch-Haagerup fish graph, then the subfactor planar algebra is embedded in the graph planar algebra.
Because of the existence of a normalizer in the Bisch-Haagerup fish graph, the planar algebra contains a trace-2 biprojection.
First we will see there is only one possible refined principal graph with respect to the biprojection.
Then in the orthogonal complement of the Fuss-Catalan planar subalgebra, there is a new generator at depth $2n$.
We will show that this generator satisfies some relations.
We hope to solve the generator with such relations in the graph planar algebra.
In the case $n\geq4$, there is no solution. So there is no subfactor planar algebra whose principal graph is the $n_{th}$ fish.
In the case $n=1,2,3$, there is a unique solution up to (planar algebra) isomorphism. So there is at most one subfactor planar algebra for each $n$.
Their existence follows from three known subfactors.

\begin{notation}
Take $\delta_a=\sqrt{2}$, $\delta_b=\frac{1+\sqrt{5}}{2}$, and $\delta=\delta_a\delta_b$. Then $\delta_b^2=\delta_b+1$.
Let $FC=FC(\delta_a,\delta_b)$ be the Fuss-Catalan planar algebra with parameters $(\delta_a,\delta_b)$.
We assume that $f_{2n}$ is the minimal projection in $FC_{2n,+}$ with middle pattern $\underbrace{abba~abba~\cdots~abba}_n$, $n$ copies of $abba$;
and $g_{2n}$ is the minimal projection in $FC_{2n,-}$ with middle pattern $\underbrace{baab~baab~\cdots~baab}_n$.
\end{notation}

\subsection{Principal graphs}
If the $n_{th}$ Bisch-Haagerup fish graph is the principal graph of a subfactor $\mathcal{N}\subset\mathcal{M}$, then
its index is $\delta^2=3+\sqrt{5}$. Because of the existence of a ``normalizer", there is an intermediate subfactor $\mathcal{P}$, such that $[\mathcal{P}:\mathcal{N}]=2$.
\begin{definition}
Let us define the subfactor planar algebra of $\mathcal{N}\subset\mathcal{M}$ to be $\mathscr{B}=\{\mathscr{B}_{m,\pm}\}$, and $e_\mathcal{P}$ to be the biprojection corresponding to the intermediate subfactor $\mathcal{P}$.
\end{definition}


\begin{lemma}\label{refined principal graph}
The refined principal graph with respect to the biprojection $e_\mathcal{P}$ is
$$\grb{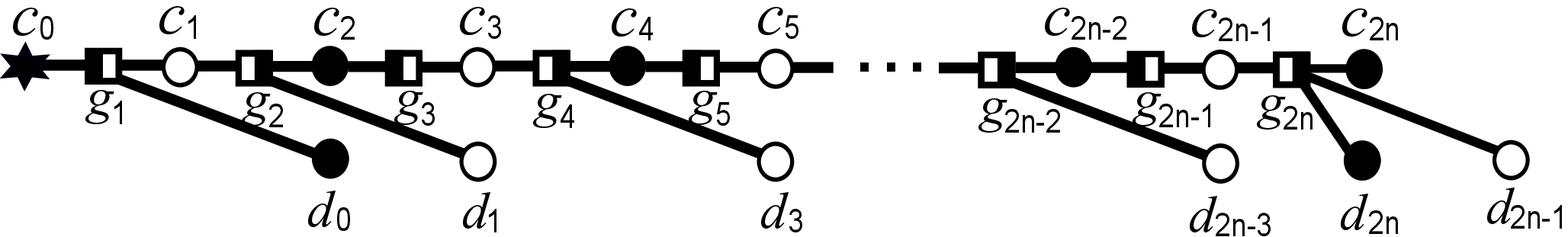}.$$
Its dimension vector $\lambda$ is given by

$\lambda(c_{2k-1})=\delta_a\delta_b^k$, for $1\leq k\leq n$;

$\lambda(d_{2k-1})=\delta_a\delta_b^{k-1}$, for $1\leq k\leq n$;

$\lambda(c_{2k})=2\delta_b^k$, for $1\leq k\leq n-1$;

$\lambda(c_0)=\lambda(d_0)=1$; $\lambda(c_{2n})=\lambda(d_{2n})=\delta_b^{n}$;

$\lambda(g_{2k-1})=\delta_a\delta_b^{k-1}$, for $1\leq k\leq n$;

$\lambda(g_{2k})=\delta_a\delta_b^{k}$, for $1\leq k\leq n$.
\end{lemma}

\begin{proof}
Note that $\delta^2=3+\sqrt{5}=\delta_a^2\delta_b^2$, so
the planar subalgebra generated by the trace-2 biprojection $e_\mathcal{P}$ is $FC=FC(\delta_a,\delta_b)$.
Observe that the principal graph of $FC$ is the same as the $n_{th}$ fish up to depth $2n-1$, so $\mathscr{B}_{2(n-1),+}=FC_{2(n-1),+}$.
Then the refined principal graph of $\mathscr{B}$ starts as
$$\grb{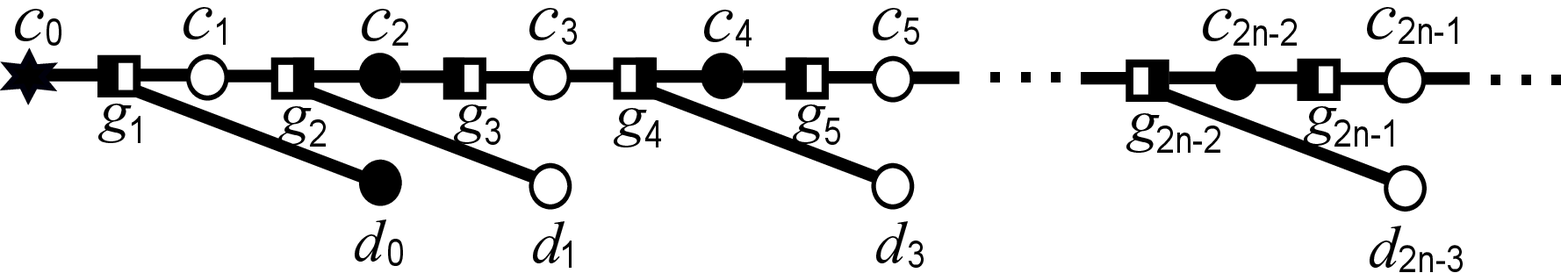}.$$

The vertex $c_{2k-1}$ corresponds to the minimal projection of $FC_{2k-1,+}$ with middle pattern
\\
$\underbrace{abba~\cdots~abba}_{k-1}ab$, $k-1$ copies of $abba$, for $1\leq k\leq n$.
So $\lambda(c_{2k-1})=\delta_a\delta_b^k$.

The vertex $d_{2k-1}$ corresponds to the minimal projection of $FC_{2k+1,+}$ with middle pattern
$\underbrace{abba~\cdots~abba}_{k-1}abbb$, for $1\leq k\leq n-1$.
So $\lambda(d_{2k-1})=\delta_a\delta_b^{k-1}$.

The vertex $c_{2k}$ corresponds to the minimal projection of $FC_{2k,+}$ with middle pattern
\\
$\underbrace{abba~\cdots~abba}_{k}$, for $1\leq k\leq n-1$.
So $\lambda(c_{2k})=2\delta_b^k$;

The vertex $c_0$ is the marked point. So $\lambda(c_0)=1$; The vertex $d_0$ corresponds to the minimal projection of $FC_{2,+}$ with middle pattern $aa$. So $\lambda(d_0)=1$;

The vertex $g_{2k-1}$ corresponds to the minimal projection of $FC'_{2k-1,+}$ with middle pattern
\\
$\underbrace{abba~\cdots~abba}_{k-1}~a$, for $1\leq k\leq n$.
So $\lambda(g_{2k-1})=\delta_a\delta_b^{k-1}$;

The vertex $g_{2k}$ corresponds to the minimal projection of $FC'_{2k,+}$ with middle pattern
\\
$\underbrace{abba~\cdots~abba}_{k-1}~abb$, for $1\leq k\leq n-1$.
So $\lambda(g_{2k})=\delta_a\delta_b^{k}$.

All these vertices are not adjacent to a new point in the refined principal graph except $c_{2n-1}$, because they are identical to the vertices of the refined principal graph of $FC$.

Note that $\delta_b\lambda(c_{2n-1})-\lambda(g_{2n-1})=\delta_a\delta_b^{n+1}-\delta_a\delta_b^{n-1}=\delta_a\delta_b^{n}$. So there is a new $\mathcal{P}$ coloured vertex, denoted by $g_{2n}$, adjacent to $c_{2n-1}$. Then $\lambda(g_{2n})\leq \delta_a\delta_b^{n}$. On the other hand $\lambda(g_{2n})\geq\delta_b^{-1}\lambda(c_{2n-1})=\delta_b^n>\frac{1}{2}\delta_a\delta_b^{n}$. So $g_{2n}$ is unique new $\mathcal{P}$ coloured vertex adjacent to $c_{2n-1}$ and $\lambda(g_{2n})=\delta_a\delta_b^n$.

While $\delta_b\lambda(g_{2n})-\lambda(c_{2n-1})=\delta_a\delta_b^{n+1}-\delta_a\delta_b^{n}=\delta_a\delta_b^{n-1}$, so there is a new $\mathcal{N}$ coloured vertex, denoted by $d_{2n-1}$, adjacent to $g_{2n}$. Then $\lambda(d_{2n-1})\leq \delta_a\delta_b^{n-1}$. On the other hand $\lambda(d_{2n-1})\geq\delta_b^{-1}\lambda(g_{2n})=\delta_a\delta_b^{n-1}$. So $d_{2n-1}$ is unique new $\mathcal{N}$ coloured vertex adjacent to $g_{2n}$ and $\lambda(d_{2n-1})=\delta_a\delta_b^{n-1}$.

Now $\delta_b\lambda(d_{2n-1})=\lambda(g_{2n})$, so there is no new $\mathcal{P}$ coloured vertex adjacent to $d_{2n-1}$.

In the principal graph, there are two $\mathcal{M}$ coloured vertices, denoted by $c_{2n},d_{2n}$, adjacent to $c_{2n-1}$. Thus $c_{2n},d_{2n}$ are adjacent to $g_{2n}$ in the refined principal graph. Moreover $\lambda(c_{2n})=\lambda(d_{2n})=\frac{1}{\delta}(\lambda(c_{2n-1})+\lambda(d_{2n-1}))=\delta_b^{n}$.
Then $\delta_a\lambda(c_{2n})=\delta_a\lambda(d_{2n})=\lambda(g_{2n})$. So there is no new $\mathcal{P}$ coloured vertices adjacent to $c_{2n}$ or $d_{2n}$.

Therefore we have the unique possible refined principal graph and its dimension vector as mentioned in the statement.
\end{proof}

Because $\mathscr{B}$ contains a biprojection, it is decomposed as an $Annular$ $Fuss-Catalan$ $module$ \cite{LiuAFC}, similar to the Temperley-Lieb case \cite{Jonann,JonRez}. The Fuss-Catalan planar subalgebra $FC$ is already a submodule of $\mathscr{B}$. There is a lowest weight vector in $\mathscr{B}_{2n,+}$ which is orthogonal to $FC$. So this vector is rotation invariant up to a phase. Moreover it is $totally$ $uncappable$, see \cite{LiuAFC}. In this special case, we have a direct proof of this result.

\begin{definition}
An element $x\in\mathscr{B}_{m,+}$ is said to be totally uncappable, if
$$\rho^k(x)P=0, ~\rho^k(\mathcal{F}(x))\mathcal{F}(P)=0, \quad \forall k\geq0;$$
An element $y\in\mathscr{B}_{m,-}$ is said to be totally uncappable,
if $\mathcal{F}(y)$ is totally uncappable.
\end{definition}

If we consider $P$ as an a,b-colour diagram,
then an element is totally uncappable means it becomes zero whenever it is capped by an a/b-colour string.

Now let us construct the totally uncappble element $S\in\mathscr{B}_{2n,+}$. If $S$ is totally uncappable, then $S$ is orthogonal to $FC_{2n,+}$. While the minimal projection $f_{2n}$ of $FC_{2n,+}$ is separated into two minimal projections in $\mathscr{B}_{2n,+}$, denoted by $P_c$ $P_d$, with fair trace. So $S$ has to be a multiple of $P_c-P_d$.
Take $S$ to be $P_c-P_d$, then $S$ satisfies the following propositions.

\begin{proposition}\label{relation of S}
For $S=P_c-P_d$ in $\mathscr{B}_{2n,+}$,  we have

(1) $S^*=S$;

(2) $S^2=f_{2n}$;

(3) $S$ is totally uncappable;

(4) $\rho(S)=\omega S$, for some $\omega\in\mathbb{C}$ satisfying $|\omega|=1$.
\end{proposition}

\begin{proof}
(1) $S^*=(P_c-P_d)^*=S$.

(2) $S^2=(P_c-P_d)^2=P_c+P_d=f_{2n}$.

(4) Note that $\rho$ preserves the inner product of $S\in\mathscr{B}_{2n,+}$, and $FC_{2n,+}$ is rotation invariant, so both $S$ and $\rho(S)$ are in the orthogonal complement of $FC_{2n,+}$ which is a one-dimensional subspace. Then we have $\rho(S)=\omega S$ for some $\omega\in\mathbb{C}$. Moreover $||\rho(S)||_2=||S||_2$, so $|\omega|=1$.

(3) From the refined principal graph, we have $S*P$ is a multiple of $f_{2n}$. By computing the trace, we have $S*P=0$. On the other hand $tr((SP)^*(SP))=tr(f_{2n}P)=0$, so $SP=0$. By proposition(4), we have $S$ is totally uncappable.
\end{proof}

If $S\in\mathscr{B}_{2n,+}$ is totally uncappable, then $\mathcal{F}(S)\in\mathscr{B}_{2n,-}$ is also totally uncappable. To describe its relations, we need the dual principal graph of $\mathscr{B}$.

\begin{lemma}
If the principal graph of $\mathscr{B}$ is the $n_{th}$ Bisch-Haagerup fish graph, then
the dual principal graph of $\mathscr{B}$ is
$$\grb{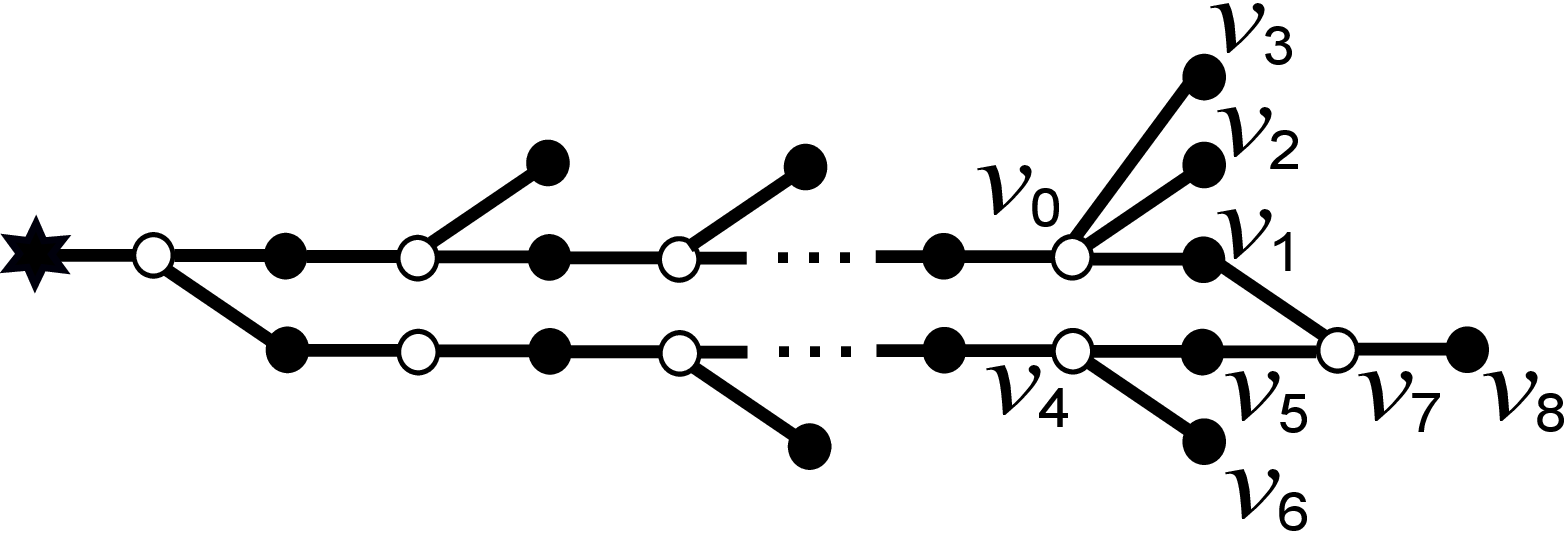}.$$

For its dimension vector $\lambda'$, we have $\lambda'(v_1)=\delta_b^n,~\lambda'(v_2)=\delta_b^{n-1}$.
\end{lemma}

\begin{proof}
Note that $\mathscr{B}_{2n-1,+}=FC_{2n-1,+}$, so $\mathscr{B}_{2n-1,-}=FC_{2n-1,-}$. Then the dual principal graph of $\mathscr{B}$ is the same as the dual principal graph of $FC$ up to depth $2n-1$. In $\mathscr{B}_{2n,-}$, there is a totally uncappable element, so the minimal projection $g_{2n}$ of $FC_{2n,-}$ is separated into two minimal projections of $\mathscr{B}_{2n,-}$, denoted by $P_c',~P_d'$. Then we have the dual principal graph up to depth $2n$ as
$$\grb{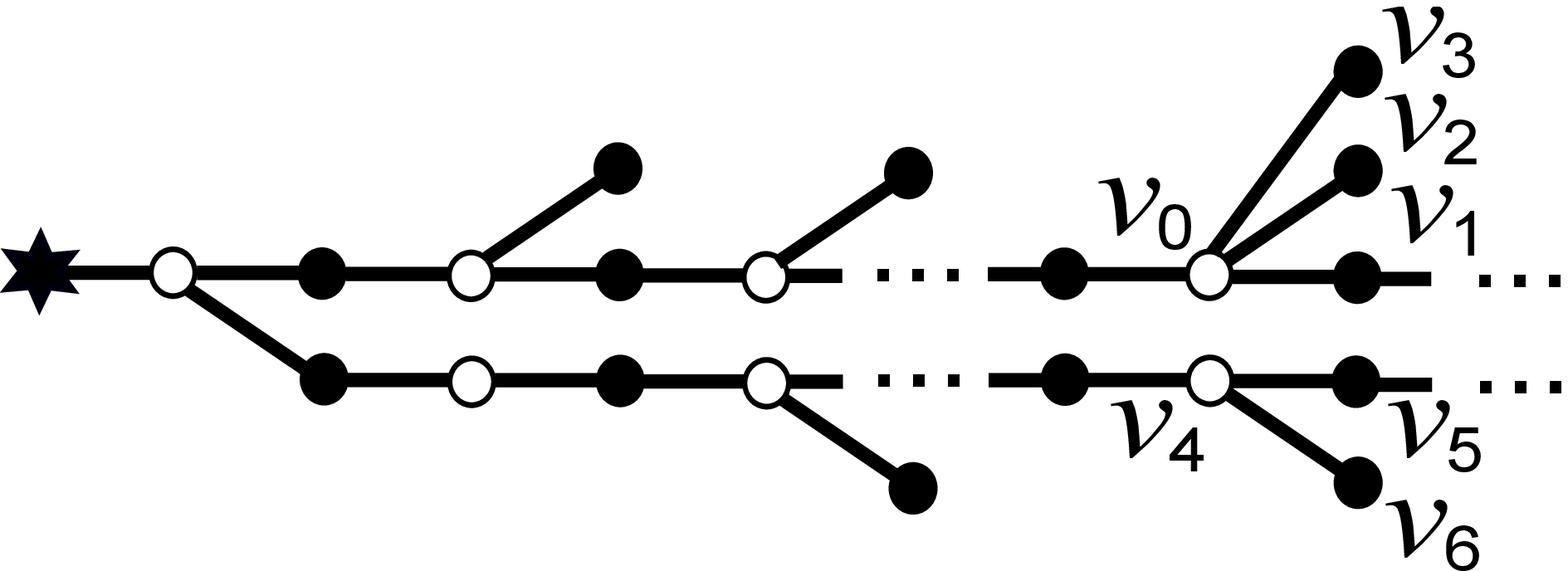}.$$

The vertex $v_0$ corresponds to the minimal projection of $FC_{2n-1,-}$ with middle pattern
\\
$\underbrace{baab~\cdots~baab}_{n-1}ba$.
So $\lambda'(v_0)=\delta_a\delta_b^{n}$;

The vertex $v_1$ corresponds to the minimal projection $P_c'$; The vertex $v_2$ corresponds to the minimal projection $P_d'$;

In the case $n=1$, there is no vertex $v_3$; In the case $n\geq 2$, the vertex $v_3$ corresponds to the minimal projection of $FC_{2n,-}$ with middle pattern $\underbrace{baab~\cdots~baab}_{n-1}bb$.
So $\lambda'(v_3)=\delta_b^{n-1}$.

In the case $n=1$, there is no vertex $v_4$; In the case $n\geq2$ the vertex $v_4$ corresponds to the minimal projection of $FC_{2n-1,-}$ with middle pattern $bb\underbrace{baab~\cdots~baab}_{n-2}~ba$.
So $\lambda'(v_4)=\delta_a\delta_b^{n-2}$;

The vertex $v_5$ corresponds to the minimal projection of $FC_{2n,-}$ with middle pattern
\\
$bb~\underbrace{baab~\cdots~baab}_{n-1}$.
So $\lambda'(v_5)=\delta_b^{n-1}$.

In the case $n\leq 2$, there is no vertex $v_6$; In the case $n\geq3$, the vertex $v_5$ corresponds to the minimal projection of $FC_{2n,-}$ with middle pattern $bb\underbrace{baab~\cdots~baab}_{n-2}bb$. So $\lambda'(v_6)=\delta_b^{n-3}$.

In the principal graph, there is one vertex at depth $2n+1$ with multiplicity 2. So in the dual principal graph, there is one vertex at depth $2n+1$ with multiplicity 2, denoted by $v_7$.

While $\delta\lambda'(v_5)-\lambda'(v_4)=\delta_a\delta_b^{n}-\delta_a\delta_b^{n-2}=\delta_a\delta_b^{n-1}$. So $v_5$ is adjacent to $v_7$. Then at most one of $v_1$ and $v_2$ is adjacent to $v_7$. Without loss of generality, we assume that $v_2$ is not adjacent to $v_7$. Then $\lambda'(v_2)=\frac{1}{\delta}\lambda'(v_0)=\delta_b^{n-1}$. So $\lambda'(v_1)=tr(g_{2n})-\lambda'(v_2)=\delta_b^{n+1}-\delta_b^{n-1}=\delta_b^n$.
Then $\delta\lambda'(v_1)-\lambda'(v_0)=\delta_a\delta_b^{n+1}-\delta_a\delta_b^{n}=\delta_a\delta_b^{n-1}$. So $v_1$ is adjacent to $v_7$, and $\lambda'(v_7)=\delta_a\delta_b^{n-1}$.
While $\delta\lambda'(v_7)-\lambda'(v_1)-\lambda'(v_5)=2\delta_b^n-\delta_b^n-\delta_b^{n-1}=\delta_b^{n-2}$. So there is a new $\mathcal{N}$ coloured vertex, denoted by $v_8$, adjacent to $v_7$. Then $\lambda'(v_8)\leq\delta_b^{n-2}$. On the other hand $\lambda'(v_8)\geq\delta^{-1}\lambda'(v_7)=\delta_b^{n-2}$. So $\lambda'(v_8)=\delta_b^{n-2}$. And there is no new vertices in the dual principal graph.

Therefore we obtain the unique possible dual principal graph.
\end{proof}

\begin{definition}
Let us define $\Gamma_n$ to be the (potential) dual principal graph of $\mathscr{B}$.
\end{definition}

Note that the minimal projection $g_{2n}$ of $FC_{2n,-}$ is separated into two minimal projections $P_c',~P_d'$ in $\mathscr{B}_{2n,-}$. And $tr(P_c')=\lambda(v_1)=\delta_b^{n},~tr(P_d')=\lambda(v_2)=\delta_b^{n-1}$. Take $R$ to be $\delta_b^{-1}P_c'-\delta_b^{-2}P_d'$, then $R$ is orthogonal to $FC_{2n,-}$ in $\mathscr{B}_{2n,-}$.
Recall that $\mathcal{F}(S)\in FC_{2n,-}$ is totally uncappable, so $\mathcal{F}(S)$ is also orthogonal to $FC{2n,-}$ in $\mathscr{B}_{2n,-}$. While the orthogonal complement of $FC_{2n,-}$ in $\mathscr{B}_{2n,-}$ is one dimensional. So $\mathcal{F}(S)$ is a multiple of $R$. Then we have the following propositions.

\begin{proposition}\label{relation of R}
For $R=\delta_b^{-1}P_d'-\delta_b^{-2}P_c'$ in $\mathscr{B}_{2n,-}$,  we have

(0) $R=\omega_0\delta^{-1}\mathcal{F}(S)$, for a constant $\omega_0$ satisfying $\omega_0^{-2}=\omega$, where $S$ and $\omega$ are given in Proposition \ref{relation of S};

(1') $R^*=R$;

(2') $R+\delta_b^{-2}g_{2n}$ is a projection;

(3') $R$ is totally uncappable;

(4') $\rho(R)=\omega R$.
\end{proposition}

\begin{proof}
(1') $R^*=(\delta_b^{-1}P_d'-\delta_b^{-2}P_c')^*=R$.

(0) By the argument above, we have $\mathcal{F}(S)$ is a multiple of $R$.
While
$$||\mathcal{F}(S)||_2^2=tr(S*S)=tr(f_{2n})=\delta_a^2\delta_b^n \quad\text{and}$$
$$||R||_2^2=tr(R^*R)=\delta_b^{-2}tr(P_c')+\delta_b^{-4}tr(P_d')=\delta_b^{-2}\delta_b^{n-1}+\delta_b^{-4}\delta_b^{n}=\delta_b^{n-2}=\delta^{-2}||\mathcal{F}(S)||_2^2.$$
So $R=\omega_0\delta^{-1}\mathcal{F}(S)$, for some phase $\omega_0$, i.e. $\omega_0\in\mathbb{C}$ and $|\omega_0|=1$.

Note that
$$(\mathcal{F}(R))^*=\mathcal{F}^{-1}(R^*)=\mathcal{F}^{-1}(R).$$
So $$(\omega_0\delta^{-1}\mathcal{F}^2(S))^*=(\mathcal{F}(R))^*=\mathcal{F}^{-1}(R)=\omega_0\delta^{-1}(S).$$
Then $$\omega_0\rho(S)=(\omega_0S)^*=\overline{\omega_0}S.$$
Recall that $\rho(S)=\omega S$.
Thus $\omega_0^{-2}=\omega$.

(2') $R+\delta_b^{-2}g_{2n}=P_d'$ is a projection.

(3') and (4') follows from (0).

\end{proof}

By the embedding theorem, we hope to solve $(S,R,\omega_0)$ in the graph planar algebra, such that $(S,R,\omega_0)$ satisfies the propositions (0)(1)(2)(3)(4)(1')(2')(3')(4') listed in Proposition(\ref{relation of S})(\ref{relation of R}). In this case, there is no essential difference to solve it in the graph planar algebra of the principal graph or the dual principal graph. But for computations, we may avoid a factor $\frac{1}{2}$ in the graph planar algebra of the dual principal graph. The factor $\frac{1}{2}$ comes from the symmetry of $c_0,d_0$ and $c_{2n},d_{2n}$ in the principal graph. Now let us describe the refined dual principal graph of $\mathscr{B}$.

\begin{lemma}\label{refined dual principal graph}
The refined principal graph of $\mathscr{B}$ with respect to the biprojection $e_\mathcal{P}$ is

$$\grb{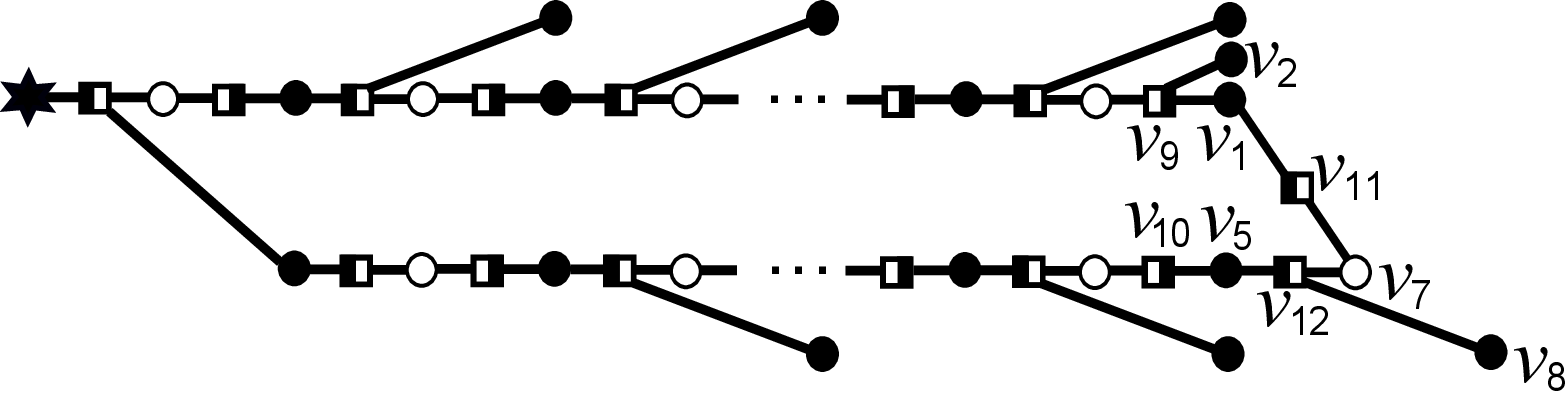}.$$

For computations, let us adjust the refined principal graph and relabel its the vertices as
$$\grc{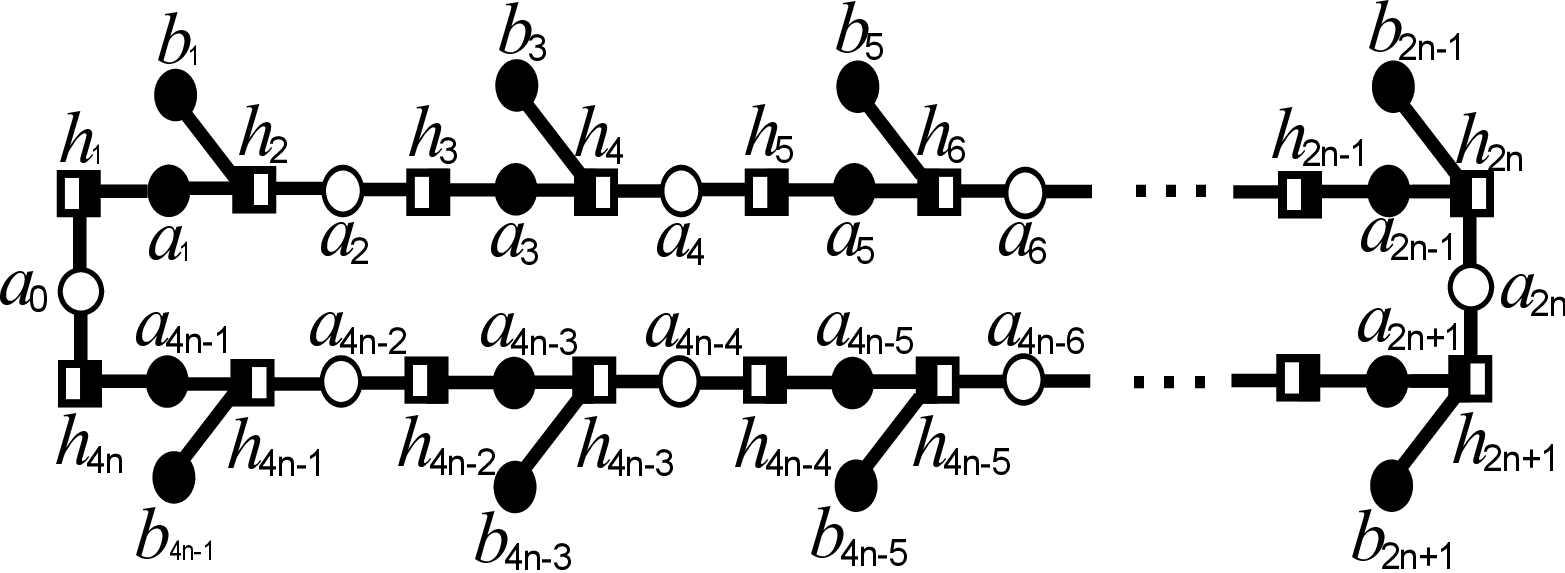},$$
where the marked vertex is $b_1$. For convenience, we assume that $a_{4n}=a_0$.

Then its dimension vector $\lambda'$ is given by

$\lambda'(a_{2k-1})=\lambda'(a_{4n-2k+1})=\delta_b^k$, for $1\leq k\leq n$;

$\lambda'(b_{2k-1})=\lambda'(b_{4n-2k+1})=\delta_b^{k-1}$, for $1\leq k\leq n$;

$\lambda'(a_{2k})=\lambda'(a_{4n-2k})=\delta_a\delta_b^k$, for $0\leq k\leq n$;

$\lambda'(h_{2k-1})=\lambda'(h_{4n-2k+2})=\delta_b^{k-1}$, for $1\leq k\leq n$;

$\lambda'(h_{2k})=\lambda'(h_{4n-2k+1})=\delta_b^{k}$, for $1\leq k\leq n$.
\end{lemma}

\begin{proof}
The proof is similar to that of Lemma \ref{refined principal graph}.

We have known that $\mathscr{B}_{2n,-}=FC_{2n,-}\oplus\mathbb{C}(R)$, where $\mathbb{C}(R)$ is the one dimensional vector space generated by the totally uncappable element $R$.
So we obtain the refined principal graph up to depth $2n$ as mentioned in the statement.

For the vertices $v_9,v_{10}$ as marked in the statement, we have $\lambda'(v_9)=\delta_b\lambda'(v_2)=\delta_b\delta_b^{n-1}=\delta_b^{n}$, $\lambda'(v_{10})=\delta_b^{-1}\lambda'(v_5)=\delta_b^{-1}\delta_b^{n-1}=\delta_b^{n-2}$.

Then $\delta_b\lambda'(v_1)-\lambda(v_9)=\delta_b\delta_b^{n}-\delta_b^{n}=\delta_b^{n-1}$. So $v_1$ is adjacent to a new $\mathcal{P}$ coloured vertex, denoted by $v_{11}$. Then $\lambda'(v_{11})\leq \delta_b^{n-1}$. On the other hand $\lambda'(v_{11})\geq\delta_b^{-1}\lambda'(v_1)=\delta_b^{n-1}$. So $v_{11}$ is the unique new $\mathcal{P}$ coloured vertex adjacent to $v_{1}$ and $\lambda'(v_{11})=\delta_b^{n-1}$. Then $\delta_b\lambda'(v_{11})=\lambda'(v_1)$ implies $v_8$ is not adjacent to $v_{11}$. And the $\mathcal{N}$ coloured vertex adjacent to $v_{11}$ has to be $v_7$.

Moreover $\delta_b\lambda'(v_5)-\lambda(v_{10})=\delta_b\delta_b^{n-1}-\delta_b^{n-2}=\delta_b^{n-1}$. So $v_1$ is adjacent to a new $\mathcal{P}$ coloured vertex, denoted by $v_{12}$. Then $\lambda'(v_{12})\leq \delta_b^{n-1}$. On the other hand $\lambda'(v_{12})\geq\delta_b^{-1}\lambda'(v_5)=\delta_b^{n-2}>\frac{1}{2}\delta_b^{n-1}$. So $v_{12}$ is the unique new $\mathcal{P}$ coloured vertex adjacent to $v_{5}$ and $\lambda'(v_{12})=\delta_b^{n-1}$. Then $\delta_b\lambda'(v_{12})-\lambda'(v_5)=\lambda'(v_8)$ implies $v_8$ is adjacent to $v_{11}$. And the $\mathcal{N}$ coloured vertex adjacent to $v_{11}$ has to be $v_7$.

While $\delta_a\lambda(v_7)=2\delta_b^{n-1}=\lambda'(v_{11})+\lambda'(v_{12})$, $\delta_b\lambda'(v_8)=\delta_b^{n-1}=\lambda'(v_{12})$. So there is no new $\mathcal{P}$ coloured vertices.
Then we have the unique possible refined dual principal of $\mathscr{B}$.

Now we adjust the refined principal graph and relabel its the vertices as
$$\grc{principalgraph24},$$
where the marked vertex is $b_1$.

The graph is vertically symmetrical, by Corollary \ref{unique dimension vector}, the dimension vector $\lambda'$ is also symmetric. So we only need to compute the value of $\lambda'$ for the upper half vertices.

The vertex $a_1$ corresponds to the minimal projection of $FC_{2,-}$ with middle pattern $bb$. So $\lambda'(a_1)=\delta_b$;
The vertex $a_{2k-1}$ corresponds to the minimal projection of $FC_{2k-2,-}$ with middle pattern $baab~\cdots~baab$, $k-1$ copies of $baab$, for $2\leq k\leq n$.
So $\lambda'(a_{2k-1})=\delta_b^k$, for $2\leq k\leq n$;

The vertex $b_1$ is the marked vertex. So $\lambda'(b_1)=1$;
The vertex $b_{2k-1}$ corresponds to the minimal projection of $FC_{2k-1,-}$ with middle pattern $baab~\cdots~baab~bb$, $k-1$ copies of $baab$, for $2\leq k\leq n$.
So $\lambda'(b_{2k-1})=\delta_b^{k-1}$, for $2\leq k\leq n$;

The vertex $a_0$ corresponds to the minimal projection of $FC_{3,-}$ with middle pattern $bbba$. So $\lambda'(a_0)=\delta_a$;
The vertex $a_{2k}$ corresponds to the minimal projection of $FC_{2k-1,-}$ with middle pattern $baab~\cdots~baab~ba$, $k-1$ copies of $baab$, for $1\leq k\leq n$.
So $\lambda'(a_{2k})=\delta_a\delta_b^k$, for $1\leq k\leq n$;

The vertex $h_1$ corresponds to the minimal projection of $FC'_{3,-}$ with middle pattern $bbb$. So $\lambda'(h_1)=1$;
The vertex $h_{2k-1}$ corresponds to the minimal projection of $FC'_{2k-1,-}$ with middle pattern $baab~\cdots~baab~baa$, $k-2$ copies of $baab$, for $2\leq k\leq n$.
So $\lambda'(a_{2k})=\delta_b^{k-1}$, for $2\leq k\leq n$;

The vertex $h_1$ corresponds to the minimal projection of $FC'_{3,-}$ with middle pattern $bbb$. So $\lambda'(h_1)=1$;
The vertex $h_{2k}$ corresponds to the minimal projection of $FC'_{2k-1,-}$ with middle pattern $baab~\cdots~baab~b$, $k-1$ copies of $baab$, for $1\leq k\leq n$.
So $\lambda'(a_{2k})=\delta_b^{k-1}$, for $1\leq k\leq n$;
\end{proof}

We hope to embed $\mathscr{B}_{m,\mp}$ in the graph planar algebra of the dual principal graph, so we will consider the biprojection $e_{\mathcal{P}_1}=\delta_a^{-1}\delta_b\mathcal{F}(e_\mathcal{P})$ in $\mathscr{B}_{2,-}$.
\begin{definition}
Let us define $\mathscr{G}=\mathscr{G}_{m\pm}$ to be the graph planar algebra of the dual principal graph $\Gamma_n$.
Then $\mathscr{B}_{m,\mp}$ is naturally embedded in $\mathscr{G}_{m,\pm}$. Let $p_1\in\mathscr{G}_{2,+}$ be the image of $e_{\mathcal{P}_1}$.
Then the planar subalgebra $FC(\delta_b,\delta_a)_{m,\pm}$ of $\mathscr{G}$ generated by $p_1$ is identical to the image of $FC(\delta_a,\delta_b)_{m,\mp}$.
The images of $f_{2n}$ and $g_{2n}$ are still denoted by $f_{2n}$ and $g_{2n}$.
\end{definition}

\begin{notation}
Note that the dual principal graph $\Gamma$ is simply laced. A path $\varepsilon$ of $\Gamma_n$ is determined by $s(\varepsilon)$ and $t(\varepsilon)$, so we may use
$$[s(\varepsilon_1)t(\varepsilon_1)s(\varepsilon_3)t(\varepsilon_3)\cdots s(\varepsilon_{2m-1})t(\varepsilon_{2m-1})]$$
to express a loop $[\varepsilon_1\varepsilon_2^*\varepsilon_3\varepsilon_4^*\cdots\varepsilon_{2m-1}\varepsilon_{2m}^*]$ in $\mathscr{G}_{2m,+}$,
similarly for loops in $\mathscr{G}_{2m,-}$.
\end{notation}

\begin{proposition}\label{loop p}
$$p_1=\sum_{k=1}^n [a_{2k-1}a_{2k-2}a_{2k-1}a_{2k-2}]
+[a_{4n-2k+1}a_{4n-2k+2}a_{4n-2k+1}a_{4n-2k+2}]$$
$$+[a_{2k-1}a_{2k}a_{2k-1}a_{2k}]
+[a_{4n-2k+1}a_{4n-2k}a_{4n-2k+1}a_{4n-2k}]$$
$$+[a_{2k-1}a_{2k}b_{2k-1}a_{2k}]
+[a_{4n-2k+1}a_{4n-2k}b_{4n-2k+1}a_{4n-2k}]$$
$$+[b_{2k-1}a_{2k}b_{2k-1}a_{2k}]
+[b_{4n-2k+1}a_{4n-2k}b_{4n-2k+1}a_{4n-2k}]$$
$$+[b_{2k-1}a_{2k}a_{2k-1}a_{2k}]
+[b_{4n-2k+1}a_{4n-2k}a_{4n-2k+1}a_{4n-2k}].$$
\end{proposition}

\begin{proof}
It follows from Theorem \ref{embedding P} and Lemma \ref{refined dual principal graph}.
\end{proof}

\begin{definition}
Note that $\mathscr{G}_{0,+}$ is abelian. Let us define $A_{k},~B_{k}$ to be the minimal projections corresponding to the vertices $a_{2k-1},~b_{2k-1}$ respectively, for $1\leq k\leq 2n$.

Note that $\mathscr{G}_{1,+}$ is abelian. Let us decompose $A_{k}$ into minimal projections $A_{k}^-$ and $A_{k}^+$ as follows,
$A_{k}^-=[a_{2k-1}a_{2k-2}]$, $A_{2n-k}^-=[a_{4n-2k+1}a_{4n-2k+2}]$, $A_{k}^+=[a_{2k-1}a_{2k}]$, $A_{2n-k}^+=[a_{4n-2k+1}a_{4n-2k}]$, for $1\leq k\leq n$.

Let us define $H_{2k-1}$, $H_{4n-2k+1}$, $H_{2k}$ and $H_{4n-2k}$ in $\mathscr{G}_{1,-}$, for $1\leq k\leq n$, as follows
$$H_{2k-1}=[a_{2k-2}a_{2k-1}], H_{2k}=[a_{2k}a_{2k-1}]+[a_{2k}b_{2k-1}],$$
$$H_{4n-2k+2}=[a_{4n-2k+2}a_{4n-2k+1}],H_{4n-2k+1}=[a_{4n-2k}a_{4n-2k+1}]+[a_{4n-2k}b_{4n-2k+1}].$$
\end{definition}

\begin{proposition}\label{ABg commute}
\mbox{}

$A_{k},~B_{k}$ are in the center of $\mathscr{G}_{2n,+}$.

$g_{2n}$ commutes with $A_{k}^+$ and $A_{k}^-$.
\end{proposition}

\begin{proof}
The first statement is obvious.
For the second statement, it is enough to check $p_1$ commutes with $A_k^{+}$ and $A_k^{-}$.
By Proposition \ref{loop p}, for $1\leq k \leq n$, we have
$$p_1A_k^+=[a_{2k-1}a_{2k}a_{2k}a_{2k-1}a_{2k}]=A_k^+p_1;$$
similarly for other cases.
\end{proof}

\subsection{The potential generater}

Now we sketch the idea of solving the generator $R$ in $\mathscr{G}$. Essentially we are considering the length $8n$ loops on the refined dual principal graph. Observe that if a loop contains a word $h_ka_kh_k$, for $1\leq k\leq 2n$,
then the vertex $a_k$ could be replaced by an a/b-colour cap, because $a_k$ is the unique $\mathcal{N}/\mathcal{M}$ coloured vertex adjacent to $h_k$. The coefficient of such a loop in the totally uncappable element $R$ has to be 0.
Therefore for a loop $l$ with non-zero coefficient in $R$, if it goes to the right, then it will not return until passing the vertex $a_{2n}$. Among these loops, there is exactly one in $A_1^-\mathscr{G}_{2n,+}A_1^+$, that tells the initial condition of $R$.
By proposition(2'), $A_kRA_k$ is determined by $A_k^-RA_k^+$. By proposition(3'), $B_kR$ is determined by $A_k^+RA_k^+$. By proposition(4'), $A_{k+1}^-RA_{k+1}$ is determined by $(A_k+B_k)R(A_k+B_k)$.
That means $R$ could be computed inductively by the initial condition.

\begin{definition}
Let us define $F\in \mathscr{G}_{2,+}$ to be the image of $\mathcal{F}(id-e_\mathcal{P})$, i.e. $F=\delta e_1-\delta_a\delta_b^{-1}p_1$.
\end{definition}
It is easy to check that $F*F=F$, $P*F=F*P=0$, and $F*g_{2n}=g_{2n}*F=0.$

Note that $e_1$ and $p_1$ could be expressed as linear sums of loops, then we have.

$$F=\sum_{1\leq k\leq n}\delta_a\delta_b^{-0.5}([a_{2k-1}a_{2k-2}a_{2k-1}a_{2k}]+[a_{4n-2k+1}a_{4n-2k+2}a_{4n-2k+1}a_{4n-2k}]$$
$$+[a_{2k-1}a_{2k}a_{2k-1}a_{2k-2}]+[a_{4n-2k+1}a_{4n-2k}a_{4n-2k+1}a_{4n-2k+2}])$$
$$+\delta_a\delta_b^{-2}([a_{2k-1}a_{2k}a_{2k-1}a_{2k}]
+[a_{4n-2k+1}a_{4n-2k}a_{4n-2k+1}a_{4n-2k}])$$
$$-\delta_a\delta_b^{-1}([a_{2k-1}a_{2k}b_{2k-1}a_{2k}]
+[a_{4n-2k+1}a_{4n-2k}b_{4n-2k+1}a_{4n-2k}])$$
$$+\delta_a([b_{2k-1}a_{2k}b_{2k-1}a_{2k}]
+[b_{4n-2k+1}a_{4n-2k}b_{4n-2k+1}a_{4n-2k}])$$
$$-\delta_a\delta_b^{-1}([b_{2k-1}a_{2k}a_{2k-1}a_{2k}]
+[b_{4n-2k+1}a_{4n-2k}a_{4n-2k+1}a_{4n-2k}]).$$

We may compute $F*l$ for a loop $l\in\mathscr{G}_{2n,+}$ by the following fact,
$$[y_0y_1y_2y_3]*[x_0x_1\cdots x_{4n-1}]=\delta_{y_1x_1}\delta_{y_2x_0}\delta_{y_3x_{4n-1}} \sqrt{\frac{\lambda'(y_2)\lambda'(y_2)}{\lambda'(y_1)\lambda'(y_3)}} [y_0x_1\cdots x_{4n-1}].$$

\begin{proposition}\label{F inv}
For a loop $l\in\mathscr{G}_{2n,+}$ and $1\leq k\leq 2n$, we have

$F*l=0$, when $l=A_{k}^-l A_{k}^-$,

$F*l=l$, when $l=A_{k}^-l A_{k}^+$ or $l=A_k^+ l A_k^-$;

$F*l=(A_{k}^++B_{k})(F*l)(A_{k}^++B_{k})$, when $l=(A_{k}^++B_{k})l(A_{k}^++B_{k})$.

So $\mathscr{G}_{2n,+}$ is separated into $6n$ invariant subspaces under the the action $F*$.
Moreover the set of length $4n$ loops, as a basis of $\mathscr{G}_{2n,+}$, is separated into $6n$ subsets simultaneously.

\end{proposition}

\begin{proof}
It could be checked by a direct computation.
\end{proof}

\begin{definition}
Let $\beta:A_k^+\mathscr{G}_{2n,+}A_k^+\rightarrow B_k\mathscr{G}_{2n,+}, ~\forall ~ 1\leq k\leq 2n$ be the linear extension of
$$\beta([a_{2k-1}a_{2k-2}x_3x_4\cdots x_{2n-1} a_{2k-2}])=[b_{2k-1}a_{2k-2}x_3x_4\cdots x_{2n-1} a_{2k-2}],$$
for any loop $[a_{2k-1}a_{2k-2}x_3x_4\cdots x_{2n-1} a_{2k-2}]\in A_k^+\mathscr{G}_{2n,+}A_k^+$.
\end{definition}

\begin{proposition}\label{F ++}
The linear map $\beta:A_k^+\mathscr{G}_{2n,+}A_k^+\rightarrow B_k\mathscr{G}_{2n,+}$ is a *-isomorphism.
Moreover
$$F*x=\delta_b^{-2}x-\delta_b^{-1}\beta(x),     ~\forall~ x\in A_k^+\mathscr{G}_{2n,+}A_k^+;$$
$$F*y=\delta_b^{-1}y-\delta_b^{-2}\beta^{-1}(y),   ~\forall ~ y\in B_k\mathscr{G}_{2n,+};$$
$$\beta(A_k^+ g_{2n})=B_k g_{2n}.$$
\end{proposition}

\begin{proof}
It is obvious that $\beta$ is a *-isomorphism.
It is easy to check the first two formulas by a direct computation.
For the third formula, by Proposition \ref{F inv} and the fact that $F*g_{2n}=0$, we have
$$F*((A_k^++B_k)g_{2n}(A_k^++B_k))=0.$$
By Proposition \ref{ABg commute}, we have
$$F*(A_k^+g_{2n})=-F*(B_k g_{2n}).$$
Then
$$\delta_b^{-2}(A_k^+g_{2n})-\delta_b^{-1}\beta(A_k^+g_{2n})=-\delta_b^{-1}(B_k g_{2n})+\delta_b^{-2}\beta^{-1}(B_k g_{2n}).$$
So
$$\beta(A_k^+ g_{2n})=B_k g_{2n}.$$
\end{proof}

\begin{lemma}\label{--=0}
\mbox{}

$A_{k}^-R A_{k}^-=0$, for $1\leq k\leq 2n$.

$H_i\mathcal{F}(R)H_i=0$, for $1\leq i \leq 4n$.

\end{lemma}

\begin{proof}
By proposition (3'), $R$ is totally uncappable, so $R=F*R$.
Then by Proposition \ref{F inv}, we have
$$(A_{k}^-RA_{k}^-)=F*(A_{k}^-RA_{k}^-)=0.$$
Note that
$$\sum_{1\leq i\leq 4n}H_i\mathcal{F}(R)H_i=\mathcal{F}(Rp_1)=0,$$
so
$$H_i\mathcal{F}(R)H_i=0,\quad \forall~ 1\leq i \leq 4n.$$
\end{proof}

\begin{lemma}\label{ini U}
\mbox{}

$A_1^-RA_1^+$ is a multiple of the loop $[a_1a_{4n}a_{4n-1}\cdots a_2]$, denote by $L_1$;

$A_{2n}^-RA_{2n}^+$ is a multiple of the loop $[a_{4n-1}a_{0}a_{1}\cdots a_{4n-2}]$, denote by $L_2$.
\end{lemma}

\begin{proof}
Note that the coefficient of a loop $l=[a_1a_{4n}x_3x_4\cdots x_{4n-1}a_2]$ in $A_1^-RA_1^+$ is the same as the coefficient of $l$ in $R$.
If it is non-zero, then
by Proposition(4'), the coefficient of $\mathcal{F}^{-2k+1}(l)$ in $\mathcal{F}(R)$ is non-zero and the coefficient of $\mathcal{F}^{-2k}(l)$ in $R$ is non-zero. Applying Lemma \ref{--=0}, we have
$$H_1\mathcal{F}(R)H_1=0 \Rightarrow x_3=a_{4n-1};$$
$$a_{4n-1}^-R a_{4n-1}^-=0 \Rightarrow x_4=a_{4n-2};$$
and for $k=1,2,\cdots, n$,
$$H_{4n+3-2k}\mathcal{F}(R)H_{4n+3-2k}=0\Rightarrow x_{2k+1}=a_{4n+1-2k};$$
$$a_{4n+1-2k}^-R a_{4n+1-2k}^-=0 \Rightarrow x_{2k+2}=a_{4n-2k}.$$
For the rest part, there is only one length $2n-2$ path from $a_{2n}$ to $a_2$.
So
$$l=[a_1a_{4n}a_{4n-1}\cdots a_{2}]=L_1.$$
That means $A_1^-RA_1^+$ is a multiple of $L_1$.
Similarly $A_{2n}^-RA_{2n}^+$ is a multiple of $L_2$
\end{proof}

\begin{definition}
For a loop $l=[x_0x_1\cdots x_{4n-1}]$ and $0\leq k\leq 4n-1$,
the point $x_k$ is said to be a cusp point of the loop $l$, if $x_{k-1}=x_{k+1}$, where $x_{-1}=x_{2n-1}, ~x_{2n}=x_0$.
Otherwise it is said to be a flat point.
\end{definition}

Similar to the proof of Lemma \ref{ini U}, Lemma \ref{--=0} tells that if the coefficient of a loop $l=[x_0x_1\cdots x_{4n-1}]$ in $R$ is non-zero, then the cusp point $x_k$ of $l$ has to be $b_{2i-1}$ or $a_{2i-1}$. In this case, we have $x_{k-1}=x_{k+1}=a_{2i}$, when $1\leq i \leq n$; Or $x_{k-1}=x_{k+1}=a_{2i-2}$, when $n+1\leq i \leq 2n$.
Furthermore if $l$ passes the point $a_0$, then it is unique up to rotation and the adjoint operation $*$;
If $l$ does not pass the point $a_0$, then it is determined by its first point and cusp points.
So we may simplify the expression of a loop by its first point and cusp points.
To compute the product of two loops, we also need the middle point $x_{2n}$.
Then the loop is separated into two length $2n$ paths from the first point to the middle point.
We may label the two paths by the first point, cusp points and the middle point.

\begin{definition}
For a loop $l=[x_0 x_1\cdots x_{4n-1}]$, $x_k\neq a_0,~\forall~ 0\leq k \leq 4n-1$, we assume that $y_1,y_2,\cdots,y_i$ are the cusp points from $x_1$ to $x_{2n-1}$ and $z_1,z_2,\cdots,z_j$ are the cusp points from $x_{2n+1}$ to $x_{4n-1}$.
Then we use $[x_0y_1y_2\cdots y_ix_{2n}\rangle$ to express the first length $2n$ path of $l$,
$\langle x_{2n} z_1z_2\cdots z_jx_{0}]$ to express the second length $2n$ path of $l$ and
$[x_0y_1y_2\cdots y_ix_{2n}\rangle\langle x_{2n} z_1z_2\cdots z_jx_{0}]$
to express the loop $l$. Furthermore
if $x_{2n}$ is a cusp point, then it could be simplified as $[x_0y_1y_2\cdots y_i x_{2n} z_1z_2\cdots z_jx_{0}]$;
if $x_{2n}$ is a flat point, then it could be simplified as $[x_0y_1y_2\cdots y_i z_1z_2\cdots z_jx_{0}]$.


\end{definition}

\begin{definition}
Suppose $R\in\mathscr{G}_{2n,+}$ is a solution of Proposition \ref{relation of R}, i.e. $R$ satisfies the following propositions,

(1') $R^*=R$;

(2') $R+\delta_b^{-2}g_{2n}$ is a projection;

(3') $R$ is totally uncappable;

(4') $\rho(R)=\omega R$,  for some $\omega\in\mathbb{C}$ satisfying $|\omega|=1$.

Let us define $U_k,~P_k,~Q_k,~\overline{P}_k,~\overline{Q}_k,~R_k$ for $1\leq k\leq 2n$ as follows

$U_k=A_{k}^-RA_{k}^+$;



$\overline{P}_k=\delta_b^{-2}(R-\delta_b^{-1}g_{2n})B_{k}$;

$\overline{Q}_k=\delta_b^{-1}(R+\delta_b^{-2}g_{2n})B_{k}$;

$P_k=-\delta_b^{-1}\beta^{-1}(\overline{P}_k)$;

$Q_k=-\delta_b^{-1}\beta^{-1}(\overline{Q}_k)$;

$R_k=(A_{k}^++B_{k})R(A_{k}^++B_{k})$.
\end{definition}

The following lemma is the key to solve the generator $R$ in the graph planar algebra $\mathscr{G}_{2n,+}$.

\begin{lemma}\label{key lemma}
\mbox{}

$U_1=\mu_1\delta_b^{-1.5}L_1$, for some $\mu_1\in\mathbb{C}$, $|\mu_1|=1$;

$U_{2n}=\mu_2\delta_b{-1.5}L_2$, for some $\mu_2\in\mathbb{C}$, $|\mu_2|=1$;

$P_k=U_k^* U_k$, for $1\leq k\leq 2n;$

$R_k=\delta_b^4 F*P_k*F$, for $1\leq k\leq 2n;$

$U_{k+1}=\omega^{-1}\rho(R_k+U_k)$ and
$U_{2n-k}=\omega^{-1}\rho(R_{2n-k+1}+U_{2n-k+1})$, for $1\leq k\leq n-1;$

$R=\sum_{1\leq k\leq 2n}U_k+U_k^*+R_k$.

So $R$ is uniquely determined by $\mu_1,~\mu_2~and~\omega$.
\end{lemma}

\begin{proof}
For $1\leq k\leq 2n$, by definition, we have
$$R B_k=-\delta_b^{-2}(\delta_b^{-1}g_{2n}-R)B_k+\delta_b^{-1}(R+\delta_b^{-2})B_k
=\overline{P}_k+\overline{Q}_k.$$
By proposition (2')(3'), we have $R+\delta_b^{-2}g_{2n}$ is a subprojection of $g_{2n}$. Then
$$g_{2n}-(R+\delta_b^{-2}g_{2n})=\delta_b^{-1}g_{2n}-R$$
is a projection. So
$$\delta_b \overline{Q}_k=(R+\delta_b^{-2})B_k,~ -\delta_b^2\overline{P}_k=(\delta_b^{-1}g_{2n}-R)B_k$$
are projections, by Proposition \ref{ABg commute}.
Note that
$$R_k=(A_{k}^++B_{k})R(A_{k}^++B_{k})=
A_k^+RA_k^+ + B_k R B_k,$$
so $F*R_k=R_k$, by Proposition \ref{F inv}.
Furthermore by Proposition \ref{F ++}, we have
$$F*R_k
=\delta_b^{-2}A_k^+RA_k^+ -\delta_b^{-1}\beta(A_k^+RA_k^+) + \delta_b^{-1}B_k R B_k-\delta_B^{-2}\beta^{-1}(B_k R B_k).$$
Thus
$$A_k^+RA_k^+=\delta_b^{-2}A_k^+RA_k^+ -\delta_b^{-2}\beta^{-1}(B_k R B_k).$$
Then
$$A_k^+RA_k^+=-\delta_b^{-1}\beta^{-1}(B_k R B_k)
=-\delta_b^{-1}\beta^{-1}(\overline{P}_k+\overline{Q}_k)=P_k+Q_k.$$
By Proposition \ref{F ++}, we have
$$A_k^+g_{2n}=\beta^{-1}(B_k g_{2n})=\beta^{-1}(-\delta_b^{2}\overline{P}_k+\delta_b\overline{Q}_k)=\delta_b^{3}P_k - \delta_b^2Q_k,$$
and $\delta_b^{3}P_k$, $-\delta_b^2Q_k$ are projections.
Then
$$A_k^+ (R+\delta_b^{-2} g_{2n}) A_k^+=(P_k+Q_k)+(\delta_bP_k-Q_k)=\delta_b^2 P_k.$$
By Proposition(\ref{ABg commute})(\ref{--=0}) and proposition(1'), we have
$$\begin{bmatrix}
A_k^- (R+\delta_b^{-2} g_{2n}) A_k^-&A_k^- (R+\delta_b^{-2} g_{2n}) A_k^+\\
A_k^+ (R+\delta_b^{-2} g_{2n}) A_k^-&A_k^+ (R+\delta_b^{-2} g_{2n}) A_k^+\\
\end{bmatrix}
=\begin{bmatrix}
\delta_b^{-2}A_k^-g_{2n}&U_k\\
U_k^*&\delta_b^2 P_k\\
\end{bmatrix}$$
Recall that $R+\delta_b^{-2}g_{2n}$ is a projection, so $A_k(R+\delta_b^{-2}g_{2n})$ is a projection.
Then the matrix
$$\begin{bmatrix}
\delta_b^{-2}A_k^-g_{2n}&U_k\\
U_k^*&\delta_b^2 P_k\\
\end{bmatrix}$$
is a projection.
While $A_k^-g_{2n}$ and $\delta_b^3 P_1$ are projections, so $\delta_b^{1.5}U_k$ is a partial isometry from $\delta_b^3 P_1$ to $A_k^-g_{2n}$.
Then
$$(\delta_b^{1.5}U_k)^*(\delta_b^{1.5}U_k)=\delta_b^3 P_k; \quad
(\delta_b^{1.5}U_k)(\delta_b^{1.5}U_k)^*=A_k^-g_{2n}.
$$
Therefore $$U_k^*U_k=P_k ~\text{and}~ U_1U_1^*=\delta_b^{-3}A_1^-g_{2n}.$$
Observe that $[a_1a_{4n}a_{4n-1}\cdots a_{2n+2}a_{2n+1}a_{2n+2}\cdots a_{4n}]$ is a subprojection of $A_1^-g_{2n}$. So $A_1^-g_{2n}\neq0$. Then $U_1\neq 0$.
By Lemma \ref{ini U}, we have
$$U_1=\mu_1\delta_b^{-1.5}L_1, ~\text{for some}~ \mu_1\in\mathbb{C},~|\mu_1|=1;$$
Symmetrically
$$U_{2n}=\mu_2\delta_b^{-1.5}L_2, ~\text{for some}~ \mu_2\in\mathbb{C},~|\mu_2|=1;$$

Note that
$$B_k(x*F)=(B_kx)*F, ~\forall~ x\in\mathscr{G}_{2n,+},$$
so
$$\delta_b^2\overline{P}_k*F=(B_kR)*F-\delta_b(B_kg_{2n})*F=B_k(R*F)-\delta_bB_k(g_{2n}*F)=B_kR=\overline{P}_k+\overline{Q}_k.$$
Observe that
$$\beta^{-1}(y*F)=\beta^{-1}(y)*F, ~\forall~ y\in B_k\mathscr{G}_{2n,+},$$
so
$$\delta_b^2P_k*F=P_k+Q_k.$$
By Proposition \ref{F ++}, we have
$$\delta_b^2 F*P_k=P_k-\delta_b\beta(P_k)=P_k+\overline{P}_k.$$
So
$$\delta_b^4 F*P_k*F=\delta_b^2(P_k+\overline{P}_k)*F=P_k+Q_k+\overline{P}_k+\overline{Q}_k$$
$$=A_k^+ R A_k^+ +RB_k=(A_k^+ +B_k)R(A_k^+ + B_k)=R_k.$$

Note that $\rho$ induces an one onto one map from the loops of $\mathscr{G}_{2n,+} (A_k^+ +B_k)$ to loops of $A_{k+1}^-\mathscr{G}_{2n,+}A_{k+1}^+$, for $1\leq k\leq n-1$.
So
$$\rho(R (A_k^+ +B_k))=A_{k+1}^-\rho(R)A_{k+1}^+.$$
Then by proposition (4'), we have
$$\rho(R (A_k^+ +B_k))=\omega A_{k+1}^- RA_{k+1}^+.$$
While
$$R (A_k^+ +B_k)=(A_k^+ +B_k)R(A_k^+ +B_k)+A_k^- R(A_k^+)=R_k+U_k,$$
thus
$$U_{k+1}=\omega^{-1}\rho(R_k+U_k).$$
Symmetrically we have
$$U_{2n-k}=\omega^{-1}\rho(R_{2n-k+1}+U_{2n-k+1}).$$

Finally
$$R=\sum_{1\leq k\leq 2n} (A_{k}+B_{k})R(A_{k}+B_{k})
=\sum_{1\leq k\leq 2n} (A_{k}^-+A_{k}^++B_{k})R(A_{k}^-+A_{k}^++B_{k})
$$
$$
=\sum_{1\leq k\leq 2n}A_{k}^-RA_{k}^+ + A_{k}^+RA_{k}^- +(A_{k}^++B_{k})R(A_{k}^++B_{k})
=\sum_{1\leq k\leq 2n}U_k+U_k^*+R_k
$$

Given $\mu_1,~\mu_2~and~\omega$, $U_k,P_k,R_k$ could be obtained inductively.
So $R$ is uniquely determined by $\mu_1,~\mu_2~and~\omega$.
\end{proof}

\subsection{Solutions}
\begin{definition}
Based on Lemma \ref{key lemma}, for fixed $\mu_1,\mu_2,\omega\in\mathbb{C}$, $|\mu_1|=|\mu_2|=|\omega|=1$, let us construct the unique possible generator $R_{\mu_1\mu_2\omega}\in\mathscr{G}_{2n,+}$ inductively,

$U_1=\mu_1\delta_b^{-1.5}L_1$;

$U_{2n}=\mu_2\delta_b^{-1.5}L_2$;

$P_k=U_k^* U_k$, for $1\leq k\leq 2n;$

$R_k=\delta_b^4 F*P_k*F$, for $1\leq k\leq 2n;$

$U_{k+1}=\omega^{-1}\rho(R_k+U_k)$ and
$U_{2n-k}=\omega^{-1}\rho(R_{2n-k+1}+U_{2n-k+1})$, for $1\leq k\leq n-1;$

$R_{\mu_1\mu_2\omega}=\sum_{1\leq k\leq 2n}U_k+U_k^*+R_k$.
\end{definition}

We hope to check proposition(1')(2')(3')(4') for $R_{\mu_1\mu_2\omega}$.
Actually proposition(1')(2')(3') are satisfied, but not obvious.
Proposition(4') fails, when $n\geq4$.
We are going to compute the coefficients of loops in $R_{\mu_1\mu_2\omega}$.
If proposition(4') is satisfied, then their absolute values are determined by the coefficients of loops in $R_k$.

\begin{lemma}\label{pro3}
$R_{\mu_1\mu_2\omega}$ is totally uncappable.
\end{lemma}

\begin{proof}
Note that $U_1$ is totally uncappable.
So
$$g_{2n}U_1g_{2n}=U_1.$$
Then
$$g_{2n}P_1g_{2n}=P_1.$$
By the exchange relation of the biprojection, we have
$$g_{2n}(F*P_1*F)g_{2n}=F*(g_{2n}*P_1*g_{2n})*F=F*p_1*F.$$
Therefore $R_1=F*P_1*F$ is totally uncappable.
Then $U_2=\omega^{-1}\rho(R_1)$ is totally uncappable.
Inductively we have $U_k,R_k$ are totally uncappable, for $k=1,2,\cdots,n.$
Symmetrically $U_i,R_i$ are totally uncappable, for $i=2n,2n-1,\cdots,n+1.$
So $R_{\mu_1\mu_2\omega}=\sum_{1\leq k\leq 2n}U_k+U_k^*+R_k$ is totally uncappable.
\end{proof}

\begin{lemma}\label{independence of initial condition}
For $1\leq k \leq 2n$,
$R_k$ does not depend on the parameters $\mu_1,\mu_2$ and $\omega$.
\end{lemma}

\begin{proof}
Note that $P_1=U_1^*U_1$ does not depend on the parameters. So $R_1=\delta_b^4F*P_1*F$ does not depend on the parameters.
By the second principal of mathematical induction,
for $k=1,2,\cdots,n-1$, assume that $R_i$, for any $i\leq k$, does not depend on the parameters.
Note that
$$P_{k+1}=U_k^*U_k$$
$$=\rho(R_k+U_k)^*\rho(R_k+U_k)$$
$$=\rho(R_k)^*\rho(R_k)+\rho(U_k)^*\rho(U_k)$$
$$=\cdots$$
$$=\rho(R_k)^*\rho(R_k)+\rho^2(R_{k-1})^*\rho^2(R_{k-1})+\cdots +\rho^k(R_1)^*\rho^k(R_1)+\rho^k(U_1)^*\rho^k(U_1)^k.$$
Moreover $\rho^k(U_1)^*\rho^k(U_1)$ does not depend on the parameters. So $P_{k+1}$ does not depend on the parameters. Then $R_{k+1}=\delta_b^4F*P_{k+1}*F$ does not depend on the parameters.
For $n+1\leq k\leq 2n$, the proof is similar.
\end{proof}

To compute $R_k$, we may fix the parameters as $\mu_1=\mu_2=\omega=1$ first.
Now let us compute the coefficients of loops in $R=R_{111}$.

\begin{definition}
For a loop $l\in\mathscr{G}_{2n,+}$, let us define $C_R(l)$ to be the coefficient of $l$ in $R=R_{111}$.
Let us define $C_P(l)$ to be the coefficient of $l$ in $P=\sum_{1\leq k \leq 2n} P_k$.
\end{definition}

If a loop $l'$ has a cusp point $b_{2i-1}$, then we may substitute $b_{2i-1}$ by $a_{2i-1}$ to obtain another loop $l$. By Proposition(\ref{F inv})(\ref{F ++}) and Lemma \ref{pro3}, we have $C_R(l')$ is determined by $C_R(l)$.
Essentially we only need to compute the coefficients of loops whose points are just $a_j's$.
Their relations are given by the following lemma.

\begin{lemma}\label{coefficient R P}
For a loop $l_1'\in \mathscr{G}_{2n+}$,
$l_1'=[x_0\cdots b_{2i-1}\cdots x_{2n}\rangle \langle x_{2n}\cdots x_0]$, we have
$$C_R(l_1')=-\delta_b^{\frac{1}{2}}C_R(l_1),$$
where $l_1=[x_0\cdots a_{2i-1}\cdots x_{2n}\rangle \langle x_{2n}\cdots x_0]$ is the loop replacing the given point $b_{2i-1}$ by $a_{2i-1}$ in $l_1'$.

For a loop $l_2\in A_k^+\mathscr{G}_{2n+}A_k^+$, $l_2=[a_{2k-1}\cdots a_{2m-1}\rangle\langle a_{2m-1}\cdots a_{2k-1}]$, we have
$$C_R(l_2)=
\left\{
\begin{array}{ll}
\delta_b^2 C_P(l_2),
&\text{when the middle point $a_{2m-1}$ is a flat point};\\
C_P(l_2)-C_P(l_2'),
&\text{when the middle point $a_{2m-1}$ is a cusp point},\\
\end{array}
\right.
$$
where $l_2'=[a_{2k-1}\cdots b_{2m-1}\rangle\langle b_{2m-1}\cdots a_{2k-1}]$ is the loop replacing the middle point $a_{2m-1}$ by $b_{2m-1}$ in $l_2$.
\end{lemma}

\begin{proof}
For a loop $l_1'\in \mathscr{G}_{2n+}$,
$$l_1'=[x_0\cdots x_{2k-1}b_{2i-1}x_{2k+1}\cdots x_{2n}\rangle \langle x_{2n}x_{2n+1}\cdots x_{4n-1}x_0],$$
we take $l_1$ to be the loop
$$l_1=[x_0\cdots x_{2k-1}a_{2i-1}x_{2k+1}\cdots x_{2n}\rangle \langle x_{2n}x_{2n+1}\cdots x_{4n-1}x_0].$$
Assume that
$$l_0'=[b_{2i-1}x_{2k+1}\cdots x_{2n+2k}\rangle \langle x_{2n+2k}\cdots x_{4n-1}x_0\cdots x_{2k-1}b_{2i-1}]$$
and
$$l_0=[a_{2i-1}x_{2k+1}\cdots x_{2n+2k}\rangle \langle x_{2n+2k}\cdots x_{4n-1}x_0\cdots x_{2k-1}a_{2i-1}].$$
Then the coefficient of $l_0'$ in $\rho^{-k}(R)$ is $$\sqrt{\frac{\lambda'(x_0)\lambda'(x_{2n})}{\lambda'(b_{2i-1})\lambda'(x_{2n+2k})}}C_R(l_1');$$
and the coefficient of $l_0$ in $\rho^{-k}(R)$ is $$\sqrt{\frac{\lambda'(x_0)\lambda'(x_{2n})}{\lambda'(a_{2i-1})\lambda'(x_{2n+2k})}}C_R(l_1).$$
By Proposition \ref{F inv}, the linear space spanned by $l_0,l_0'$ is invariant under the coproduct of $F$ on the left side.
By Lemma \ref{pro3}, we have
$$F*(\rho^{-k}(R))=\rho^{-k}(R).$$
So
$$\sqrt{\frac{\lambda'(x_0)\lambda'(x_{2n})}{\lambda'(b_{2i-1})\lambda'(x_{2n+2k})}}C_R(l_1')l_0'
+\sqrt{\frac{\lambda'(x_0)\lambda'(x_{2n})}{\lambda'(a_{2i-1})\lambda'(x_{2n+2k})}}C_R(l_1)l_0
$$
is invariant under the coproduct of $F$ on the left side.
By Proposition \ref{F ++}, we have
$$\sqrt{\frac{\lambda'(x_0)\lambda'(x_{2n})}{\lambda'(b_{2i-1})\lambda'(x_{2n+2k})}}C_R(l_1')
+\delta_b\sqrt{\frac{\lambda'(x_0)\lambda'(x_{2n})}{\lambda'(a_{2i-1})\lambda'(x_{2n+2k})}}C_R(l_1)=0
.$$
Thus $$C_R(l_1')=-\delta_b^{\frac{1}{2}}C_R(l_1).$$

For a loop $l_2\in A_k^+\mathscr{G}_{2n+}A_k^+$, $l_2=[a_{2k-1}\cdots a_{2m-1}\rangle\langle a_{2m-1}\cdots a_{2k-1}]$, we have
$$C_R(l_2)=\frac{tr(R l_2^*)}{tr(l_2 l_2^*)}=\frac{tr(R_k l_2^*)}{tr(l_2 l_2^*)}
=\delta_b^4\frac{tr((F*P_k*F) l_2^*)}{tr(l_2l_2^*)}
.$$
Note that
$$tr((F*P_k*F) l_2^*)=tr(P_k (F*l_2^**F))$$ by a diagram isotopy.
So
$$C_R(l_2)=\delta_b^4\frac{tr(P_k (F*l_2^**F))}{tr(l_2l_2^*)}
=\delta_b^4\frac{tr(P_k (F*l_2*F)^*)}{tr(l_2l_2^*)}.$$

If $a_{2m-1}$ is a flat point, then $l_2*F=l_2$, by a direct computation.
By Proposition \ref{F ++}, we have
$$F*l_2=\delta_b^{-2}l_2-\delta_b^{-1}\beta(l_2).$$
So
$$C_R(l_2)=\delta_b^4\delta_b^{-2}\frac{tr(P_k l_2^*)}{tr(l_2l_2^*)}=\delta_b^2C_P(l_2).$$

If $a_{2m-1}$ is a cusp point, then
$$l_2*F=\delta_b^{-2}l_2-\delta_b^{-1}l_2',$$
by Proposition \ref{F ++} and an $180^\circ $ rotation,
where $l_2'=[a_{2k-1}\cdots b_{2m-1}\rangle\langle b_{2m-1}\cdots a_{2k-1}]$ is the loop replacing the middle point $a_{2m-1}$ by $b_{2m-1}$ in $l_2$.
Again by Proposition \ref{F ++}, we have
$$F*l_2*F=
\delta_b^{-4}l_2-\delta_b^{-3}\beta(l_2)
-\delta_b^{-3}l_2'+\delta_b^{-2}\beta(l_2').$$
So
$$C_R(l_2)=\delta_b^4\delta_b^{-4}\frac{tr(P_k l_2^*)}{tr(l_2l_2^*)}
-\delta_b^4\delta_b^{-3}\frac{tr(P_k l_2'^*)}{tr(l_2l_2^*)}.$$
Observe that $$tr(l_2l_2^*)=\delta_b tr(l_2'l_2'^*).$$
Therefore
$$C_R(l_2)=C_P(l_2)-C_P(l_2').$$
\end{proof}

Note that $P_k=U_k^*U_k$, to compute the coefficient of a loop in $P_k$ we only need the coefficients of loops in $U_k$. They are determined by the coefficients of loops in $R_{k-1}$.

\begin{definition}
For $1\leq k \leq n$, let us define $[a_{2k-1},y\rangle$ to be the set of all length $2n$ pathes from $a_{2k-1}$ to $y$ starting with $a_{2k-1}a_{2k-2}$.
For a path $\eta=[z_0z_1\cdots z_{k-1} z_{k}\rangle$, let us define $\eta^*$ to be the path $\langle z_{k},z_{k-1},\cdots,z_1,z_0]$.
\end{definition}

\begin{lemma}\label{coefficient P U}
For a loop $\eta_1\eta_2^*\in A_k^+\mathscr{G}_{2n,+}A_k^+$ whose first point is $a_{2k-1}$, suppose its middle point is $y$. Then we have
$$C_P(\eta_1\eta_2^*)=\sum_{\eta\in[a_{2k-1}y\rangle}C_R(\eta_1\eta^*) C_R(\eta\eta_2^*).$$
\end{lemma}

\begin{proof}
Note that a length $2n$ path ${\eta\in[a_{2k-1}y\rangle}$ starts with $a_{2k-1}a_{2k-2}$,
so $C_R(\eta^*\eta_2)$ is the coefficient of $\eta^*\eta_2$ in $U_k$
and $C_R(\eta_1\eta^*)$ is the coefficient of $\eta_1\eta^*$ in $U_k^*$.
Then the statement follows from the fact $P_k=U_k^*U_k$.
\end{proof}

When the initial condition $\mu_1=\mu_2=\omega=1$ is fixed, given a loop $$l=[a_{k_1}a_{2n+k_2}a_{2n-k_3}\cdots a_{2n+k_{2t}}a_{k_1}], ~\text{for} ~1\leq k_1,k_2,\cdots,k_{2t} \leq 2n-1,$$ we may compute $C_R(l)$ by repeating Lemma(\ref{coefficient R P})(\ref{coefficient P U}). A significant fact is that the computation only depends on $k_1,k_2,\cdots,k_{2t}$, in other words, $C_R(l)$ is independent of $n$.
We list all the coefficients for $k_1\leq 7$ in the Appendix. This is enough to rule out the $4_{th}$ fish by comparing the coefficients $C_R([a_5a_9a_5a_9a_5])$ and $C_R([a_7a_{11}a_7a_{11}a_7])$.
It is possible to rule out finitely many Bisch-Haagerup fish graphs by computing more coefficients.
To rule out the $n_{th}$ Bisch-Haagerup fish graph, for all $n\geq 4$, we need formulas for the coefficients of two families of loops which do not match the proposition(4'). Then only the first three Bisch-Haagerup fish graphs are the principal graphs of subfactors.

\begin{lemma}\label{d3}
$$C_R([a_{2k-1}a_{2n+2k-1} a_{2k-1}])=\delta_b^{-3}, \forall ~ 1\leq k \leq n.$$
\end{lemma}

\begin{proof}
For $1\leq k \leq n$, by Lemma \ref{coefficient R P}, we have
$$C_R([a_{2k-1}a_{2n+2k-1} a_{2k-1}])
=C_P([a_{2k-1}a_{2n+2k-1} a_{2k-1}])-C_P([a_{2k-1}b_{2n+2k-1} a_{2k-1}]).$$
By Lemma \ref{coefficient P U}, we have
$$C_P([a_{2k-1}a_{2n+2k-1} a_{2k-1}])$$
$$=C_R([a_{2k-1}\cdots a_{4n-1}a_0\cdots a_{2k-1} a_{2k-2}]) C_R([a_{2k-1}a_{2k-2}\cdots a_0a_{4n-1}\cdots a_{2k}]),$$
because $$[a_{2k-1}a_{2n+2k-1} a_{2k-1}]=[a_{2k-1}a_{2n+2k-1}\rangle \langle a_{2n+2k-1} a_{2k-1}],$$
and $a_{2k-1} a_{2k-2}\cdots a_0a_{4n-1}\cdots a_{2n+2k-1}$ is the unique path in $[a_{2k-1},a_{2n+2k-1}\rangle$.
Note that
$$[a_{2k-1}\cdots a_{4n-1}a_0\cdots a_{2k-1} a_{2k-2}]^*=[a_{2k-1}a_{2k-2}\cdots a_0a_{4n-1}\cdots a_{2k}],$$
and $R=R^*$, so
$$C_R([a_{2k-1}\cdots a_{4n-1}a_0\cdots a_{2k-1} a_{2k-2}])=\overline{C_R([a_{2k-1}a_{2k-2}\cdots a_0a_{4n-1}\cdots a_{2k}])}.$$
Observe that $$\rho[a_1a_0a_{4n-1}\cdots a_2]$$
$$=\sqrt{\frac{\lambda'(a_1)\lambda'(a_{2n+1})}{\lambda'(a_{2k-1})\lambda'(a_{2n+2k-1})}}   [a_{2k-1}a_{2k-2}\cdots a_1a_0a_{4n-1}\cdots a_{2k}]$$
$$=[a_{2k-1}a_{2k-2}\cdots a_1a_0a_{4n-1}\cdots a_{2k}].$$
and $\rho(R)=R$, (we assumed that $\mu_1=\mu_2=\omega=1$,) so
$$C_R([a_{2k-1}a_{2k-2}\cdots a_0a_{4n-1}\cdots a_{2k}])=C_R([a_1a_0a_{4n-1}\cdots a_2])=\delta_b^{-1.5}.$$
Then
$$C_P([a_{2k-1}a_{2n+2k-1} a_{2k-1}])=\delta_b^{-3}.$$
On the other hand,
$$[a_{2k-1}b_{2n+2k-1} a_{2k-1}]=[a_{2k-1}b_{2n+2k-1}\rangle \langle b_{2n+2k-1} a_{2k-1}],$$
but there is no path in $[a_{2k-1},b_{2n+2k-1}]$, so
$$C_P([a_{2k-1}b_{2n+2k-1} a_{2k-1}])=0.$$
Then
$$C_R([a_{2k-1}a_{2n+2k-1} a_{2k-1}])=\delta_b^{-3}.$$
\end{proof}

\begin{lemma}\label{d5}
$$C_R([a_{2k-1}a_{2n+1} a_{2n-2k+3} a_{2n+1}a_{2k-1}])=\delta_b^{-5}, \forall ~ 2\leq k \leq n;$$
$$C_R([a_{2k-1}a_{2n+1}a_{2n-1} a_{2n+2k-3} a_{2k-1}])=\delta_b^{-5.5}, \forall ~ 3\leq k \leq n;$$
$$C_R([a_{2k-1} a_{2n+2k-3}a_{2n-1}a_{2n+1} a_{2k-1}])=\delta_b^{-5.5}, \forall ~ 3\leq k \leq n.$$
\end{lemma}

\begin{proof}
For $2\leq k \leq n$, by Lemma \ref{coefficient R P}, we have
$$C_P([a_{2k-1}a_{2n+1}a_{2n-2k+3} a_{2n+1}a_{2k-1}])
=C_P([a_{2k-1}a_{2n+1}a_{2n-2k+3}\rangle \langle a_{2n-2k+3} a_{2n+1}a_{2k-1}])$$
$$=C_R([a_{2k-1}a_{2n+1}a_{2n-2k+3}\rangle \langle a_{2n-2k+3} a_1 a_{2k-1}])C_R([a_{2k-1} a_1 a_{2n-2k+3}\rangle \langle a_{2n-2k+3}a_{2n+1}a_{2k-1}])$$
$$+C_R([a_{2k-1}a_{2n+1}a_{2n-2k+3}\rangle \langle a_{2n-2k+3} b_1 a_{2k-1}])C_R([a_{2k-1} b_1 a_{2n-2k+3}\rangle \langle a_{2n-2k+3}a_{2n+1}a_{2k-1}])$$
By Lemma \ref{coefficient R P}, we have
$$C_R([a_{2k-1} b_1 a_{2n-2k+3}\rangle \langle a_{2n-2k+3}a_{2n+1}a_{2k-1}]=-\delta_b^{0.5}C_R([a_{2k-1} a_1 a_{2n-2k+3}\rangle \langle a_{2n-2k+3}a_{2n+1}a_{2k-1}].$$
So the formula is simplified as
$$C_P([a_{2k-1}a_{2n+1}a_{2n-2k+3} a_{2n+1}a_{2k-1}])$$
$$=\delta_b^{2}C_R([a_{2k-1}a_{2n+1}a_{2n-2k+3}\rangle \langle a_{2n-2k+3} a_1 a_{2k-1}])C_R([a_{2k-1} a_1 a_{2n-2k+3}\rangle \langle a_{2n-2k+3}a_{2n+1}a_{2k-1}]),$$
where $\delta_b^2$ is given by $1+(-\delta_b^{0.5})^2=\delta_b^2$.

We see that the cusp point of a path in $[a_{2k-1} a_{2n-2k+3}\rangle$ could be $a_1$ or $b_1$, but we may ignore the path with the cusp point $b_1$ by adding a factor $\delta_b^2$.

While
$$C_R([a_{2k-1} a_1 a_{2n+1}a_{2k-1}])$$
$$=\sqrt{\frac{\lambda'(a_1)\lambda'(a_{2n+1})}{\lambda'(a_{2k-1})\lambda'(a_{2n-2k+3})}}C_R([a_1 a_{2n+1} a_1])=\delta_b^{-0.5}\delta_b^{-3}=\delta_b^{-3.5}.$$
So
$$C_P([a_{2k-1}a_{2n+1}a_{2n-2k+3} a_{2n+1}a_{2k-1}])=\delta_b^2(\delta_b^{-3.5})^2=\delta_b^{-5}.$$

On the other hand, there is no path in $[a_{2k-1}b_{2n-2k+3}\rangle$, so
$$C_P([a_{2k-1}a_{2n+1}b_{2n-2k+3} a_{2n+1}a_{2k-1}])=0.$$
Then
$$C_R([a_{2k-1}a_{2n+1}a_{2n-2k+3} a_{2n+1}a_{2k-1}])=\delta_b^{-5}.$$

For the formula $C_R([a_{2k-1}a_{2n+1}a_{2n-1} a_{2n+2k-3} a_{2k-1}])$,
when $k=3$, we have
$$C_R([a_{3}a_{2n+1} a_{2n-1} a_{2n+1} a_{3}])=\delta_b^{-5}.$$

When $k\geq 3$, by Lemma \ref{coefficient R P}, we have
$$C_P([a_{2k-1}a_{2n+1}a_{2n-1} a_{2n+2k-3} a_{2k-1}])$$
$$=C_P([a_{2k-1}a_{2n+1}a_{2n-1}a_{2n+2k-5}\rangle \langle a_{2n+2k-5} a_{2n+2k-3} a_{2k-1}])$$
$$=\delta_b^2C_R([a_{2k-1}a_{2n+1}a_{2n-1}a_{2n+2k-5}\rangle \langle a_{2n+2k-5} a_{2k-3} a_{2k-1}])$$
$$\times C_R([a_{2k-1} a_{2k-3} a_{2n+2k-5} \rangle \langle a_{2n+2k-5} a_{2n+2k-3} a_{2k-1}]),$$

where the factor $\delta_b^2$ comes from the choice the cusp point $a_{2k-3}$. While

$$C_R([a_{2k-1}a_{2n+1}a_{2n-1} a_{2n+2k-3} a_{2k-1}])$$
$$=C_R([a_{2k-1}a_{2n+1}a_{2n-1}a_{2n+2k-5}\rangle \langle a_{2n+2k-5} a_{2k-3} a_{2k-1}])$$
$$=\sqrt{\frac{\lambda'(a_{2k-3})\lambda'(a_{2n+2k-7})}{\lambda'(a_{2k-1})\lambda'(a_{2n+2k-5})}}C_R([a_{2k-3}a_{2n+1}a_{2n-1}a_{2n+2k-5}a_{2k-3}])$$

\begin{equation*}
  =\left\{
   \begin{aligned}
   &\delta_b^{-0.5}C_R([a_3a_{2n+1}a_{2n-1}a_{2n+1}a_3])=\delta_b^{-5.5}& ~when~ k=3;  \\
   &C_R([a_{2k-3}a_{2n+1}a_{2n-1}a_{2n+2k-5}a_{2k-3}])& ~when ~k\geq 4. \\
   \end{aligned}
   \right.
\end{equation*}

$$C_R([a_{2k-1}a_{2k-3}a_{2n+2k-3}a_{2k-1}])$$
$$=\sqrt{\frac{\lambda'(a_{2k-3})\lambda'(a_{2n+2k-3})}{\lambda'(a_{2k-1})\lambda'(a_{2n+2k-5})}}C_R([a_{2k-3}a_{2n+2k-3}a_{2k-3}])=\delta_b^{-1}\delta_b^{-3}=\delta_b^{-4}.$$

Note that the middle point $a_{2n+2k-5}$ is a flat point, by Lemma \ref{coefficient R P}, we have
$$C_R([a_{2k-1}a_{2n+1}a_{2n-1} a_{2n+2k-3} a_{2k-1}])=\delta_b^2 C_P([a_{2k-1}a_{2n+1}a_{2n-1} a_{2n+2k-3} a_{2k-1}]).$$
Then $C_R([a_{2k-1}a_{2n+1}a_{2n-1} a_{2n+2k-3} a_{2k-1}])=\delta_b^{-5.5}$ when $k=3$;

$C_R([a_{2k-1}a_{2n+1}a_{2n-1} a_{2n+2k-3} a_{2k-1}])=C_R([a_{2k-3}a_{2n+1}a_{2n-1}a_{2n+2k-5}a_{2k-3}])$ when $k\geq 4$.

Therefore we have
$C_R([a_{2k-1}a_{2n+1}a_{2n-1} a_{2n+2k-3} a_{2k-1}])=\delta_b^{-5.5}$ inductively, for $3\leq k\leq n$.

Take the adjoint, we have $C_R([a_{2k-1} a_{2n+2k-3}a_{2n-1}a_{2n+1} a_{2k-1}])=\delta_b^{-5.5}.$

\end{proof}

\begin{lemma}\label{d8}
$$C_R([a_{2k-1}a_{2n+1}a_{2n-1} a_{2n+2k-5}a_{2n-1}a_{2n+1}a_{2k-1}])=-\delta_b^{-8}, \forall ~3\leq k \leq n$$
\end{lemma}

\begin{proof}
For $3\leq k \leq n$, by Lemma \ref{coefficient P U}, we have
$$C_P([a_{2k-1}a_{2n+1}a_{2n-1}a_{2n+2k-5}a_{2n-1}a_{2n+1}a_{2k-1}])$$
$$=C_P([a_{2k-1}a_{2n+1}a_{2n-1}a_{2n+2k-5}\rangle \langle a_{2n+2k-5}a_{2n-1}a_{2n+1}a_{2k-1}])$$
$$=\delta_b^2 C_R([a_{2k-1}a_{2n+1}a_{2n-1}a_{2n+2k-5}\rangle \langle a_{2n+2k-5} a_{2k-3} a_{2k-1}])$$
$$\times C_R([a_{2k-1} a_{2k-3} a_{2n+2k-5}\rangle \langle a_{2n+2k-5}a_{2n-1}a_{2n+1}a_{2k-1}]),$$
where $\delta_b^2$ is given by the choice of $a_{2k-3}$.

On the other hand
$$C_P([a_{2k-1}a_{2n+1}a_{2n-1}b_{2n+2k-5}a_{2n-1}a_{2n+1}a_{2k-1}])$$
$$=C_P([a_{2k-1}a_{2n+1}a_{2n-1}b_{2n+2k-5}\rangle \langle b_{2n+2k-5}a_{2n-1}a_{2n+1}a_{2k-1}])$$
$$=\delta_b^2 C_R([a_{2k-1}a_{2n+1}a_{2n-1}b_{2n+2k-5}\rangle \langle b_{2n+2k-5} a_{2k-3} a_{2k-1}])$$
$$\times C_R([a_{2k-1} a_{2k-3} b_{2n+2k-5}\rangle \langle b_{2n+2k-5}a_{2n-1}a_{2n+1}a_{2k-1}]),$$
where $\delta_b^2$ is given by the choice of $a_{2k-3}$.

Note that
$$C_R([a_{2k-1}a_{2n+1}a_{2n-1} b_{2n+2k-5} a_{2k-3} a_{2k-1}])$$
$$=\delta_b^{-1}C_R([a_{2k-1}a_{2n+1}a_{2n-1}a_{2n+2k-5} a_{2k-3} a_{2k-1}]);$$
$$C_R([a_{2k-1} a_{2k-3} b_{2n+2k-5} a_{2n-1}a_{2n+1}a_{2k-1}])$$
$$=\delta_b^{-1} C_R([a_{2k-1} a_{2k-3} a_{2n+2k-5} a_{2n-1}a_{2n+1}a_{2k-1}]).$$

By Lemma \ref{coefficient R P}, we have
$$C_R([a_{2k-1}a_{2n+1}a_{2n-1}a_{2n+2k-5}a_{2n-1}a_{2n+1}a_{2k-1}])$$
$$=C_P([a_{2k-1}a_{2n+1}a_{2n-1}a_{2n+2k-5}a_{2n-1}a_{2n+1}a_{2k-1}])$$
$$-C_P([a_{2k-1}a_{2n+1}a_{2n-1}b_{2n+2k-5}a_{2n-1}a_{2n+1}a_{2k-1}])$$
$$=\delta_b^{-1}\delta_b^2 C_R([a_{2k-1}a_{2n+1}a_{2n-1}a_{2n+2k-5}\rangle \langle a_{2n+2k-5} a_{2k-3} a_{2k-1}])$$
$$\times C_R([a_{2k-1} a_{2k-3} a_{2n+2k-5}\rangle \langle a_{2n+2k-5}a_{2n-1}a_{2n+1}a_{2k-1}]).$$
where $-\delta_b$ is given by $1-(\delta_b^{-1})^2=-\delta_b$.

We see that if the middle point is a cusp point, and both $a_{2n+2k-5}$ and $b_{2n+2k-5}$ contribute to the middle point of a loop in the multiplication, then we may ignore the loop with middle point $b_{2n+2k-5}$ by adding a factor $-\delta_b$.

While
$$C_R([a_{2k-1} a_{2k-3} a_{2n+2k-5} a_{2n-1}a_{2n+1}a_{2k-1}]$$
$$=\sqrt{\frac{\lambda'(a_{2k-3})\lambda'(a_{2n+2k-7})}{\lambda'(a_{2k-1})\lambda'(a_{2n+2k-5})}}C_R([a_{2k-3}a_{2n+2k-5}a_{2n-1}a_{2n+1}a_{2k-1}])$$

\begin{equation*}
  =\left\{
   \begin{aligned}
   &\delta_b^{-0.5}C_R([a_{3}a_{2n+1}a_{2n-1}a_{2n+1}a_{3}])=\delta_b^{-5.5}& ~when~ k=3;  \\
   &C_R([a_{2k-3}a_{2n+2k-5}a_{2n-1}a_{2n+1}a_{2k-1}])=\delta_b^{-5.5}& ~when ~k\geq 4. \\
   \end{aligned}
   \right.
\end{equation*}

So $$C_R([a_{2k-1}a_{2n+1}a_{2n-1}a_{2n+2k-5}a_{2n-1}a_{2n+1}a_{2k-1}])=-\delta_b\delta_b^2(\delta_b^{-5.5})^2=-\delta_b^{-8}, ~\forall~ k\geq3.$$

\end{proof}

\begin{lemma}\label{d22}
For $5\leq k \leq n$,
we assume that

$\eta_{k1}=[a_{2k-1}a_{2n+2k-5}a_{2n+2k-9}\rangle;$

$\eta_{k2}=[a_{2k-1}a_{2n+1}a_{2n-1}a_{2n+2k-7}a_{2n+2k-9}\rangle;$

$\tilde{\eta}_{k1}=[a_{2k-1}a_{2k-5}a_{2n+2k-9} \rangle ;$

$\tilde{\eta}_{k2}=[a_{2k-1}a_{2k-3}a_{2n+1}a_{2n-1}a_{2n+2k-9} \rangle;$

Then

$\begin{bmatrix}
C_R(\eta_{k1}\tilde{\eta}_{k1}^*)&C_R(\eta_{k1}\tilde{\eta}_{k2}^*)\\
C_R(\eta_{k2}\tilde{\eta}_{k1}^*)&C_R(\eta_{k2}\tilde{\eta}_{k2}^*)\\
\end{bmatrix}$
=
$\begin{bmatrix}
\delta_b^{-5}&\delta_b^{-6.5}\\
\delta_b^{-6.5}&-\delta_b^{-9}\\
\end{bmatrix}$
\end{lemma}

\begin{proof}
$$C_R(\eta_{k1}\tilde{\eta}_{k1}^*)=C_R([a_{2k-1}a_{2n+2k-5}a_{2k-5}a_{2k-1}])$$
$$=\delta_b^{-2}C_R([a_{2k-5}a_{2n+2k-5}a_{2k-5}])=\delta_b^{-5}, ~\text{by Lemma \ref{d3}};$$
$$C_R(\eta_{k1}\tilde{\eta}_{k2}^*)=C_R([a_{2k-1}a_{2n+2k-5}a_{2n-1}a_{2n+1}a_{2k-3}a_{2k-1}])$$
$$=\delta_b^{-1}C_R([a_{2k-3}a_{2n+2k-5}a_{2n-1}a_{2n+1}a_{2k-3}])=\delta_b^{-6.5}, ~\text{by Lemma \ref{d5}};$$
$$C_R(\eta_{k2}\tilde{\eta}_{k1}^*)=C_R([a_{2k-1}a_{2n+1}a_{2n-1}a_{2n+2k-7}a_{2k-5}a_{2k-1}])$$
$$=\delta_b^{-1}C_R([a_{2k-5}a_{2n+1}a_{2n-1}a_{2n+2k-7}a_{2k-5}])=\delta_b^{-6.5}, ~\text{by Lemma \ref{d5}};$$
$$C_R(\eta_{k2}\tilde{\eta}_{k2}^*)=C_R([a_{2k-1}a_{2n+1}a_{2n-1}a_{2n+2k-7}a_{2n-1}a_{2n+1}a_{2k-3}a_{2k-1}])$$
$$=\delta_b^{-1}C_R([a_{2k-3}a_{2n+1}a_{2n-1}a_{2n+2k-7}a_{2n-1}a_{2n+1}a_{2k-3}])=-\delta_b^{-9}, ~\text{by Lemma \ref{d8}}.$$
\end{proof}

\begin{lemma}\label{etak}
For $5\leq k\leq n$,
we assume that

$\eta_{k3}=[a_{2k-1}a_{2n+1}a_{2n-1}a_{2n+1}a_{2n-1}a_{2n+2k-9}\rangle$;

$\eta_{k4}=[a_{2k-1}a_{2n+3}a_{2n-1}a_{2n+2k-9}\rangle$;

$\eta_{k5}=[a_{2k-1}a_{2n+1}a_{2n-3}a_{2n+2k-9}\rangle$.

Then

\begin{center}
\begin{tabular}{|c|c|c|c|c|c|}

\hline

$k$&$5l+5$&$5l+6$&$5l+7$&$5l+8$&$5l+9$\\

\hline
$C_R(\eta_{k1}\eta_{k3}^*)$& $0$ & $0$&$\delta_b^{-8}$ & $-\delta_b^{-9}$ &$ \delta_b^{-8}$ \\
\hline
$C_R(\eta_{k2}\eta_{k3}^*)$& $-\delta_b^{-10.5}$ &$\delta_b^{-9.5} $&$ -\delta_b^{-10.5} $&$ \delta_b^{-11.5} $&$ \delta_b^{-11.5}$\\
\hline
$C_R(\eta_{k1}\eta_{k4}^*)$& $\delta_b^{-5.5}$&$0 $& $\delta_b^{-6.5}$&$\delta_b^{-6.5} $&$0$ \\
\hline
$C_R(\eta_{k2}\eta_{k4}^*)$&$0$ &$0$ &$ \delta_b^{-8} $&$ -\delta_b^{-9}$&$ \delta_b^{-8}$\\
\hline
$C_R(\eta_{k1}\eta_{k5}^*)$&$ 0 $ &$ \delta_b^{-5.5} $ &$ 0 $&$ \delta_b^{-6.5} $&$ \delta_b^{-6.5} $ \\
\hline
$C_R(\eta_{k2}\eta_{k5}^*)$&$ \delta_b^{-8} $&$ 0 $&$0 $&$\delta_b^{-8}  $&$ -\delta_b^{-9} $\\
\hline

$C_R(\eta_{k3}\eta_{k3}^*)$& $\delta_b^{-13}$ & $-\delta_b^{-12}$&$-\delta_b^{-12}$ & $\delta_b^{-13}$ &$-\delta_b^{-14}$ \\
\hline
$C_R(\eta_{k3}\eta_{k4}^*)$& $0$ &$0$&$0$&$\delta_b^{-9.5} $&$-\delta_b^{-10.5}$\\
\hline
$C_R(\eta_{k3}\eta_{k5}^*)$& $\delta_b^{-9.5}$&$0 $& $0$&$0$&$-\delta_b^{-10.5}$ \\
\hline
$C_R(\eta_{k4}\eta_{k4}^*)$&$-\delta_b^{-8}$ &$\delta_b^{-7}$ &$-\delta_b^{-8} $&$0$&$0$\\
\hline
$C_R(\eta_{k4}\eta_{k5}^*)$&$0$&$0$&$0$&$0$&$\delta_b^{-8}$ \\
\hline
$C_R(\eta_{k5}\eta_{k5}^*)$&$0$&$-\delta_b^{-8}$&$\delta_b^{-7}$&$-\delta_b^{-8}$&$0$\\
\hline

\hline

\end{tabular}
.
\end{center}
\end{lemma}

\begin{proof}
For $5\leq k \leq n$, $i=3,4,5$, we assume that

$$\begin{bmatrix}
\alpha_{ki}\\
\beta_{ki}\\
\end{bmatrix}
=\begin{bmatrix}
C_R(\eta_{k1}\eta_{ki}^*)\\
C_R(\eta_{k2}\eta_{ki}^*)\\
\end{bmatrix}.$$

Then
$$\begin{bmatrix}
\alpha_{ki}\\
\beta_{ki}\\
\end{bmatrix}
=\begin{bmatrix}
C_R(\eta_{k1}\eta_{ki}^*)\\
C_R(\eta_{k2}\eta_{ki}^*)\\
\end{bmatrix}
=\delta_b^2\begin{bmatrix}
C_R(\eta_{k1}\tilde{\eta}_{k1}^*)&C_R(\eta_{k1}\tilde{\eta}_{k2}^*)\\
C_R(\eta_{k2}\tilde{\eta}_{k1}^*)&C_R(\eta_{k2}\tilde{\eta}_{k2}^*)\\
\end{bmatrix}
\begin{bmatrix}
\delta_b^2C_R(\tilde{\eta}_{k1}\eta_{ki}^*)\\
\delta_b^6C_R(\tilde{\eta}_{k2}\eta_{ki}^*)\\
\end{bmatrix}
$$

Furthermore we have
$$
C_R(\tilde{\eta}_{k1}\eta_{ki}^*)
=
C_R(\rho^{-2}(\tilde{\eta}_{(k-2)1}\eta_{(k-2)i}^*))
=
C_R(\tilde{\eta}_{(k-2)1}\eta_{(k-2)i}^*)
=
\alpha_{(k-2)i}
, ~\text{when} ~k\geq7.
$$
$$
C_R(\tilde{\eta}_{k2}\eta_{k2}^*)
=
C_R(\rho^{-1}(\tilde{\eta}_{(k-1)2}\eta_{(k-1)i}^*))
=
C_R(\tilde{\eta}_{(k-1)2}\eta_{(k-1)i}^*)
=
\beta_{(k-1)i}
, ~\text{when} ~k\geq6.
$$

So
$$\begin{bmatrix}
\alpha_{ki}\\
\beta_{ki}\\
\end{bmatrix}
=\delta_b^2
\begin{bmatrix}
\delta_b^{-5}&\delta_b^{-6.5}\\
\delta_b^{-6.5}&-\delta_b^{-9}\\
\end{bmatrix}
\begin{bmatrix}
\delta_b^{2}\alpha_{(k-2)i}\\
\delta_b^{6}\beta_{(k-1)i}\\
\end{bmatrix}
=
\begin{bmatrix}
\delta_b^{-1}&\delta_b^{1.5}\\
\delta_b^{-2.5}&-\delta_b^{-1}\\
\end{bmatrix}
\begin{bmatrix}
\alpha_{(k-2)i}\\
\beta_{(k-1)i}\\
\end{bmatrix}.
$$

Substituting $\beta_{ki}$ by $\alpha_{ki}$, we have

$$\alpha_{(k+1)i}+\delta_b^{-1}\alpha_{ki}-\delta_b^{-1}\alpha_{(k-1)i}-\alpha_{(k-2)i}=0.$$ 

While $x^3+\delta_b^{-1}x^2-\delta_b^{-1}x-1=0$ has three roots $1,-q_b,-q_b^{-1}$.
So $$\alpha_{ki}=r_{1i}+r_{2i}(-q_b)^k+r_{3i}(-q_b)^{-k},$$ for some constant $r_{1i},r_{2i},r_{3i}$.
Then the periodicity is $5$.

Based on the results listed in the Appendix, the initial condition is

$$
\begin{bmatrix}
\alpha_{33}\\
\alpha_{43}\\
\beta_{43}\\
\end{bmatrix}
=
\begin{bmatrix}
C_R(\tilde{\eta}_{51}\eta_{53}^*)\\
C_R(\tilde{\eta}_{61}\eta_{53}^*)\\
C_R(\tilde{\eta}_{52}\eta_{53}^*)\\
\end{bmatrix}
=
\begin{bmatrix}
-\delta_b^{-9}\\
\delta_b^{-8}\\
\delta_b^{-11.5}\\
\end{bmatrix};$$
$$
\begin{bmatrix}
\alpha_{34}\\
\alpha_{44}\\
\beta_{44}\\
\end{bmatrix}
=
\begin{bmatrix}
C_R(\tilde{\eta}_{51}\eta_{54}^*)\\
C_R(\tilde{\eta}_{61}\eta_{54}^*)\\
C_R(\tilde{\eta}_{52}\eta_{54}^*)\\
\end{bmatrix}
=
\begin{bmatrix}
\delta_b^{-6.5}\\
0\\
\delta_b^{-8}\\
\end{bmatrix};$$
$$\begin{bmatrix}
\alpha_{35}\\
\alpha_{45}\\
\beta_{45}\\
\end{bmatrix}
=
\begin{bmatrix}
C_R(\tilde{\eta}_{51}\eta_{55}^*)\\
C_R(\tilde{\eta}_{61}\eta_{55}^*)\\
C_R(\tilde{\eta}_{52}\eta_{55}^*)\\
\end{bmatrix}
=
\begin{bmatrix}
\delta_b^{-6.5}\\
\delta_b^{-6.5}\\
-\delta_b^{-9}\\
\end{bmatrix}.
$$

For example,
$$\alpha_{33}=C_R(\tilde{\eta}_{51}\eta_{53}^*)=C_R([a_9a_{5}a_{2n+1}a_{2n-1}a_{2n+1}a_{2n-1}a_{2n+1}a_9)$$
$$=\delta_b^{-1}C_R([a_5a_{2n+1}a_{2n-1}a_{2n+1}a_{2n-1}a_{2n+1}a_5])=-\delta_b^{-9}.$$

The others are similar.

Then $\begin{bmatrix}
\alpha_{ki}\\
\beta_{ki}\\
\end{bmatrix}$
is obtained inductively. The result is listed in the following table

\begin{center}
\begin{tabular}{|c|c|c|c|c|c|}

\hline

$k$&$5l+5$&$5l+6$&$5l+7$&$5l+8$&$5l+9$\\

\hline

$\alpha_{k3}$& $0$ & $0$&$\delta_b^{-8}$ & $-\delta_b^{-9}$ &$ \delta_b^{-8}$ \\
\hline
$\beta_{k3}$& $-\delta_b^{-10.5}$ &$\delta_b^{-9.5} $&$ -\delta_b^{-10.5} $&$ \delta_b^{-11.5} $&$ \delta_b^{-11.5}$\\
\hline
$\alpha_{k4}$& $\delta_b^{-5.5}$&$0 $& $\delta_b^{-6.5}$&$\delta_b^{-6.5} $&$0$ \\
\hline
$\beta_{k4}$&$0$ &$0$ &$ \delta_b^{-8} $&$ -\delta_b^{-9}$&$ \delta_b^{-8}$\\
\hline
$\alpha_{k5}$&$ 0 $ &$ \delta_b^{-5.5} $ &$ 0 $&$ \delta_b^{-6.5} $&$ \delta_b^{-6.5} $ \\
\hline
$\beta_{k5}$&$ \delta_b^{-8} $&$ 0 $&$0 $&$\delta_b^{-8}  $&$ -\delta_b^{-9} $\\
\hline

\hline

\end{tabular}
.
\end{center}

For $5\leq k \leq n$, $3 \leq i,j\leq 5$, by Lemma(\ref{coefficient R P})(\ref{coefficient P U}), we have
$$C_R(\eta_{ki}\eta_{kj}^*)=
-\delta_b(\delta_b^2C_R(\eta_{ki}\tilde{\eta}_{k1}^*)C_R(\tilde{\eta}_{k1}\eta_{kj}^*)+\delta_b^6C_R(\eta_{ki}\tilde{\eta}_{k2}^*)C_R(\tilde{\eta}_{k2}\eta_{kj}^*))$$
$$+\delta_b^4C_R(\eta_{ki}\langle a_{2n+2k-9}a_{2n+2k-7} a_{2k-3}a_{2k-1}])C_R([a_{2k-1}a_{2k-3}a_{2n+2k-7}a_{2n+2k-9}\rangle\eta_{kj}^*)$$
$$
=-\delta_b(\delta_b^2C_R(\eta_{ki}\tilde{\eta}_{k1}^*)C_R(\tilde{\eta}_{k1}\eta_{kj}^*)+\delta_b^6C_R(\eta_{ki}\tilde{\eta}_{k2}^*)C_R(\tilde{\eta}_{k2}\eta_{kj}^*))$$
$$+\delta_b^4C_R(\eta_{(k+1)i}\tilde{\eta}_{(k+1)1}^*)C_R(\tilde{\eta}_{(k+1)1}\eta_{kj}^*).$$
$$=-\delta_b^3\alpha_{(k-2)i}\alpha_{(k-2)j}-\delta_b^{7}\beta_{(k-1)i}\beta_{(k-1)j}+\delta_b^4\alpha_{(k-1)i}\alpha_{(k-1)j}.$$

Then
\begin{center}
\begin{tabular}{|c|c|c|c|c|c|}

\hline

$k$&$5l+5$&$5l+6$&$5l+7$&$5l+8$&$5l+9$\\

\hline

$C_R(\eta_{k3}\eta_{k3}^*)$& $\delta_b^{-13}$ & $-\delta_b^{-12}$&$-\delta_b^{-12}$ & $\delta_b^{-13}$ &$-\delta_b^{-14}$ \\
\hline
$C_R(\eta_{k3}\eta_{k4}^*)$& $0$ &$0$&$0$&$\delta_b^{-9.5} $&$-\delta_b^{-10.5}$\\
\hline
$C_R(\eta_{k3}\eta_{k5}^*)$& $\delta_b^{-9.5}$&$0 $& $0$&$0$&$-\delta_b^{-10.5}$ \\
\hline
$C_R(\eta_{k4}\eta_{k4}^*)$&$-\delta_b^{-8}$ &$\delta_b^{-7}$ &$-\delta_b^{-8} $&$0$&$0$\\
\hline
$C_R(\eta_{k4}\eta_{k5}^*)$&$0$&$0$&$0$&$0$&$\delta_b^{-8}$ \\
\hline
$C_R(\eta_{k5}\eta_{k5}^*)$&$0$&$-\delta_b^{-8}$&$\delta_b^{-7}$&$-\delta_b^{-8}$&$0$\\
\hline

\hline

\end{tabular}
.
\end{center}

\end{proof}

\begin{lemma}\label{dfinal}
$$C_R(a_{2n-1}a_{4n-7}a_{2n-1}a_{2n+1}a_{2n-1}a_{2n+1}a_{2n-1}a_{2n+1}a_{2n-1})$$
\begin{equation*}
  =\left\{
   \begin{aligned}
   &-\delta_b^{-13.5} &~\text{when}~ n=20l+8;~  \\
   &-\delta_b^{-13.5} &~\text{when}~ n=20l+13;  \\
   &-\delta_b^{-11.5} &~\text{when}~ n=20l+18;  \\
   &-\delta_b^{-11.5} &~\text{when}~ n=20l+23.  \\
   \end{aligned}
   \right. \quad~\forall ~l\geq0.
\end{equation*}

\end{lemma}

\begin{proof}
When $7\leq k \leq n$,
we assume that

$\xi_{k1}=[a_{2k-1}a_{2n+2k-7}a_{2n+2k-13}\rangle;$

$\xi_{k2}=[a_{2k-1}a_{2n+1}a_{2n-1}a_{2n+2k-9}a_{2n+2k-13}\rangle;$

$\xi_{k3}=[a_{2k-1}a_{2n+1}a_{2n-1}a_{2n+1}a_{2n-1}a_{2n+2k-11}a_{2n+2k-13}\rangle;$

$\xi_{k4}=[a_{2k-1}a_{2n+3}a_{2n-1}a_{2n+2k-11}a_{2n+2k-13}\rangle;$

$\xi_{k5}=[a_{2k-1}a_{2n+1}a_{2n-3}a_{2n+2k-11}a_{2n+2k-13}\rangle;$

$\tilde{\xi}_{k1}=[a_{2k-1}a_{2k-7}a_{2n+2k-13} \rangle ;$

$\tilde{\xi}_{k2}=[a_{2k-1}a_{2k-5}a_{2n+1}a_{2n-1}a_{2n+2k-13} \rangle;$

$\tilde{\xi}_{k3}=[a_{2k-1}a_{2k-3}a_{2n+1}a_{2n-1}a_{2n+1}a_{2n-1}a_{2n+2k-13} \rangle;$

$\tilde{\xi}_{k4}=[a_{2k-1}a_{2k-3}a_{2n+3}a_{2n-1}a_{2n+2k-13} \rangle;$

$\tilde{\xi}_{k5}=[a_{2k-1}a_{2k-3}a_{2n+1}a_{2n-3}a_{2n+2k-13} \rangle.$

By Lemma(\ref{d22})(\ref{etak}), we may compute $T_k$, for $k\geq7$, where
$$T_k=\begin{bmatrix}
C_R(\xi_{k1}\tilde{\xi}_{k1}^*)&C_R(\xi_{k1}\tilde{\xi}_{k2}^*)&C_R(\xi_{k1}\tilde{\xi}_{k3}^*)&C_R(\xi_{k1}\tilde{\xi}_{k4}^*)&C_R(\xi_{k1}\tilde{\xi}_{k5}^*)\\
C_R(\xi_{k2}\tilde{\xi}_{k1}^*)&C_R(\xi_{k2}\tilde{\xi}_{k2}^*)&C_R(\xi_{k2}\tilde{\xi}_{k3}^*)&C_R(\xi_{k2}\tilde{\xi}_{k4}^*)&C_R(\xi_{k2}\tilde{\xi}_{k5}^*)\\
C_R(\xi_{k3}\tilde{\xi}_{k1}^*)&C_R(\xi_{k3}\tilde{\xi}_{k2}^*)&C_R(\xi_{k3}\tilde{\xi}_{k3}^*)&C_R(\xi_{k3}\tilde{\xi}_{k4}^*)&C_R(\xi_{k3}\tilde{\xi}_{k5}^*)\\
C_R(\xi_{k4}\tilde{\xi}_{k1}^*)&C_R(\xi_{k4}\tilde{\xi}_{k2}^*)&C_R(\xi_{k4}\tilde{\xi}_{k3}^*)&C_R(\xi_{k4}\tilde{\xi}_{k4}^*)&C_R(\xi_{k4}\tilde{\xi}_{k5}^*)\\
C_R(\xi_{k5}\tilde{\xi}_{k1}^*)&C_R(\xi_{k5}\tilde{\xi}_{k2}^*)&C_R(\xi_{k5}\tilde{\xi}_{k3}^*)&C_R(\xi_{k5}\tilde{\xi}_{k4}^*)&C_R(\xi_{k5}\tilde{\xi}_{k5}^*)\\
\end{bmatrix}.$$

For $1\leq i,j\leq 2$, we have
$$C_R(\xi_{ki}\tilde{\xi}_{kj}^*)=\delta_b^{-1}C_R(\eta_{(k-1)i}\tilde{\eta}_{(k-1)j}^*) \quad \forall k\geq7.$$

For $1\leq i\leq 5,3\leq j\leq 5$, we have
$$C_R(\xi_{ki}\tilde{\xi}_{kj}^*)=\delta_b^{-1}C_R(\eta_{(k-1)i}\tilde{\eta}_{(k-1)j}^*) \quad \forall k\geq7.$$

For $3\leq i\leq 5,j=2$, we have
$$C_R(\xi_{ki}\tilde{\xi}_{kj}^*)=\delta_b^{-1}C_R(\eta_{(k-2)i}\tilde{\eta}_{(k-2)j}^*) \quad \forall k\geq7.$$

For $3\leq i\leq 5,j=1$, we have
$$C_R(\xi_{ki}\tilde{\xi}_{kj}^*)=\delta_b^{-1}C_R(\eta_{(k-3)i}\tilde{\eta}_{(k-3)j}^*) \quad \forall k\geq8.$$

Based on the results listed in the Appendix, we have
$$C_R(\xi_{73}\tilde{\xi}_{71}^*)=\delta_b^{-9};   \quad
C_R(\xi_{74}\tilde{\xi}_{71}^*)=0; \quad
C_R(\xi_{75}\tilde{\xi}_{71}^*)=\delta_b^{-7.5}.$$
For example,
$$C_R(\xi_{73}\tilde{\xi}_{71}^*)
=C_R([a_{13}a_{2n+1}a_{2n-1}a_{2n+1}a_{2n-1}a_{2n+3}a_7a_{13}])
$$
$$=\delta_b^{-1.5}C_R([a_7a_{2n+1}a_{2n-1}a_{2n+1}a_{2n-1}a_{2n+3}])
=\delta_b^{-9}.
$$
The others are similar.

Then
$$T_k=\begin{bmatrix}
\delta_b^{-6}&\delta_b^{-7.5}&0&0&\delta_b^{-6.5}\\
\delta_b^{-7.5}&-\delta_b^{-10}&\delta_b^{-10.5}&0&0\\
\delta_b^{-9}&-\delta_b^{-11.5}&-\delta_b^{-13}&0&0\\
0&0&0&\delta_b^{-8}&0\\
\delta_b^{-7.5}&\delta_b^{-9}&0&0&-\delta_b^{-9}\\
\end{bmatrix}, ~\text{when} ~k=5l+7;$$
$$=\begin{bmatrix}
\delta_b^{-6}&\delta_b^{-7.5}&\delta_b^{-9}&\delta_b^{-7.5}&0\\
\delta_b^{-7.5}&-\delta_b^{-10}&-\delta_b^{-11.5}&\delta_b^{-9}&0\\
0&\delta_b^{-10.5}&-\delta_b^{-13}&0&0\\
\delta_b^{-6.5}&0&0&-\delta_b^{-9}&0\\
0&0&0&0&\delta_b^{-8}\\
\end{bmatrix},~\text{when} ~k=5l+8;$$
$$=\begin{bmatrix}
\delta_b^{-6}&\delta_b^{-7.5}&-\delta_b^{-10}&\delta_b^{-7.5}&\delta_b^{-7.5}\\
\delta_b^{-7.5}&-\delta_b^{-10}&\delta_b^{-12.5}&-\delta_b^{-10}&\delta_b^{-9}\\
0&-\delta_b^{-11.5}&\delta_b^{-14}&\delta_b^{10.5}&0\\
0&\delta_b^{-9}&\delta_b^{10.5}&0&0\\
\delta_b^{-6.5}&0&0&0&-\delta_b^{-9}\\
\end{bmatrix}, ~\text{when} ~k=5l+9;$$
$$=\begin{bmatrix}
\delta_b^{-6}&\delta_b^{-7.5}&\delta_b^{-9}&0&\delta_b^{-7.5}\\
\delta_b^{-7.5}&-\delta_b^{-10}&\delta_b^{-12.5}&\delta_b^{-9}&-\delta_b^{-10}\\
\delta_b^{-9}&\delta_b^{-12.5}&-\delta_b^{-15}&-\delta_b^{-11.5}&-\delta_b^{-11.5}\\
\delta_b^{-7.5}&-\delta_b^{-10}&-\delta_b^{-11.5}&0&\delta_b^{-9}\\
0&\delta_b^{-9}&-\delta_b^{-11.5}&\delta_b^{-9}&0\\
\end{bmatrix}, ~\text{when} ~k=5l+10;$$
$$=\begin{bmatrix}
\delta_b^{-6}&\delta_b^{-7.5}&0&\delta_b^{-6.5}&0\\
\delta_b^{-7.5}&-\delta_b^{-10}&-\delta_b^{-11.5}&0&\delta_b^{-9}\\
-\delta_b^{-10}&\delta_b^{-12.5}&\delta_b^{-14}&0&\delta_b^{-10.5}\\
\delta_b^{-7.5}&\delta_b^{-9}&0&-\delta_b^{-9}&0\\
\delta_b^{-7.5}&-\delta_b^{-10}&\delta_b^{10.5}&0&0\\
\end{bmatrix}, ~\text{when} ~k=5l+11.$$

Take $\xi_k$ to be $[a_{2k-1}a_{2n+1}a_{2n-1}a_{2n+1}a_{2n-1}a_{2n+1}a_{2n-1}a_{2n+2k-13} \rangle$, then
$$\begin{bmatrix}
C_R(\xi_{k1}\xi_k^*)\\
C_R(\xi_{k2}\xi_k^*)\\
C_R(\xi_{k3}\xi_k^*)\\
C_R(\xi_{k4}\xi_k^*)\\
C_R(\xi_{k5}\xi_k^*)\\
\end{bmatrix}
=\delta_b^2T_k
\begin{bmatrix}
\delta_b^2 C_R(\tilde{\xi}_{k1}\xi_k^*)\\
\delta_b^6 C_R(\tilde{\xi}_{k2}\xi_k^*)\\
\delta_b^{10} C_R(\tilde{\xi}_{k3}\xi_k^*)\\
\delta_b^6 C_R(\tilde{\xi}_{k4}\xi_k^*)\\
\delta_b^6 C_R(\tilde{\xi}_{k5}\xi_k^*)\\
\end{bmatrix}, ~\forall k\geq7.$$

Furthermore
$$C_R(\tilde{\xi}_{k1}\xi_k^*)=C_R({\xi}_{(k-3)1}\xi_{k-3}^*), ~\text{when} ~k\geq10;$$
$$C_R(\tilde{\xi}_{k2}\xi_k^*)=C_R({\xi}_{(k-2)2}\xi_{k-2}^*), ~\text{when} ~k\geq9;$$
$$C_R(\tilde{\xi}_{k3}\xi_k^*)=C_R({\xi}_{(k-1)3}\xi_{k-1}^*), ~\text{when} ~k\geq8;$$
$$C_R(\tilde{\xi}_{k4}\xi_k^*)=C_R({\xi}_{(k-1)4}\xi_{k-1}^*), ~\text{when} ~k\geq8;$$
$$C_R(\tilde{\xi}_{k5}\xi_k^*)=C_R({\xi}_{(k-1)5}\xi_{k-1}^*), ~\text{when} ~k\geq8.$$

So we may compute it inductively. Based on Lemma(\ref{d22})(\ref{etak}), by a direct computation, the initial condition is

$$\begin{bmatrix}
C_R(\tilde{\xi}_{71}\xi_7^*)&C_R(\tilde{\xi}_{81}\xi_8^*)&C_R(\tilde{\xi}_{91}\xi_9^*)\\
&C_R(\tilde{\xi}_{72}\xi_7^*)&C_R(\tilde{\xi}_{82}\xi_8^*)\\
&&C_R(\tilde{\xi}_{73}\xi_7^*)\\
&&C_R(\tilde{\xi}_{74}\xi_7^*)\\
&&C_R(\tilde{\xi}_{75}\xi_7^*)\\
\end{bmatrix}
=
\begin{bmatrix}
\delta_b^{-12.5}&0&-\delta_b^{-12.5}\\
&\delta_b^{-14}&\delta_b^{-13}\\
&&\delta_b^{-14.5}\\
&&-\delta_b^{-14}\\
&&-\delta_b^{-14}\\
\end{bmatrix}.$$
For example,
$$C_R(\tilde{\xi}_{91}\xi_9^*)
=C_R([a_{17}a_{11}a_{2n+5}a_{2n-1}a_{2n+1}a_{2n-1}a_{2n+1}a_{2n-1}a_{2n+1}a_{17}])
$$
$$
=\delta_b^{-0.5}C_R([a_{11}a_{2n+5}a_{2n-1}a_{2n+1}a_{2n-1}a_{2n+1}a_{2n-1}a_{2n+1}a_{11}])
$$
$$
=\delta_b^{-0.5}(\delta_b^2C_R([a_{11}a_{2n+5} a_5a_{11}])
C_R([a_{11}a_5 a_{2n+1}a_{2n-1}a_{2n+1}a_{2n-1}a_{2n+1}a_{11}])
$$
$$
+\delta_b^6C_R([a_{11}a_{2n+5}a_{2n-3}a_{2n+1}a_9a_{11}])
C_R([a_{11}a_9a_{2n+1}a_{2n-3}a_{2n+1}a_{2n-1}a_{2n+1}a_{2n-1}a_{2n+1}a_{11}])
$$
$$
-\delta_b\delta_b^4C_R([a_{11}a_{2n+5}a_{2n-1}a_{2n+1}a_7a_{11}])
C_R([a_{11}a_7a_{2n+1}a_{2n-1}a_{2n+1}a_{2n-1}a_{2n+1}a_{2n-1}a_{2n+1}a_{11}])
$$
$$
-\delta_b\delta_b^4C_R([a_{11}a_{2n+5}a_{2n-1}a_{2n+3}a_9a_{11}])
C_R([a_{11}a_9a_{2n+3}a_{2n-1}a_{2n+1}a_{2n-1}a_{2n+1}a_{2n-1}a_{2n+1}a_{11}])
$$
$$
-\delta_b\delta_b^8C_R([a_{11}a_{2n+5}a_{2n-1}a_{2n+1}a_{2n-1}a_{2n+1}a_9a_{11}])
$$
$$
C_R([a_{11}a_9a_{2n+1}a_{2n-1}a_{2n+1}a_{2n-1}a_{2n+1}a_{2n-1}a_{2n+1}a_{2n-1}a_{2n+1}a_{11}])
)
$$
$$
=\delta_b^{-0.5}(\delta_b^2\delta_b^{-2.5}\delta_b^{-3}\delta_b^{-1.5}(-\delta_b^{-8})
+\delta_b^6\delta_b^{-0.5}C_R(\eta_{51}\eta_{55}^*)
\delta_b^{-0.5}C_R(\eta_{55}\eta_{53}^*)
$$
$$
-\delta_b\delta_b^4\delta_b^{-1.5}\delta_b^{-7.5}
\delta_b^{-1}\delta_b^{-11}
-\delta_b\delta_b^4\delta_b^{-0.5}C_R(\eta_{51}\eta_{54}^*)
\delta_b^{-0.5}C_R(\eta_{54}\eta_{53}^*)
$$
$$
-\delta_b\delta_b^8\delta_b^{-0.5}C_R(\eta_{51}\eta_{53}^*)
\delta_b^{-0.5}C_R(\eta_{53}\eta_{53}^*)
)
$$
$$
=\delta_b^{-0.5}(\delta_b^2\delta_b^{-2.5}\delta_b^{-3}\delta_b^{-1.5}(-\delta_b^{-8})
+0
-\delta_b\delta_b^4\delta_b^{-1.5}\delta_b^{-5.5}
\delta_b^{-1}\delta_b^{-11}
+0
+0
)
=-\delta_b^{-12.5}.
$$
The others are similar.

Then we have

\begin{center}
\begin{tabular}{|c|c|c|c|c|c|c|c|}

\hline

$k$&$7$&$8$&$9$&$10$&$11$&$12$&$13$\\

\hline

$C_R(\xi_{k1}\xi_k^*)$&$0$&$-\delta_b^{-13.5}$&$\delta_b^{-12.5}$&$-\delta_b^{-12.5}$&$-\delta_b^{-12.5}$&$\delta_b^{-12.5}$&$-\delta_b^{-13.5}$\\
\hline
$C_R(\xi_{k2}\xi_k^*)$&$\delta_b^{-13}$ &$-\delta_b^{-15}$&$\delta_b^{-16}$&$\delta_b^{-12}$ &$\delta_b^{-15}$&0&$\delta_b^{-16}$\\
\hline
$C_R(\xi_{k3}\xi_k^*)$&$-\delta_b^{-15.5}$ &$\delta_b^{-14.5}$ &$\delta_b^{-17.5}$ &$-\delta_b^{-16.5}$ &$\delta_b^{-16.5}$&$-\delta_b^{-15.5}$&$\delta_b^{-15.5}$\\
\hline
$C_R(\xi_{k4}\xi_k^*)$&$-\delta_b^{-14}$&$\delta_b^{-15}$&$\delta_b^{-12}$&$-\delta_b^{-14}$&$\delta_b^{-15}$&$\delta_b^{-15}$&$-\delta_b^{-14}$\\
\hline
$C_R(\xi_{k5}\xi_k^*)$&$\delta_b^{-13}$&$\delta_b^{-13}$&$-\delta_b^{-13}$&$\delta_b^{-14}$&$-\delta_b^{-14}$&$\delta_b^{-12}$&$\delta_b^{-12}$\\
\hline

\hline

\end{tabular}
\end{center}

\begin{center}
\begin{tabular}{|c|c|c|c|c|c|c|c|c|}

\hline

$14$&$15$&$16$&$17$&$18$&$19$&$20$&$21$\\

\hline

$0$&$-\delta_b^{-12.5}$&$0$&$\delta_b^{-12.5}$&$-\delta_b^{-11.5}$&$0$&$0$&$0$\\
\hline
$\delta_b^{-13}+\delta_b^{-16}$&$\delta_b^{-13}$&$-\delta_b^{-14}$&$\delta_b^{-14}$&$\delta_b^{-13}+\delta_b^{-16}$&$\delta_b^{-14}$&$0$&$0$\\
\hline
$-\delta_b^{-18.5}$&$\delta_b^{-15.5}$&$\delta_b^{-14.5}$&$-\delta_b^{-14.5}$&$\delta_b^{-17.5}$&$-\delta_b^{-15.5}-\delta_b^{-18.5}$&$\delta_b^{-15.5}$&$\delta_b^{-13.5}$\\
\hline
$\delta_b^{-14}$&$-\delta_b^{-14}$&$\delta_b^{-13}$&$\delta_b^{-13}$&$-\delta_b^{-13}$&$\delta_b^{-14}$&$0$&$0$\\
\hline
$-\delta_b^{-13}-\delta_b^{-15}$&$\delta_b^{-14}$&$0$&$\delta_b^{-14}$&$\delta_b^{-14}$&$-\delta_b^{-15}$&$\delta_b^{-12}$&$0$\\
\hline

\hline

\end{tabular}
\end{center}

\begin{center}
\begin{tabular}{|c|c|c|c|c|c|c|c|c|}

\hline

$22$&$23$&$24$&$25$&$26$&$27$&$28$&$29$\\

\hline

$0$&$-\delta_b^{-11.5}$&$\delta_b^{-12.5}$&$0$&$-\delta_b^{-12.5}$&$0$&$-\delta_b^{-13.5}$&$\delta_b^{-12.5}$\\
\hline
$\delta_b^{-12}$&$\delta_b^{-14}$&$-\delta_b^{-15}$&$\delta_b^{-14}$&$\delta_b^{-13}$&$\delta_b^{-13}$&$-\delta_b^{-15}$&$\delta_b^{-16}$\\
\hline
$-\delta_b^{-14.5}$&$\delta_b^{-15.5}$&$-\delta_b^{-16.5}$&$-\delta_b^{-16.5}$&$\delta_b^{-14.5}$&$-\delta_b^{-15.5}$&$\delta_b^{-14.5}$&$\delta_b^{-17.5}$\\
\hline
$0$&$0$&$\delta_b^{-12}$&$0$&$-\delta_b^{-14}$&$-\delta_b^{-14}$&$\delta_b^{-15}$&$\delta_b^{-12}$\\
\hline
$0$&$0$&$0$&$\delta_b^{-12}$&$-\delta_b^{-14}$&$\delta_b^{-13}$&$\delta_b^{-13}$&$-\delta_b^{-13}$\\
\hline

\hline

\end{tabular}
.
\end{center}

Note that the periodicity is $20$. So
$$C_R(a_{2n-1}a_{4n-7}a_{2n-1}a_{2n+1}a_{2n-1}a_{2n+1}a_{2n-1}a_{2n+1}a_{2n-1})=C_R(\xi_{n1}\xi_n^*)$$
\begin{equation*}
  =\left\{
   \begin{aligned}
   &-\delta_b^{-13.5} &~when~ n=20l+8;  \\
   &-\delta_b^{-13.5} &~when~ n=20l+13;  \\
   &-\delta_b^{-11.5} &~when~ n=20l+18;  \\
   &-\delta_b^{-11.5} &~when~ n=20l+23.  \\
   \end{aligned}
   \right.
\end{equation*}

\end{proof}

\begin{theorem}
When $n\geq4$, the $n_{th}$ Bisch-Haagerup fish graph is not the principal graph of a subfactor.
\end{theorem}

\begin{proof}
By Lemma \ref{independence of initial condition},
to compute the coefficients $C_R$ of loops in $A_k^+\mathscr{G}_{2n,+}A_k^+$,
we may fix the initial condition as $\mu_1=\mu_2=\omega=1$.

When $n=4$, from the Appendix, we have
$C_R([a_5a_{9}a_5a_{9}a_{5}])=\delta_b^{-5}$ and
$C_R([a_7a_{11}a_7a_{11}a_{7}])=0$.
By the symmetry of the dual principal graph, we may substitute $2k-1$ by $4n-2k+1$.
Then $C_R([a_9a_{5}a_9a_{5}a_{9}])=0$.
By Lemma \ref{independence of initial condition}, these coefficients are independent of the parameters $\mu_1,\mu_2,\omega$.
If $R_{\mu_1\mu_2\omega}$ is a solution of Proposition \ref{relation of R},
then $$\frac{\lambda'(a_5)}{\lambda'(a_9)}C_R([a_5a_{9}a_5a_{9}a_{5}])=\omega^2C_R([a_9a_{5}a_9a_{5}a_{9}]).$$
So $$|\delta_b^{-1}C_R([a_5a_{9}a_5a_{9}a_{5}])|=|C_R([a_9a_{5}a_9a_{5}a_{9}])|.$$
It is a contradiction. That means the $4_{th}$ Bisch-Haagerup fish graph is not the principal graph of a subfactor.

By the symmetry of the dual principal graph, we may substitute $2k-1$ by $4n-2k+1$.
Then $C_R([a_9a_{5}a_9a_{5}a_{9}])=0$.
So $C_R([a_5a_{9}a_5a_{9}a_{5}])=0$, by proposition(4'). It is a contradiction.
That means the $4_{th}$ Bisch-Haagerup fish graph is not the principal graph of a subfactor.

When $n\geq5$, by Lemma \ref{d8}, we have
$C_R([a_5a_{2n+1}a_{2n-1}a_{2n+1}a_{2n-1}a_{2n+1}a_{5}])=-\delta_b^{-8}$.
By the symmetry of the dual principal graph, we have  $C_R([a_{4n-5}a_{2n-1}a_{2n+1}a_{2n-1}a_{2n+1}a_{2n-1}a_{4n-5}])=-\delta_b^{-8}$.
If $R_{\mu_1\mu_2\omega}$ is a solution of Proposition \ref{relation of R}, then by Lemma \ref{independence of initial condition}, we have
$$|C_R([a_{2n-1}a_{4n-5}a_{2n-1}a_{2n+1}a_{2n-1}a_{2n+1}a_{2n-1}])|$$
$$=|\delta_b^{-1}C_R([a_{4n-5}a_{2n-1}a_{2n+1}a_{2n-1}a_{2n+1}a_{2n-1}a_{4n-5}])|=\delta_b^{-9}.$$
On the other hand, by Lemma \ref{etak},
$$|C_R(\eta_{n1}\eta_{n3}^*)|=|C_R([a_{2n-1}a_{4n-5}a_{2n-1}a_{2n+1}a_{2n-1}a_{2n+1}a_{2n-1}])|=\delta_b^{-9}$$ implies $5|n-3$.

When $n\geq8$ and $5|n-3$, from the Appendix, we have
$$C_R([a_7a_{2n+1}a_{2n-1}a_{2n+1}a_{2n-1}a_{2n+1}a_{2n-1}a_{2n+1}a_{7}])=\delta_b^{-11}.$$
By the symmetry of the dual principal graph, we have  $$C_R([a_{4n-7}a_{2n-1}a_{2n+1}a_{2n-1}a_{2n+1}a_{2n-1}a_{2n+1}a_{2n-1}a_{4n-7}])=\delta_b^{-11}.$$
So $$|C_R([a_{2n-1}a_{4n-7}a_{2n-1}a_{2n+1}a_{2n-1}a_{2n+1}a_{2n-1}a_{2n+1}a_{2n-1}])|=\delta_b^{-12.5}.$$
On the other hand, by Lemma \ref{dfinal}, we have
$$|C_R([a_{2n-1}a_{4n-7}a_{2n-1}a_{2n+1}a_{2n-1}a_{2n+1}a_{2n-1}a_{2n+1}a_{2n-1}])|$$ is $\delta_b^{-11.5}$ or $\delta_b^{-13.5}$. It is a contradiction.

Therefore the $n_{th}$ Bisch-Haagerup fish graph is not the principal graph of a subfactor whenever $n\geq4$.
\end{proof}

\subsection{Uniqueness}

\begin{theorem}\label{uniqueness}
There is only one subfactor planar algebra whose principal graph is the $n_{th}$ Bisch-Haagerup fish graph, for $n=1,2,3$.
\end{theorem}

It is easy to generalize the Jellyfish technic \cite{BMPS} for Fuss-Catalan tangles, or tangles labeled by the biprojection. We are going to check the Fuss-Catalan Jellyfish relations for the generators $S$ and $R$.
Before that let us prove two Lemmas which tell the Fuss-Catalan Jellyfish relations.

\begin{lemma}\label{R jellyfish}
If $R$ is a solution of Proposition \ref{relation of R} in a subfactor planar algebra with a biprojection, then
$$P=\delta^2Pe_{2n}P,$$
where $P=\delta_b^{-1}g_{2n}-R$.
\end{lemma}

\begin{proof}
Note that $P=\delta_b^{-1}g_{2n}-R$ is a projection. It is easy to check that
$\delta^2Pe_{2n}P$ is a subprojection of $P$. Moreover they have the same trace. So
$P=\delta^2Pe_{2n}P.$
\end{proof}

\begin{remark}
This is Wenzl's formula \cite{Wen87} \cite{Liuex} for the minimal projection $P$.
\end{remark}

\begin{lemma}\label{S jellyfish}
If $S$ is a solution of Proposition \ref{relation of S} in a subfactor planar algebra with a biprojection, then
$$Q=\delta\delta_a Qp_{2n}Q,$$
where $Q=\frac{1}{2}(f_{2n}+S)$.
\end{lemma}

\begin{proof}
Note that $Q=\frac{1}{2}(f_{2n}+S)$ is a projection. It is easy to check that
$\delta\delta_a Qp_{2n}Q$ is a subprojection of $Q$. Moreover they have the same trace. So
$Q=\delta\delta_a Qp_{2n}Q.$
\end{proof}

\begin{proof}[Proof of Theorem \ref{uniqueness}]
We have known three examples whose principal graphs are the first three Bisch-Haagerup fish graphs. We only need to prove the uniqueness.

For $n=1,2,3$, suppose $R_{\mu_1\mu_2\omega}$ is a solution of Proposition \ref{relation of R}.
Note that the loop
$$[\underbrace{a_{2n-1}a_{2n+1}\cdots a_{2n-1}a_{2n+1}}_n]$$
is rotation invariant. Moreover its coefficient in $R$ is non-zero. So $\omega=1$.

If $(S,R,\omega_0)$ is a solution of Proposition(\ref{relation of S})(\ref{relation of R}), then $(-S,R,-\omega_0)$ is also a solution. Up to this isomorphism, we may assume $\omega_0=1$.

Suppose $\mathscr{B}$ is a subfactor planar algebra whose principal graph is the $n_{th}$ Bisch-Haagerup fish graph, and its generators $R,S$ satisfy Proposition(\ref{relation of S})(\ref{relation of R}), such that $\omega_0=1$.
Let us consider the linear subspaces $V_\pm$ of $\mathscr{B}_{2n+1,\pm}$ generated by annular Fuss-Catalan tangles acting on $R$.
We claim that the space $V_\pm$ satisfies Fuss-Catalan Jellyfish relations. Therefore the subfactor planar algebra is unique.

Obviously $V_\pm$ is * closed and rotation invariant.
The multiplication on $V_\pm$ is implied by the Lemma \ref{R jellyfish}.
Now let us check the Fuss-Catalan Jellyfish relations.

When we add one string in an unshaded region,
for example, we add one string on the right of $\tilde{R}$, where $\tilde{R}\in V_-$ is the diagram adding one string on the right of $R$.
Then by Lemma \ref{R jellyfish}, we have $\delta_b^{-1}g_{2n}-R\in \mathscr{I}_{2n+2,-}$, where $\mathscr{I}_{2n+2,-}$ is the two sided ideal of $\mathscr{B}_{2n,-}$ generated by the Jones projection. That implies the Jellyfish relation of $\tilde{R}$ while adding one string on the right. Other Jellyfish relations are similar.

When we add one string in a shaded region,
for example, we add one string on the right of $\tilde{S}$, where $\tilde{S}\in V_+$ is the diagram adding one string on the right of $S$.
Then by Lemma \ref{S jellyfish}, and the fact that $p_{2n}\in \mathscr{I}_{2n+2,+}$, where $\mathscr{I}_{2n+2,+}$ is the two sided ideal of $\mathscr{B}_{2n,+}$ generated by the Jones projection,
we have $\frac{1}{2}(f_{2n}+S)\in \mathscr{I}_{2n+2,+}$.  That implies the Jellyfish relation of $\tilde{S}$ while adding one string on the right. Other Jellyfish relations are similar.
\end{proof}

It is easy to check that the possible solution $(R,S)$, for $\mu_1=\mu_2=\pm1,\omega_0=1$, in the graph planar algebra does satisfy Proposition(\ref{relation of S})(\ref{relation of R}).
The skein theoretic construction of the three subfactor planar algebras corresponding to the first three Bisch-Haagerup fish graphs could be realized by the Fuss-Catalan Jellyfish relations of the generating vector space $V_\pm$ mentioned above.
We leave the details to the reader.

\section{Composed inclusions of two $A_4$ subfactors}
In this section, we will consider composed inclusions $\mathcal{N}\subset\mathcal{P}\subset\mathcal{M}$ of two $A_4$ subfactors.
Let $id$ be the trivial $(\mathcal{P},\mathcal{P})$ bimodule, and $\rho_1,\rho_2$ be the non trivial $(\mathcal{P},\mathcal{P})$ bimodules arise from $\mathcal{N}\subset\mathcal{P}$, $\mathcal{P}\subset\mathcal{M}$ respectively.
Then $\rho_i^2=\rho_i\oplus id$, for $i=1,2$. If it is a free composed inclusion, i.e., there is no relation between $\rho_1$ and $\rho_2$, then its planar algebra is $FC(\delta_b,\delta_b)$; Otherwise
take $w$ to be a shortest word of $\rho_1,\rho_2$ which contains $id$.
If $w=(\rho_1\rho_2)^n\rho_1$, and $n$ is even, then by Frobenius reciprocity, we have
$$\dim(\hom((\rho_1\rho_2)^\frac{n}{2}\rho_1,(\rho_1\rho_2)^\frac{n}{2}))=c\geq 1.$$
So
$$\dim(\hom((\rho_1\rho_2)^\frac{n}{2}\rho_1^2,(\rho_1\rho_2)^\frac{n}{2}))=\dim(\hom((\rho_1\rho_2)^\frac{n}{2}\rho_1,(\rho_1\rho_2)^\frac{n}{2}\rho_1))\geq c+1.$$
Note that $\rho_1^2=\rho\oplus id$, we have
$$\dim(\hom((\rho_1\rho_2)^\frac{n}{2},(\rho_1\rho_2)^\frac{n}{2}))\geq 1.$$
So $(\rho_1\rho_2)^n$ contains $id$, which contradicts to the assumption that $w$ is shortest. It is similar for the other cases. Without loss of generality, we have $w=(\rho_1\rho_2)^n$, for some $n\geq 1$.

Considering the planar algebra $\mathscr{B}$ of $\mathcal{N}\subset\mathcal{M}$ as an annular Fuss-Catalan module, then it contains a lowest weight vector $T\in\mathscr{B}_{n,+}$ which induces a morphism from $(\rho_1\rho_2)^n$ to $id$. So $T$ is totally uncappable.

\begin{remark}
There is another proof without using bimodules.
The lowest weight vector $T\in\mathscr{B}_{n,+}$ is totally uncappable, for $n\geq 2$, see \cite{LiuAFC}.
For the case $n=1$, to show it is totally uncappable, we need the fact that the biprojection cutdown induces an planar algebra isomorphism \cite{BisJonfree}.
\end{remark}

\begin{definition}
Let us define $\Omega_n$, for $n\geq 1$, to be the ($\mathcal{N},\mathcal{P},\mathcal{M}$) coloured graph with parameter $(\delta_b,\delta_b)$ as
$$\grb{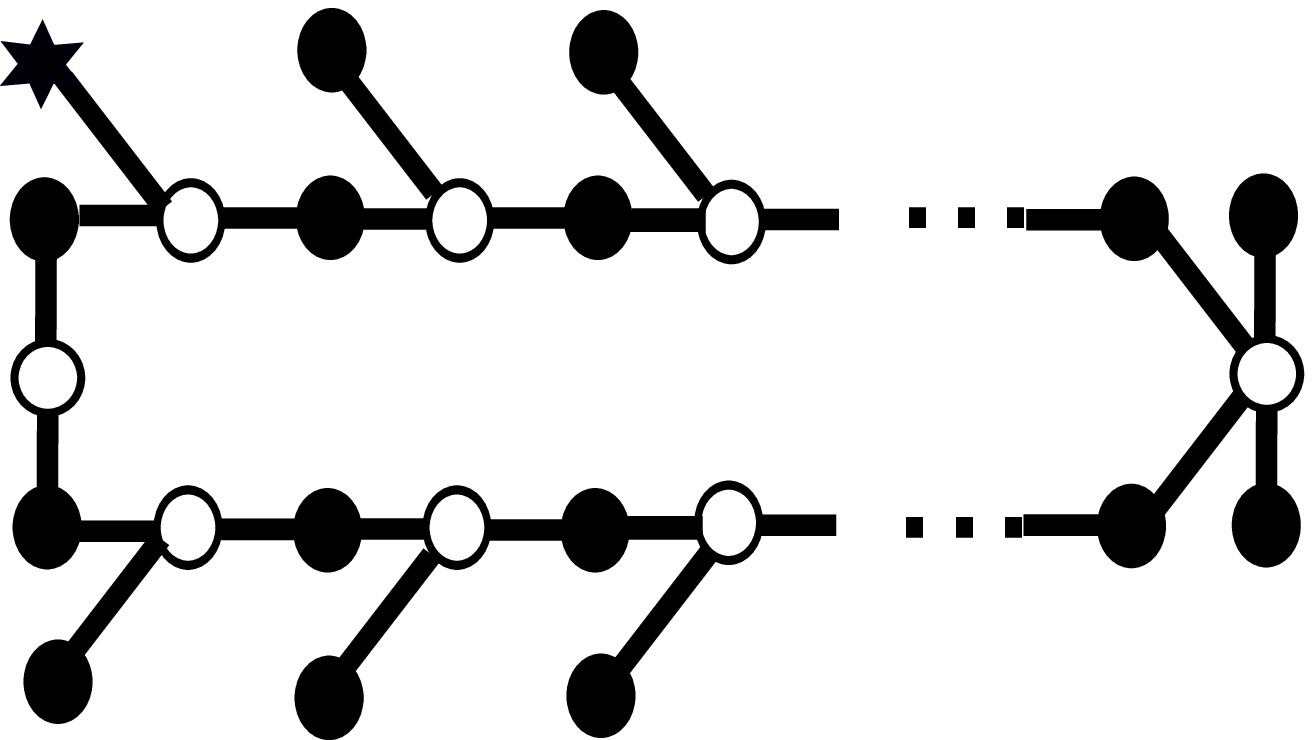},$$
where the black vertices are $\mathcal{N},\mathcal{M}$ coloured, and the white vertices are $\mathcal{P}$ coloured, and the number of white vertices is $2n$.
\end{definition}

\begin{lemma}
Suppose $\mathscr{B}$ is a composition of two $A_4$ Temperley-Lieb planar algebras.
Then either $\mathscr{B}$ is Fuss-Catalan, or its refined principal graph is $\Omega_n$, for some $n\geq 1$.
\end{lemma}

\begin{proof}
If $\mathscr{B}$ is not Fuss-Catalan, then it contains a lowest weight vector $T\in\mathscr{B}_{n,+}$ which is totally uncappable, for some $n\geq1$.
So the refined principal graph of $\mathscr{B}$ is the same as that of $FC(\delta_b,\delta_b)$, until the vertex corresponding to $f_n$ splits, where $f_n$ the minimal projection of $FC(\delta_b,\delta_b)_{n,+}$ with middle pattern $abba\cdots abba(ab)$.

By the embedding theorem, $T$ is embedded in the graph planar algebra.
Similar to the proof of Lemma \ref{key lemma},
the loop passing the vertex, corresponding to the middle pattern $aaa$, has non-zero coefficient in $S$. Similar to the proof of Lemma \ref{ini U}, it has to be a length 2n flat loop, a loop whose vertices are all flat.
Via computing the trace, there is a unique way to complete the refined principal graph as
$$\grc{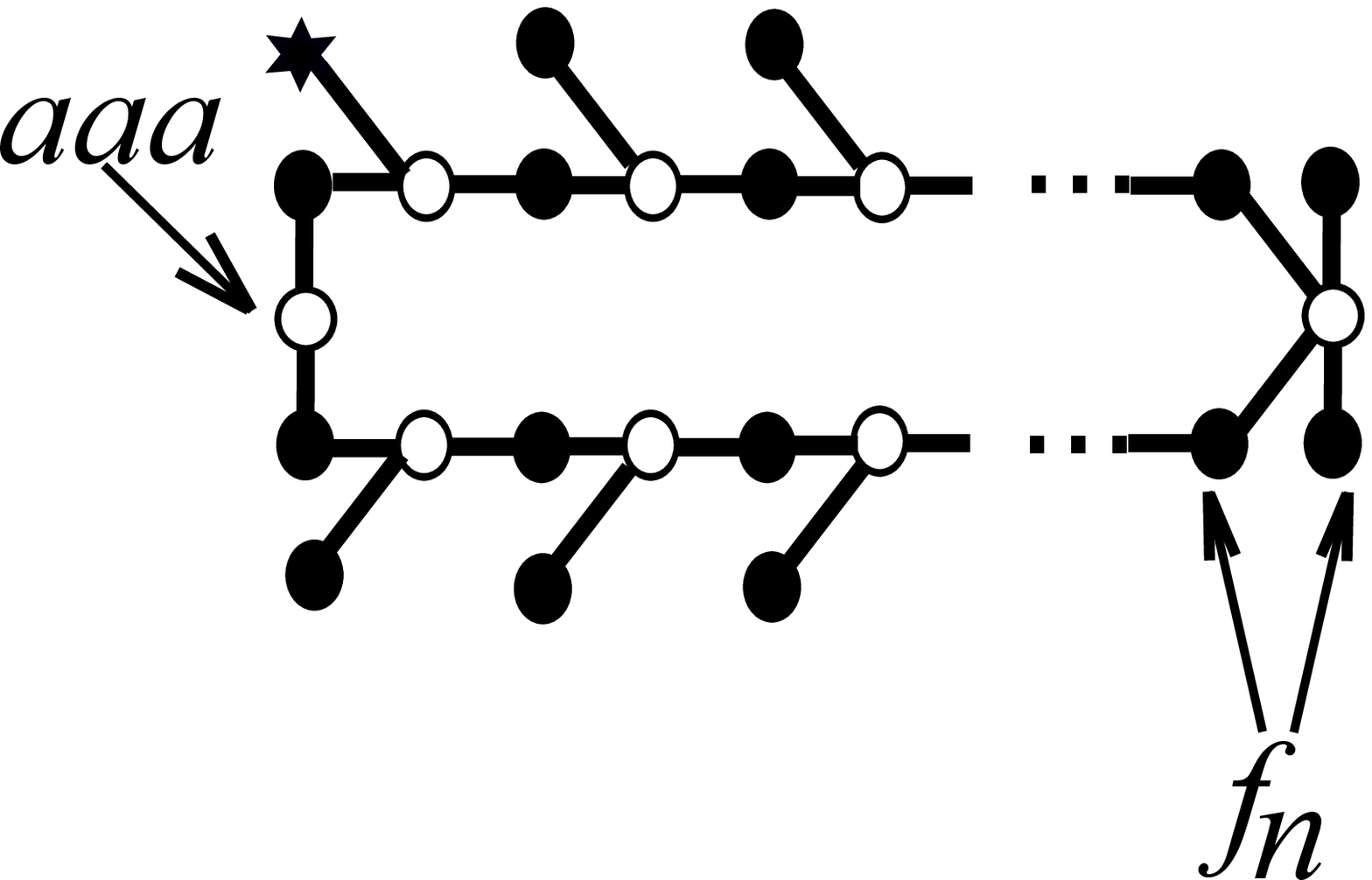}.$$
\end{proof}

For $n=1,2,3$, it is easy to check that $\Omega_n$ is the refined principal graph of the reduced subfactor from the vertex $a_3$, corresponding to the middle pattern $baab$, in the (refined) dual principal graph of the $n_{th}$ fish factor.

Comparing this refine principal graph with the one obtained in Lemma \ref{refined dual principal graph}, they share the same black and white vertice and the same dimension vector on these vertices.
Similar to Proposition \ref{relation of R}, we have the following result.
\begin{proposition}\label{relation of R1}
Suppose $\mathscr{B}$ is a planar algebra as a composition of two $A_4$ planar algebras, and it is not Fuss-Catalan.
Then there is a lowest weight vector $T\in\mathscr{B}_{n,+}$, such that

(1) $T*=T$;

(2) $T+\delta_b^{-2}f_{n}$ is a projection;

(3) $T$ is totally uncappable;

(4) $\rho(T)=\omega T$,\\
where $f_n$ is the minimal projection of $FC(\delta_b,\delta_b)_{n,+}$ with middle pattern $abba\cdots abba (ab)$.
\end{proposition}

Note that the dual of $\mathscr{B}$ is still a composition of two $A_4$ planar algebras. So the refined dual principal graph is the same as $\Omega_n$.
Then there is a lowest weight vector $T'\in\mathscr{B}_{n,-}$ satisfying similar propositions.
Solving this generators $T,T'$ in the graph planar algebra is the same as solving $R$ for the compositions of $A_3$ with $A_4$, while the rotation is replaced by the Fourier transform.
Therefore we have the following result.

\begin{theorem}
There are exactly four subfactor planar algebras as a composition of two $A_4$ planar algebras.
\end{theorem}

\begin{proof}
Suppose $\mathscr{B}$ is a planar algebra as a composition of two $A_4$ planar algebras.
If $\mathscr{B}$ is not Fuss-Catalan, then there is a lowest weight vector $T\in\mathscr{B}_{n,+}$ satisfying proposition (1)(2)(3)(4), and $T'\in\mathscr{B}_{n,+}$ satisfying similar propositions.
Comparing with the process of solving $R$ in the graph planar algebra for the composition of $A_3$ and $A_4$, we have the $\Omega_n$, for $n\geq4$, is not the refined principal graph of a subfactor.

For $n=1,2,3$, three examples are known as reduced subfactors. We only need to prove the uniqueness.
Similar to the proof of Theorem \ref{uniqueness}, by comparing the coefficient of the rotation invariant loop, we have $T=\mathcal{F}(T')=\rho(T)$. So $\omega=1$.
Furthermore the linear subspaces $V_\pm$ of $\mathscr{B}_{n+1,\pm}$ generated by annular Fuss-Catalan tangles acting on $T$ satisfy Fuss-Catalan Jellyfish relations, which are derived from Wenzl's formula similar to Lemma \ref{S jellyfish} and Theorem \ref{uniqueness}. Therefore the subfactor planar algebra is unique.
\end{proof}

Similarly we may construct the generators $(T,T')$ in the graph planar algebra.
The skein theoretic construction of the three subfactor planar algebras could be realized by the Fuss-Catalan Jellyfish relations of the generating vector space $V_\pm$.

\appendix
\section{the initial conditions}
Up to the rotation, we only need $C_R(l)$ for a loop $l\in A_k^+\mathscr{G}_{2n,+}A_k^+$.
Now we list of results up to adjoint for $1\leq k\leq 4$. They are obtained by a direct computation by Lemma(\ref{coefficient R P})(\ref{coefficient P U}).

When $n\geq1$,

$C_R([a_1a_{2n+1}])=\delta_b^{-3}.$

When $n\geq2$,

$C_R([a_3a_{2n+3}])=\delta_b^{-3};$

$C_R([a_3a_{2n+1}a_{2n-1}a_{2n+1}])=\delta_b^{-5}.$

When $n\geq3$,

$C_R([a_5a_{2n+5}])=\delta_b^{-3};$

$C_R([a_5a_{2n+1}a_{2n-3}a_{2n+1}])=\delta_b^{-5};$

$C_R([a_5a_{2n+1}a_{2n-1}a_{2n+1} a_{2n-1}a_{2n+1}])=-\delta_b^{-8};$

$C_R([a_5a_{2n+1}a_{2n-1}a_{2n+3}])=\delta_b^{-5.5}.$

When $n\geq4$,

$C_R([a_7a_{2n+7}])=\delta_b^{-3};$

$C_R([a_7a_{2n+1}a_{2n-5}a_{2n+1}])=\delta_b^{-5};$

$C_R([a_7a_{2n+3}a_{2n-1}a_{2n+3}])=0;$

$C_R([a_7a_{2n+3}a_{2n-1}a_{2n+1}a_{2n-1}a_{2n+1}])=\delta_b^{-7.5};$

$C_R([a_7a_{2n+1}a_{2n-1}a_{2n+1}a_{2n-1}a_{2n+1}a_{2n-1}a_{2n+1}])=\delta_b^{-11};$

$C_R([a_7a_{2n+1}a_{2n-3}a_{2n+1}a_{2n-1}a_{2n+1}])=-\delta_b^{-8.5};$

$C_R([a_7a_{2n+1}a_{2n-3}a_{2n+3}])=\delta_b^{-6};$

$C_R([a_7a_{2n+5}a_{2n-1}a_{2n+1}])=\delta_b^{-5.5};$

$C_R([a_7a_{2n+1}a_{2n-1}a_{2n+3}a_{2n-1}a_{2n+1}])=-\delta_b^{-8}.$




\bibliography{bibliography}
\bibliographystyle{amsalpha}

\end{document}